\documentclass[11pt]{amsart}

\usepackage{upgreek}
\usepackage{mathpazo}
\usepackage{ytableau}
\usepackage{adjustbox}
\usepackage{longtable}

\usepackage{amsmath, amsxtra, amsthm, amsfonts,amssymb,mathtools,bm}
\usepackage{mathrsfs}
\usepackage[normalem]{ulem}
\usepackage[mathscr]{euscript}
\usepackage{graphicx}
\usepackage{url}
\usepackage{color}
\usepackage{bbm}
\usepackage{tikz-cd}
\usepackage{tikz}

\usepackage[T1]{fontenc}
\graphicspath{Figures/}

\usepackage{float}
\floatstyle{plaintop}
\restylefloat{table}

\usepackage{enumitem}
\usepackage{comment}
\usepackage[pdftex,hidelinks,backref=page]{hyperref}
\hypersetup{
    colorlinks,
    citecolor=magenta,
    filecolor=magenta,
    linkcolor=blue,
    urlcolor=black
}
\usepackage{cleveref}

\renewcommand{\emptyset}{\varnothing}
\newcommand{\NN}{\mathbb N}
\newcommand{\QQ}{\mathbb Q}
\newcommand{\RR}{\mathbb R}

\newcommand{\ZZ}{\mathbb Z}

\theoremstyle{definition}
\newtheorem{thm}{Theorem}[section]
\newtheorem{cor}[thm]{Corollary}
\newtheorem{lem}[thm]{Lemma}
\newtheorem{clm}[thm]{Claim}
\newtheorem{prop}[thm]{Proposition}
\newtheorem{defn}[thm]{Definition}

\newtheorem{eg}[thm]{Example}
\newtheorem{rem}[thm]{Remark}
\newtheorem{obs}[thm]{Observation}

\newtheorem{question}[thm]{Question}
	\newtheorem{fact}[thm]{Fact}

\numberwithin{equation}{section}

\newcommand{\indexedforests}[1][]{ 
{\ifx&#1&%
    \operatorname{\mathsf{For}}
\else
    \operatorname{\mathsf{For}}^{#1}
\fi}
}
\newcommand{\forestpoly}[2][]{{\mathfrak{P}_{#2}^{\underline{#1}}}} 

\newcommand{\oneforestpoly}[1]{\mathfrak{P}_{#1}} 

\newcommand{\qsym}[2][]{
{\ifx&#1&%
  {\operatorname{QSym}_{#2}}
\else
  {{}^{#1}\!\operatorname{QSym}_{#2}}
\fi}
} 

\newcommand{\qseq}[2][]{
{\ifx&#1&%
  {\operatorname{QSeq}_{#2}}
\else
  {{}^{#1}\!\operatorname{QSeq}_{#2}}
\fi}
}

\newcommand{\qsymide}[2][]{
{\ifx&#1&%
  {\operatorname{QSym}_{#2}^+}
\else
  {{}^{#1}\!\operatorname{QSym}_{#2}^+}
\fi}
} 
\newcommand{\sym}[1]{\operatorname{Sym}_{#1}} 
\newcommand{\symide}[1]{\sym{#1}^+} 
\newcommand{\lcode}[1]{\operatorname{lcode}(#1)} 
\newcommand{\supp}{\operatorname{supp}} 
\newcommand{\compatible}[2][]{
{\ifx&#1&%
  {\mathcal{C}(#2)}
\else
  {\mathcal{C}^{m}(#2)}
\fi}
} 
\newcommand{\internal}[1]{\operatorname{IN}(#1)} 
\newcommand{\suchthat}{\;|\;}
\newcommand{\im}[1]{\operatorname{Im}(#1)} 
\newcommand{\nvect}{\mathsf{Codes}}

\newcommand{\grass}[1]{\mathsf{Grass_{#1}}} 

\newcommand{\poly}{\operatorname{Pol}} 
\newcommand{\schub}[1]{\mathfrak{S}_{#1}} 
\newcommand{\red}[1]{\operatorname{Red}(#1)} 
\newcommand{\des}[1]{\operatorname{Des}(#1)} 
\date{}
\newcommand{\cat}[1]{\operatorname{Cat}_{#1}} 

\newcommand{\idem}{\operatorname{id}} 
\newcommand{\slide}[2][]{
{\ifx&#1&%
  {\mathfrak{F}_{#2}}
\else
  {\mathfrak{F}_{#2}^{\underline{#1}}}
\fi}
} 
\newcommand{\ct}{\operatorname{ev_0}} 
\newcommand{\qdes}[2][]{\operatorname{LTer}_{#1}(#2)} 
\newcommand{\tope}[2][]{
{\ifx&#1&%
  {\mathsf{T}_{#2}}
\else
  {\mathsf{T}_{#2}^{\underline{#1}}}
\fi}
} 
\newcommand{\bope}[2][]{
{\ifx&#1&%
  {\mathsf{B}_{#2}}
\else
  {\mathsf{B}_{#2}^{(#1)}}
\fi}
} 
\newcommand{\rope}[1]{\mathsf{R}_{#1}} 
\newcommand{\End}{\operatorname{End}} 
\newcommand{\Hom}{\operatorname{Hom}} 
\newcommand{\Trim}[1]{\operatorname{Trim}({#1})} 
\newcommand{\Th}[1][]{\mathsf{ThMon}^{\underline{#1}}} 

\newcommand{\sfx}{\mathsf{x}} 

\newcommand{\zigzag}[2][]
{
{\ifx&#1&%
  {\mathsf{ZigZag}_{#2}}
\else
  {\mathsf{ZigZag}_{#2}^{#1}}
\fi}
}

\newcommand{\ltfor}[2][] 
{
{\ifx&#1&%
  {\mathsf{LTFor}_{#2}}
\else
  {\mathsf{LTFor}_{#2}^{#1}}
\fi}
}

\newcommand{\rtfor}[2][] 
{
{\ifx&#1&%
  {\mathsf{RTFor}_{>#2}}
\else
  {\mathsf{RTFor}_{>#2}^{#1}}
\fi}
}

\newcommand{\suppfor}[2][] 
{
{\ifx&#1&%
  {\mathsf{For}_{#2}}
\else
  {\mathsf{For}_{#2}^{#1}}
\fi}
}

\newcommand{\binfor}[1][]{
{\ifx&#1&%
    \operatorname{\mathsf{BinFor}}
\else
    \operatorname{\mathsf{BinFor}}^{#1}
\fi}
} 

\newcommand{\Bin}[1]{\widehat{#1}}
\newcommand{\hqsym}[2][]{
{\ifx&#1&%
  {\operatorname{HQSym}_{#2}}
\else
  {\operatorname{HQSym}_{#2}^{#1}}
\fi}
} 

\newcommand{\coinv}[1]{\operatorname{Coinv}_{#1}} 

\newcommand{\qscoinv}[2][]{
{\ifx&#1&%
  {\operatorname{QSCoinv}_{#2}}
\else
  {{}^{#1}\!\operatorname{QSCoinv}_{#2}}
\fi}
}

\newcommand{\hsym}[1]{\operatorname{HSym}_{#1}} 
\newcommand{\cube}{\mathbf{C}} 
\newcommand{\vol}{\operatorname{Vol}} 
\definecolor{ao}{rgb}{0.0, 0.5, 0.0}
\newcommand{\col}{\overline{\rho}_F} 
\newcommand{\deri}[1]{\operatorname{D}_{#1}} 
\newcommand{\vope}[2][]{
{\ifx&#1&%
  {\mathsf{V}_{#2}}
\else
  {\mathsf{V}_{#2}^{\underline{#1}}}
\fi}
} 

\newcommand{\paths}[1]{\operatorname{Paths}(#1)}

\newcommand{\sfc}{\mathsf{c}}
\newcommand{\sfd}{\mathsf{d}}

\author{Philippe Nadeau}
\address{Universite Claude Bernard Lyon 1, CNRS, Ecole Centrale de Lyon, INSA Lyon, Université Jean Monnet, ICJ UMR5208, 69622 Villeurbanne, France}
\email{\href{mailto:nadeau@math.univ-lyon1.fr}{nadeau@math.univ-lyon1.fr}}

\author{Hunter Spink}
\address{Department of Mathematics,
University of Toronto, Toronto, ON M5S 2E4, Canada}
\email{\href{mailto:hunter.spink@utoronto.ca}{hunter.spink@utoronto.ca}}

\author{Vasu Tewari}
\address{Department of Mathematical and Computational Sciences, University of Toronto Mississauga, Mississauga, ON L5L 1C6, Canada}
\email{\href{mailto:vasu.tewari@utoronto.ca}{vasu.tewari@utoronto.ca}}

\thanks{
PN was partially supported by French ANR grant ANR-19-CE48-0011 (COMBIN\'E). HS and VT acknowledge the support of the Natural Sciences and Engineering Research Council of Canada (NSERC), respectively [RGPIN-2024-04181] and [RGPIN-2024-05433].}


\begin{document}
\title{Quasisymmetric divided differences}

\begin{abstract}
We develop a quasisymmetric analogue of the combinatorial theory of Schubert polynomials and the associated divided difference operators. Our counterparts are ``forest polynomials'', and a new family of  linear operators, whose theory of compositions is governed by forests and the ``Thompson monoid''. 

We then give several applications of our theory to fundamental quasisymmetric polynomials, the study of quasisymmetric coinvariant rings and their associated harmonics, and positivity results for various expansions. In particular we resolve a conjecture of Aval--Bergeron--Li regarding quasisymmetric harmonics.  Our approach extends naturally to $m$-colored quasisymmetric polynomials. 
\end{abstract}

\maketitle

\setcounter{tocdepth}{1}
\tableofcontents

\section{Introduction}
The ring of quasisymmetric functions $\qsym{}$, first introduced in Stanley's thesis \cite{StThesis} and further developed by Gessel \cite{Ges84}, is ubiquitous throughout combinatorics; see \cite{ABS06} for a high-level explanation and \cite{grinberg2020hopf} for thorough exposition.
Truncating to finitely many variables $\{x_1,\dots,x_n\}$ gives the ring of quasisymmetric polynomials $\qsym{n}$. 
The quasisymmetric polynomials are characterized by a weaker form of variable symmetry, and so contain the ring of symmetric polynomials $\sym{n}$.

Letting $\symide{n}$ denote the ideal in $\poly_{n}\coloneqq \ZZ[x_1,\dots,x_n]$ generated by positive degree homogeneous symmetric polynomials, the coinvariant algebra $\coinv{n}\coloneqq \poly_{n}/\symide{n}$ has been a central object of study for the past several decades. 
An important reason for this is its distinguished basis of Schubert polynomials \cite{LS82} and the divided difference operators \cite{BGG73} that interact nicely with this family-- see \cite{BS98, BJS93, BKTY04, FK96, FS94, HPW22, KM05, LeSo03} for a sampling of the combinatorics underlying this story.
In fact Schubert polynomials lift to a basis  of $\poly_n$.
The close relationship between the  combinatorics of symmetric and quasisymmetric polynomials leads to the natural question, first posed in \cite{ABB04}, of what can be said about the analogous quotient $\qscoinv{n}\coloneqq \poly_n/\qsymide{n}$, where $\qsymide{n}$ is the ideal generated by positive degree homogeneous quasisymmetric polynomials?

\smallskip

In this paper we develop a quasisymmetric analogue of the \textbf{combinatorial theory} of Schubert polynomials $\schub{w}$ and the divided differences $\partial_i$ which recursively generate them. 
The reader well-versed with the classical story should refer to Table~\ref{tab:sym_vs_qsym} for a comparison.
The role of Schubert polynomials $\schub{w}$ is played by the \emph{forest polynomials} $\oneforestpoly{F}$ of \cite{NT_forest}, and the role of the $\partial_i$ operators are played by certain new \emph{trimming operators} $\tope{i}$.
Just as Schubert polynomials generalize Schur polynomials, the forest polynomials generalize fundamental quasisymmetric polynomials, a distinguished basis of $\qsym{n}$.
The duality between compositions of trimming operators and forest polynomials allows us to expand any polynomial in the basis of forest polynomials.
In fact, a special case of our framework gives a remarkably simple method for directly extracting the coefficients of the expansion of a quasisymmetric polynomial in the basis of fundamental quasisymmetric polynomials.

The interaction between forest polynomials and trimming operators descends nicely to quotients by $\qsymide{n}$, and we thus obtain a basis comprising certain forest polynomials for $\qscoinv{n}$ as well.
Our techniques are robust enough to gain a complete understanding even in the case one quotients by homogeneous quasisymmetric polynomials of degree at least $k$ for any $k\geq 1$. By considering the adjoint operator to trimming under a natural pairing on the polynomial ring, we are able to easily construct $\qsymide{n}$-harmonics, which turn out to have a basis given by  the volume polynomials of certain polytopes, answering a question of Aval--Bergeron--Li \cite{ABL10}.

\smallskip

In \cite{ NST_3} we investigate the underlying \textbf{geometric theory}, drawing upon the geometric significance of the ordinary divided difference operator. 
We now proceed to a more detailed description of background as well as results.

\subsection*{Background and results}
Let us briefly recall the classical theory of symmetric and Schubert polynomials. Let $S_{\infty}$ be the permutations of $\NN=\{1,2,\ldots\}$ fixing all but finitely many elements, generated by the adjacent transpositions $s_i\coloneqq (i,i+1)$, and identify $S_n$, the permutations of $[n]\coloneqq \{1,\ldots,n\}$, with the subgroup $\langle s_1,\ldots,s_{n-1}\rangle$ fixing all $i\ge n+1$. Let $\poly_n\coloneqq \ZZ[x_1,\ldots,x_n]$, and denote by $\poly\coloneqq \bigcup_n \poly_n =\ZZ[x_1,x_2,\ldots]$ for the ring of polynomials in infinitely many variables. $S_n$ acts on $\poly_n$ by permuting variable subscripts, and we denote by $\sym{n}\subset \poly_n$ for the invariant subring of symmetric polynomials. Two of the most important tools for understanding $\poly_n$ as a $\sym{n}$-module are the $\ZZ$-basis of $\poly$ given by the \emph{Schubert polynomials} $\schub{w}$ of Lascoux--Sch\"utzenberger \cite{LS82}, and the \emph{divided difference} operators $\partial_i:\poly\to \poly$ given by \begin{align}
\label{eqn:partiali}
\partial_i(f)=\frac{f-s_i f}{x_i-x_{i+1}}
\end{align} 
where $s_i$ swaps $x_i,x_{i+1}$. 
They are related by the fact that Schubert polynomials are the unique family of homogeneous polynomials indexed by $w\in S_{\infty}$ such that $\schub{\idem}=1$, and denoting $\des{w}=\{i:w(i)>w(i+1)\}$ for the descent set of $w$ we have
\begin{align}
\label{eqn:partiali_on_schubs}
\partial_i\schub{w}=\begin{cases}\schub{ws_i}&\text{ if }i\in \des{w},\\0&\text{otherwise.}\end{cases}
\end{align}
The divided differences satisfy the relations $\partial_i^2=0$, $\partial_i\partial_{i+1}\partial_i=\partial_{i+1}\partial_i\partial_{i+1}$ and $\partial_i\partial_j=\partial_j\partial_i$ for $|i-j|\ge 2$. 
The monoid defined by this presentation is the \emph{nilCoxeter monoid}.  
These relations imply that $\partial_{i_1}\cdots \partial_{i_k}=0$ if $s_{i_1}\cdots s_{i_k}$ is not a reduced word, and we may define $\partial_w\coloneqq \partial_{i_1}\cdots \partial_{i_k}$ where $s_{i_1}\cdots s_{i_k}$ is any reduced word for $w$. 
The operators $\{\partial_w\suchthat w\in S_{\infty}\}$ are the nonzero composites of the $\partial_i$, and if we let $\ct f=f(0,0,\ldots)$ denote the constant term map, then  $\partial_w$ and $\schub{w}$ satisfy the~duality
\begin{align*}
\ct \partial_w\schub{w'}=\delta_{w,w'}.
\end{align*}
The following are a representative sampling of classical results concerning the relationship between $\sym{n}$ and $\poly_n$ which are solved by Schubert polynomials and divided differences.
\begin{enumerate}[align=parleft,left=0pt,label=(Fact \arabic*)]
    \item \label{intro:it1} (cf. \cite{BGG73,LS82}) The Schubert polynomials 
    \begin{itemize}
        \item $\{\schub{w}\suchthat \des{w}\subset [n]\}$ are a $\ZZ$-basis of $\poly_n$,
        \item $\{\schub{w}\suchthat w\not\in S_n \text{ and } \des{w}\subset [n]\}$ are a $\ZZ$-basis for $\symide{n}\subset \poly_n$, the ideal generated by positive degree homogeneous symmetric polynomials, and
        \item $\{\schub{w}:w\in S_n\}$ are a $\ZZ$-basis for the coinvariant algebra $\coinv{n}\coloneqq \poly_n/\symide{n}$.
    \end{itemize}
    \item \label{intro:it2} (cf. \cite{Man01})) The nil-Hecke algebra $\End_{\sym{n}}(\poly_n)$ of endomorphisms $\phi:\poly_n\to \poly_n$ such that $\phi(fg)=f\phi(g)$ whenever $f\in \sym{n}$, is generated as a noncommutative algebra by the divided differences $\partial_1,\ldots,\partial_{n-1}$ and (multiplication by) $x_1,\ldots,x_n$.
    \item \label{intro:it3} (cf. \cite{BGG73,St64}) The $S_n$-harmonics $\hsym{n}$, defined as the set of polynomials $f\in \QQ[\lambda_1,\ldots,\lambda_n]$ such that $g(\frac{d}{d\lambda_1},\ldots,\frac{d}{d\lambda_n})f=0$ whenever $g\in \sym{n}$ is homogeneous of positive degree, has a basis given by the ``degree polynomials'' $\schub{w}(\frac{d}{d\lambda_1},\ldots,\frac{d}{d\lambda_n})\prod_{i<j}(\lambda_i-\lambda_j)$ for $w\in S_n$.
\end{enumerate}

A research program \cite{ABB04, BeGa23, PeSa23,  PeSa22} that has garnered attention in recent years revolves around answering the following question, which is the focus of this article.
\begin{question}
    How do such results generalize to the \emph{quasisymmetric polynomials} $\qsym{n}\subset \poly_n$?
\end{question}

Recall that $f\in\qsym{n}$ if for any sequence $a_1,\ldots,a_k\ge 1$, the coefficients of $x_{i_1}^{a_1}\cdots x_{i_k}^{a_k}$ and $x_{j_1}^{a_1}\cdots x_{j_k}^{a_k}$ in $f$ are equal whenever $1\le i_1<\cdots < i_k \le n$ and $1\le j_1<\cdots < j_k \le n$.
Concretely, just as the ring of symmetric polynomials $\sym{n}\subset \poly_n$ are invariant under the natural action of the symmetric group $S_n$ permuting variable indices, the quasisymmetric polynomials $\qsym{n}$ are the ring of invariants under the \emph{quasisymmetrizing} action of $S_{n}$ on $\poly_n$ due to Hivert \cite{Hi00} where the transposition $(i,i+1)$ acts on monomials $\sfx^{\sfc}\coloneqq x_1^{c_1}\cdots x_n^{c_n}$ by
\begin{align}
\label{eqn:hivert1}
\bm\sigma_i\,\sfx^{\mathsf{c}}=\begin{cases}s_{i}\cdot \sfx^{\sfc} &\text{if }c_i=0\text{ or }c_{i+1}=0,\\ \sfx^{\sfc}&\text{otherwise.}
\end{cases}
\end{align}
Under this action, the orbit of $\sfx^\sfc$ is the set of $\sfx^{\sfc'}$ where the ordered sequence of nonzero entries of $\sfc'$ is the same as for $\sfc$, so e.g.  $x_1^3x_2+x_1^3x_3+x_2^3x_3\in \qsym{3}$.

Pursuing this parallel further, Aval--Bergeron--Bergeron \cite{ABB04} studied the \emph{quasisymmetric coinvariants} $\qscoinv{n}\coloneqq \ZZ[x_1,\dots,x_n]/\qsymide{n}$ and produced a basis of monomials indexed by Dyck paths which in particular implies that the dimension of this space is given by the $n$th Catalan number $\cat{n}$.
Subsequent work of Aval \cite{Av0507} and Aval--Chapoton \cite{AvCh1618} generalized these results to a variant of quasisymmetric polynomials $\qsym[m]{n}$ in several sets of equisized variables called $m$-quasisymmetric polynomials. On the other hand, in \cite{Hi00} an isobaric quasisymmetric divided difference $\frac{x_{i+1}f-x_{i}\bm\sigma_i f}{x_{i+1}-x_i}$ was studied, which was obtained by replacing $s_i$ with $\bm\sigma_i$ in the usual isobaric divided difference $\frac{x_{i+1}f-x_is_i\cdot f}{x_{i+1}-x_i}$ used to define Grothendieck polynomials. Unfortunately, the operators obtained by replacing $s_i$ with $\bm\sigma_i$ in the definition of $\partial_i$ do not appear to behave well under composition, nor do they descend to $\qscoinv{n}$ (unlike $\partial_1,\ldots,\partial_{n-1}\in \End_{\sym{n}}(\poly_n)$ which descend to endomorphisms of $\coinv{n}$).

We introduce a ``quasisymmetric divided difference formalism''\footnote{Unrelated to the similarly named ``quasisymmetric Schubert calculus'' of \cite{PeSa22}.} built around linear \textit{trimming} operators $\tope{i}:\poly\to \poly$ satisfying the relations
\begin{align*}
\tope{i}\tope{j}=\tope{j}\tope{i+1} \text{ for } i>j
\end{align*}
of the (positive) \emph{Thompson monoid} \cite{Sunic07}, implying that composite operators $\tope{F}$ are indexed by \emph{binary indexed forests} $F$ \cite[\S 3.1]{NT_forest} (see also \cite{BelkBrown05}).
Just as $\ker(\partial_1|_{\poly_n})\cap \cdots \cap \ker(\partial_{n-1}|_{\poly_n})=\sym{n}$, we have $\ker(\tope{1}|_{\poly_n})\cap \cdots \cap \ker(\tope{n-1}|_{\poly_{n}})=\qsym{n}$, justifying the name, and they descend to operators $\tope{1},\ldots,\tope{n-1}:\qscoinv{n}\to \qscoinv{n-1}$. 
We will  see that they interact with the family of forest polynomials $\oneforestpoly{F}$ \cite{NT_forest} analogously to how $\partial_i$ interacts with $\schub{w}$ with the role of $ws_i$ being played by a certain ``trimmed forest'' $F/i$, allowing us to tightly follow the classical theory to obtain analogues of all of the above results.
In particular, we resolve the following question.
\begin{question}[Aval--Bergeron--Li \cite{ABL10}] For $\hqsym{n}$ the analogously defined ``quasisymmetric harmonics'', find a combinatorially defined basis and show that every element of $\hqsym{n}$ is in the span of the partial derivatives of the degree $n-1$ quasisymmetric harmonics.
\label{intro:ABL10}
\end{question}

We will also state $\qsym[m]{n}$-analogues of all of the above results. For each $m$ we will define trimming operators $\tope[m]{i}$ satisfying the relations 
\begin{align*}
\tope[m]{i}\,\tope[m]{j}=\tope[m]{j}\,\tope[m]{i+m}\text{ for $i>j$}
\end{align*} of the $m$-Thompson monoid $\Th[m]$, whose compositions $\tope[m]{F}$ are indexed by \emph{$(m+1)$-ary indexed forests} $F\in \indexedforests[m]$. They interact analogously with a new family of ``$m$-forest polynomials'' $\{\forestpoly{F}: F\in \indexedforests[m]\}$ which when $m=1$ specialize to the aforementioned forest polynomials of \cite{NT_forest}, and when $m\to \infty$ become the monomial basis.

\subsection*{Outline of article}
See~\Cref{tab:sym_vs_qsym} for a quick overview of where the constructions and results analogous to the theory of Schubert polynomials appear in this paper.
\begin{table}
    \centering
    \begin{adjustbox}{max width=\textwidth}
    \renewcommand{\arraystretch}{1.2}
    \begin{tabular}{|c|c|c|c|}
    \hline
        $\mathsection$&& $\qsym{n}$  & $\sym{n}$\\
        \hline
        \ref{sec:quasisymmetric_polynomials}&\textbf{Divided differences} & $\tope{i}$ &  $\partial_i$  \\
        \hline
         \ref{sec:indexed_forest}&\textbf{Indexing combinatorics} & $F\in \indexedforests$ & $w\in S_\infty$\\
         & &Fully supported forests $\suppfor{n}$ &$S_n$\\
         &  &Forest code $\sfc(F)$ &Lehmer code $\lcode{w}$\\
         & &Left terminal set $\qdes{F}$ &Descent set $\des{w}$\\
         &  & $F/i$ for $i\in \qdes{F}$ &$ws_i$ for $i\in \des{w}$\\
         & &Trimming sequences $\Trim{F}$ & Reduced words $\red{w}$\\
         & &Zigzag forests $Z\in\zigzag{n}$ & Grassmannian permutations $\lambda$\\
         \hline
         \ref{sec:thompson}&\textbf{Monoid}&Thompson monoid &nilCoxeter monoid\\
         \hline
         \ref{sec:forest_polynomials} & \textbf{$\poly$-basis}  & Forest polynomials $\oneforestpoly{F}$ & Schuberts $\schub{w}$\\
         &\textbf{Composites} &$\tope{F}=\tope{i_1}\cdots \tope{i_k}$ for $\textbf{i}\in \Trim{F}$ &$\partial_w=\partial_{i_1}\cdots \partial_{i_k}$ for $\textbf{i}\in \red{w}$\\
         \hline
         \ref{sec:forests_as_analogues_of_schuberts}&\textbf{$\poly_n$-basis}  & $\{\oneforestpoly{F}\suchthat \qdes{F}\subset [n]\}$ & $\{\schub{w}\suchthat \des{w}\subset [n]\}$\\
         &\textbf{Duality} &$\ct \tope{F}\oneforestpoly{G}=\delta_{F,G}$ & $\ct \partial_w\schub{w'}=\delta_{w,w'}$\\
         \hline
         \ref{sec:positivity}&\textbf{Positive expansions} &$\oneforestpoly{F}\oneforestpoly{H}=\sum c^G_{F,H}\oneforestpoly{G}$,  $c^G_{F,H}\ge 0$ & $\schub{u}\schub{w}=\sum c^v_{u,w}\schub{v}$, $c^v_{u,w}\ge 0$\\
         \hline
        \ref{sec:fundamental}&\textbf{Invariant basis} & Fundamental qsyms $\oneforestpoly{Z}$ & Schur polynomials $s_\lambda$\\
        \hline
         \ref{sec:coinvs}&\textbf{Coinvariant basis}  & $\{\oneforestpoly{F}\suchthat F\in \suppfor{n}\}$ &$\{\schub{w}\suchthat w\in S_n\}$\\
         &\textbf{Coinvariant action}&$\tope{i}:\qscoinv{n}\to \qscoinv{n-1}$ &$\partial_i:\coinv{n}\to \coinv{n}$\\
         \hline
         \ref{sec:harmonics}&\textbf{Harmonic basis} & Forest volume polynomials & Degree polynomials\\
        \hline
    \end{tabular}
    \renewcommand{\arraystretch}{1}
    \end{adjustbox}
    \caption{Comparing the symmetric and quasisymmetric stories}
    \label{tab:sym_vs_qsym}
\end{table}
In \Cref{sec:quasisymmetric_polynomials} we introduce operators $\rope{i}$ and $\tope{i}$ which can be used to characterize quasisymmetry. 
In \Cref{sec:indexed_forest} we describe the combinatorics of certain binary forests $\indexedforests$. 
In \Cref{sec:thompson} we show that the compositional structure on $\indexedforests$ is given by the ``Thompson monoid''. In \Cref{sec:forest_polynomials} we define the forest polynomials $\forestpoly{F}$ for $F\in \indexedforests$ and show that the $\tope{i}$ operators give a representation of the Thompson monoid, implying their composites $\tope{F}$ are also indexed by $F\in \indexedforests$. 

In \Cref{sec:forests_as_analogues_of_schuberts} we show that $\tope{i}$ interacts with the forest polynomials $\oneforestpoly{F}$, which then leads to a number of spanning and independence properties for the forest polynomials. In particular, we show how to extract individual coefficients in forest polynomial expansions. 
In \Cref{sec:positivity} we show a number of positivity results concerning these expansions.
In \Cref{sec:fundamental} we show that the fundamental quasisymmetric polynomials are a subset of the forest polynomials, and use this to derive a simple formula for the fundamental quasisymmetric expansion of an arbitrary $f\in \qsym{n}$.
In \Cref{sec:coinvs} we show how the quasisymmetric divided difference formalism implies the analogue of \ref{intro:it1}.
In \Cref{sec:harmonics} we show the quasisymmetric analogue of \ref{intro:it3}, and resolve \Cref{intro:ABL10}.
As for \ref{intro:it2}, we study its quasisymmetric analogue in Section~\ref{sec:quasisymmetric_nil_hecke}.

In \Cref{sec:mQuasi} we describe a single-alphabet approach to the set $\qsym[m]{n}$ of $m$-quasisymmetric polynomials. 
We introduce operators $\tope[m]{i}$ which can be used to characterize them. From this we get analogues of essentially all results above. 
In \Cref{sec:ProofThatTrimsWork} we combinatorially prove the interaction between $\tope[m]{i}$ and $\forestpoly[m]{F}$. In \Cref{sec:table} we give a table of forest polynomials up to $4$ internal nodes.

\subsection*{Acknowledgements}
We would like to thank Dave Anderson, Nantel Bergeron, Lucas Gagnon, Darij Grinberg, Allen Knutson, Cristian Lenart, Oliver Pechenik, Linus Setiabrata, and Frank Sottile for several stimulating conversations/correspondence.

\section{Quasisymmetric polynomials}
\label{sec:quasisymmetric_polynomials}

Let $\nvect$ denote the set of all sequences $(c_i)_{i\in \NN}$ of nonnegative integers with \emph{finite support}, i.e. there are only finitely many nonzero $c_i$.
Given $\mathsf{c}\in \nvect$ we let
\begin{align*}
    \sfx^{\mathsf{c}}\coloneqq \prod_{i\geq 1}x_i^{c_i}.
\end{align*}

The ring $\qsym{n}$ of quasisymmetric polynomials was recalled in the introduction. Note that the defining condition on the monomials whose coefficients must be equal can be rephrased as: the coefficients of $\sfx^{\mathsf{c}}$ and $\sfx^{\mathsf{c'}}$ are equal if $\mathsf{c'}$ can be obtained from $\mathsf{c}$ by adding  or removing consecutive strings of zeros in $\mathsf{c}$. 
This essentially shows the following result due to Hivert, based on his quasisymmetrizing action~\eqref{eqn:hivert1}.

\begin{lem}[{\cite[Proposition 3.15]{Hi00}}]
\label{lem:sigmaqsymtest}
    $f(x_1,\ldots,x_n)\in \poly_n$ is quasisymmetric if and only if $f=\bm\sigma_1f=\cdots = \bm \sigma_{n-1}f$.
\end{lem}


\subsection{Quasisymmetry via the Bergeron--Sottile map $\rope{i}$}

It should be noted that $\bm\sigma_i$ does not respect multiplication , so for example the fixed point property in Lemma~\ref{lem:sigmaqsymtest} does not immediately imply that $\qsym{n}$ is a ring. The following result is at the heart of our understanding of quasisymmetric functions. It fixes this deficit of $\bm\sigma_i$ by using the equality of certain ring homomorphisms $\rope{i}$ to characterize quasisymmetry. 
This characterization does not seem to be widely known, although it was implicitly used in the study of the connection between quasisymmetric functions and James spaces by Pechenik--Satriano \cite{PeSa22}. We call $\rope{i}$ the \emph{Bergeron--Sottile map} because they were the first to introduce it \cite{BS98}, somewhat surprisingly, in the context of  Schubert calculus (see also \cite{BS02, LSS06}).

\begin{defn}
    For $f\in \poly$ we define
    \begin{align*}\rope{i}(f)=f(x_1,\ldots,x_{i-1},0,x_{i},\ldots).\end{align*}
\end{defn}
In other words, $\rope{i}(f)$ sets $x_i=0$ and shifts $x_j\mapsto x_{j-1}$ for all $j\ge i+1$. In particular, for $f\in \poly_n$ and $i\le n$ we have $\rope{i}(f)\in \poly_{n-1}$ is given by
\begin{align*}
    \rope{i}(f)=f(x_1,\ldots,x_{i-1},0,x_i,\ldots,x_{n-1}).
\end{align*}


\begin{thm}
\label{thm:Rqsymchar}
    $f\in \poly_n$ has $f\in \qsym{n}$ if and only if $\rope{1}f=\cdots =\rope{n}f$.
\end{thm}
\begin{proof}
  For $1\le i \le n-1$, and $\sfc=(c_1,\ldots,c_{n-1})$, the $\sfx^{\sfc}$-coefficient in $(\rope{i+1}-\rope{i})f$ is the difference of the coefficients of $\sfx^{\sfc'}$ and $\sfx^{\sfc''}$ where $\sfc'=(c_1,\ldots,c_{i-1},c_{i},0,c_{i+1},\ldots,c_{n-1})$ and $\sfc''=(c_1,\ldots,c_{i-1},0,c_i,c_{i+1},\ldots,c_{n-1})$.
This difference is $0$ if $f\in \qsym{n}$ and therefore $\rope{i+1}f-\rope{i}f=0$ in that case.

Conversely, the vanishing of $(\rope{i+1}-\rope{i})f$ implies by the above computation that for all $\sfc$ as above the difference of the $\sfx^{\sfc'}$ and $\sfx^{\sfc''}$ coefficients in $f$ is $0$.
Noting that for each $\sfd=(d_1,\ldots,d_n)$ we have either $\sfx^{\sfd}=\bm\sigma_i\, \sfx^{\sfd}$ or $\{\sfx^{\sfd},\bm\sigma_i\, \sfx^{\sfd}\}=\{\sfx^{\sfc'},\sfx^{\sfc''}\}$ for some $\sfc=(c_1,\ldots,c_{n-1})$, we deduce that  $(\idem-\bm\sigma_i)\cdot f=0$. 
Since this is true for $1\le i \le n-1$, we have $f\in\qsym{n}$  by~\Cref{lem:sigmaqsymtest}.
\end{proof}

\begin{cor}
\label{cor:qsym_ring}
    $\qsym{n}$ is a ring.
\end{cor}

\begin{proof}
    If $f,g\in \qsym{n}$ then
    $\rope{i}(fg)=\rope{i}(f)\rope{i}(g)=\rope{i+1}(f)\rope{i+1}(g)=\rope{i+1}(fg)$ for $1\le i \le n-1$, so $fg\in \qsym{n}$.
\end{proof}

This is a classical result; see \cite[Proposition 5.1.3]{grinberg2020hopf} for a proof in the setting of quasisymmetric functions.
The typical proof that $\qsym{n}$ is closed under multiplication involves identifying an explicit basis whose multiplication can be explicitly computed.  
In contrast our algebraic proof only uses that $\rope{i}$ respects multiplication.

\subsection{Quasisymmetric divided differences}
\label{sec:quasi_divided_differences}

We now define the quasisymmetric analogue of $\partial_i$.

\begin{defn}
\label{defn:tope}
We define the operator $\tope{i}:\poly\to \poly$ by any of the equivalent expressions
\begin{align*}
\tope{i}f\coloneqq \rope{i}\partial_if=\rope{i+1}\partial_if=\frac{\rope{i+1}f-\rope{i}f}{x_i},
\end{align*}
\end{defn}

This is the quasisymmetric divided difference at the core of this work. 
We will usually call $\tope{i}$ a \emph{trimming operator}, for reasons that will be clearer in \Cref{sec:forests_as_analogues_of_schuberts}. 
For $f\in \poly_n$ and $1\le i \le n-1$, we have that $\tope{i}f\in \poly_{n-1}$ is given explicitly by
\begin{align*}\tope{i}(f)=\frac{f(x_1,\ldots,x_{i-1},x_i,0,x_{i+1},\ldots,x_{n-1})-f(x_1,\ldots,x_{i-1},0,x_i,x_{i+1},\ldots,x_{n-1})}{x_i}.\end{align*}

\begin{thm}
\label{thm:Tqsymchar}
    $f\in \poly_n$ is quasisymmetric if and only if $\tope{1}f=\cdots =\tope{n-1}f=0$.
\end{thm}
\begin{proof}
    This is a rephrasing of \Cref{thm:Rqsymchar} since we have $\tope{i}(f)=0\Longleftrightarrow \rope{i+1}(f)=\rope{i}(f)$.
\end{proof}
\begin{eg}
    Let  $f=x_1^2x_2+x_1^2x_3+x_2^2x_3$.
    Then we can verify by inspection that $f\in \qsym{3}$. Alternatively, we can  compute
    \begin{align*}
    \tope{1}(f)&=\frac{1}{x_1}(f(x_1,0,x_2)-f(0,x_1,x_2))=\frac{1}{x_1}(0+x_1^2x_2+0-0-0-x_1^2x_2)=0,\\
    \tope{2}(f)&=\frac{1}{x_2}(f(x_1,x_2,0)-f(x_1,0,x_2))=\frac{1}{x_2}(x_1^2x_2+0+0-0-x_1^2x_2-0)=0,
    \end{align*}
    which by \Cref{thm:Tqsymchar} implies $f\in \qsym{3}$.
\end{eg}

The twisted Leibniz rule for $\partial_i$ is $
\partial_i(fg)=\partial_i(f)g+(s_{i}\cdot f)\partial_i(g).$ Applying $\rope{i}$ to both sides of this equality and noting that $\rope{i+1}=\rope{i}s_i$ gives 
 an analogous rule for $\tope{i}$.

\begin{lem}[Twisted Leibniz rule]
\label{lem:leibniz}
For $f,g\in \poly$ we have
\begin{align*} 
\tope{i}(fg)=\tope{i}(f)\rope{i+1}(g)+\rope{i}(f)\tope{i}(g).
\end{align*}
\end{lem}


\section{Indexed forests}
\label{sec:indexed_forest}

We now discuss our primary data structure, namely indexed forests.
These forests, along with several combinatorial properties, already appear in \cite{NT_forest}.
We shall throughout compare our notions with their classical $S_\infty$--counterparts, for which we refer the reader to \cite{Mac91,Man01,St99}.

The collection of indexed forests $\indexedforests$ serves for $\tope{i}$ a role analogous to that of $S_{\infty}$ for $\partial_i$.
In \Cref{sec:thompson} we will describe a natural monoid product $F\cdot G$ on $\indexedforests$ and in \Cref{sec:forest_polynomials} it will be shown that composites of the $\tope{i}$ are indexed by $F\in \indexedforests$ in such a way that $\tope{F}\tope{G}=\tope{F\cdot G}$.

\subsection{Binary trees and indexed forests}
\label{subsec:trees_and_forest}

A \emph{binary tree} is a rooted tree where a node $v$ either has no children (in which case it is called a \emph{leaf}) or has two  ordered children $v_L,v_R$, its left child and right child (in which case $v$ is called \emph{internal}). 
We write $\internal{T}$ for the set of internal nodes. 
We write $|T|=|\internal{T}|$, and refer to this as the \emph{size} of $T$.

We write $\ast$ for the \emph{trivial} singleton rooted binary tree with $|\ast|=0$, and all other trees we call \emph{nontrivial}.
Note that $\internal{\ast}=\emptyset$, and the unique node of $\ast$ is both a root node and a leaf.

We are now ready to introduce our main combinatorial object.

\begin{defn}\label{def:indexed_forest}
    An \emph{indexed forest} is an infinite sequence $T_1,T_2,\ldots$ of binary trees where all but finitely many of the trees are $\ast$.
    We write $\indexedforests$ for the set of all indexed forests.
\end{defn}

Note that by labeling the leaves of each tree successively, we identify the leaves of $F$ with $\mathbb{N}$, associating the $i$'th leaf with $i\in \NN$.
Figure~\ref{fig:indexed_forest_eg} depicts an $F\in \indexedforests$. 
The bottom labels are the leaves, represented by crosses, identified with $\NN$.
Note that $T_2,T_4$ and $T_7$ are the only nontrivial trees.

Notions that apply to trees are now inherited by indexed forests. We write $\internal{F}=\bigcup_{i=1}^{\infty}\internal{T_i}$, and $|F|=|\internal{F}|$. 
In this way, the totality of nodes in $F$ is identified with $\internal{F}\sqcup \NN$. 
For $v\in \internal{F}$ we always write $v_L,v_R\in \internal{F}\sqcup \NN$.

We say that a  node $v\in \internal{F}$ is  \emph{terminal} if all its children are leaves.
 The forest all of whose trees are trivial is called the \emph{empty} forest and is denoted by $\emptyset$. 
 Finally, for $F\in \indexedforests$ we define its \emph{support} $\supp(F)$ to be the set of leaves in $\NN$ associated to the nontrivial trees in $F$, and for fixed $n\ge 1$ we define the class of forests
\begin{align*}
\suppfor{n}=\{F\in \indexedforests\suchthat \supp(F)\subset [n]\}.
\end{align*}
This class of forests plays the role in our theory of $S_n\subset S_{\infty}$ for fixed $n$.
\begin{figure}[!ht]
    \centering
    \includegraphics[scale=0.7]{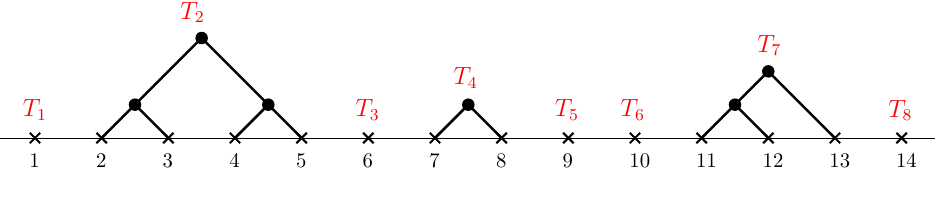}
    \caption{An indexed forest in $\suppfor{14}$}
    \label{fig:indexed_forest_eg}
\end{figure}

The indexed forest $F$ in Figure~\ref{fig:indexed_forest_eg}
 has six internal nodes and four terminal nodes. 
 In particular, its size $|F|$ is equal to $6$.
 Its support $\supp(F)$ equals $\{2,3,4,5,7,8,11,12,13\}$.
 It follows that $F$ belongs to $\suppfor{n}$ for any $n\geq 13$.

\begin{rem}
    Indexed forests were introduced  in~\cite{NT_forest}  with a slightly different notion of support, defined as follows. 
   Given a finite set $S$ of positive integers, an indexed forest with support $S$ is  the data of a plane binary tree with leaves $\{a,\ldots,b\}$ for each maximal interval $I=\{a,a+1,\ldots,b-1\}$ in $S$. By ordering these binary trees from left to right, and interspersing trivial trees given by the leaf labels that are not part of any nontrivial tree, we obtain objects clearly equivalent to the indexed forests of \Cref{def:indexed_forest}.
    This notion of support from \cite{NT_forest} was adapted to a ``parking function'' interpretation, and our notion of support is computed by replacing $S$ with $S\cup \{x: x-1\in S\}$. Although a slightly coarser notion, it is more suited to the perspective of this work.
\end{rem}

We have the following characterization:

\begin{fact}
\label{lem:fullysupported}
    $F\in \suppfor{n}$ if and only if the leaves of the first $n-|F|$ trees of $F$ are $\{1,\ldots,n\}$ and these contain all nontrivial trees.
\end{fact}

By an application of the \emph{Cycle Lemma} (cf. \cite[\S 2.1]{DZ90}), we get the enumeration $|\suppfor{n}|=\frac{1}{n+1}\binom{2n}{n}.$
This is the usual Catalan number $\cat{n}$, which we know to be the dimension of $\qscoinv{n}$ by \cite{ABB04}, a fact that will be reproved in \Cref{sec:coinvs}.

\subsection{The code $\sfc(F)$}
\label{subsec:codes_and_descents}

We now discuss an encoding of indexed forests by sequences of nonnegative numbers, playing the role of the \emph{Lehmer code} on $S_{\infty}$. The latter is defined as the sequence $\lcode{w}\coloneqq (c_i)_{i\in \NN}\in \nvect$ given by $c_i=\#\{i<j\suchthat w(i)>w(j)\}$.

Let $F\in \indexedforests$. 
We define the \emph{flag} $\rho_F:\internal{F}\to \NN$ by setting $\rho_F(v)$ to be  the label of the leaf obtained by going down left edges starting from $v$.

\begin{defn}
\label{defn:code_forest}
The \emph{code} $\sfc(F)$ is defined as
\begin{align*}
    \sfc(F)=(c_i)_{i\in \NN} \text{ where } c_i=|\{v\in \internal{F}\suchthat \rho_F(v)=i\}|.
\end{align*}
\end{defn}

The following result is ~\cite[Proposition 3.3]{NT_forest}.

\begin{thm}
\label{thm:ForesttoCode}
    The map $\sfc:\indexedforests\to \nvect$ is a bijection.
\end{thm}

In particular any mathematical object indexed by $\indexedforests$ can be indexed by $\nvect$.
For the $F$ in Figure~\ref{fig:indexed_forest_eg} we have $\sfc(F)=(0,2,0,1,0,0,1,0,0,0,2,0,\dots)$.

\subsection{The left terminal set $\qdes{F}$}

Given an indexed forest $F$, we associate to it a set of indices that shall play a role analogous to the descent set $\des{w}$ for a permutation $w\in S_{\infty}$.
We let
\begin{align*}
    \qdes{F}\coloneqq \{\rho_F(v)\suchthat v \text{ a terminal node in } F\}.
\end{align*}
These are precisely the leaves arising as the leftmost children of terminal nodes.
For $F$ in \Cref{fig:indexed_forest_eg} we have $\qdes{F}=\{2,4,7,11\}$.
In terms of $\sfc(F)=(c_i)_{i\in \NN}$, the following criterion is an immediate consequence of the prefix traversal aspect of our bijection:
\begin{align}
\label{eq:qdes_criterion}
    i\in \qdes{F} \Longleftrightarrow c_{i}>0 \text{ and } c_{i+1}=0.
\end{align}
In particular,
\begin{align}
\label{eq:spacedqdes} i,j\in \qdes{F} \implies |i-j|\ge 2.
\end{align}

\subsection{Left and right terminally supported forests}
\label{subsec:lterf_rterf}

The following class of left-terminally supported forests plays the role of the set of permutations $w\in S_{\infty}$ with $\des{w}\subset [n]$ or equivalently $\lcode{w}=(c_1,\ldots,c_n,0,\ldots)$:
\begin{align*}
 \ltfor{n} &\coloneqq \{F\in \indexedforests\suchthat \qdes{F}\subset [n]\}\\
 &=\{F\in \indexedforests\suchthat \sfc(F)=(c_1,\ldots,c_n,0,\ldots)\}\\
 &=\{F\in \indexedforests\suchthat \rho_F(v)\le n\text{ for all }v\in \internal{F}\}
\end{align*}
where the second equality follows from~\eqref{eq:qdes_criterion}. 
$\ltfor{n}$ thus consists of those $F\in\indexedforests$ whose leaves arising as left children of internal nodes are supported on $[n]$, or equivalently such that the leftmost leaf descendant of any internal node lies in $[n]$. This latter identification implies 
$   \suppfor{n}\subset \ltfor{n}$.
For the forest $F$ in Figure~\ref{fig:indexed_forest_eg}, we have $F\in \ltfor{n}$ for all $n\geq 11$.

More generally for any subset $A\subset \NN$, an analogue of the set of permutations $w\in S_{\infty}$ with $\des{w}\subset A$ is
\begin{align*}
    \ltfor{A}=\{F\in \indexedforests\suchthat \qdes{F}\subset A\},
\end{align*}
and for $A=[n]$ we recover $\ltfor{n}=\ltfor{A}$.

The following class of right-terminally supported forests play the role of the set of permutations $w\in S_{\infty}$ with $\des{w}\cap [n-1]=\emptyset$.
For a given $n\geq 1$ and $F\in \indexedforests$, say that an internal node $v\in \internal{F}$ is supported on $[n]$ if all leaves that are descendants of $v$ lie in $[n]$. In particular $F\in  \suppfor{n}$ if and only if all its internal nodes are supported on $[n]$. In contrast, let
\begin{align*}
\rtfor{n}\coloneqq &\{F\in \indexedforests\suchthat \text{ no }v\in \internal{F}\text{ is supported on }[n]\}.\\
=&\{F\in \indexedforests\suchthat v_R>n\text{ for all terminal }v\in \internal{F}\}\\
=&\ltfor{\{n,n+1,\ldots\}}.
\end{align*}

To reorient the reader, in terms of leaves we note the following characterizations: $F\in \rtfor{n}$ (\emph{resp.}  $\ltfor{n}$, \emph{resp.} $\suppfor{n}$) if and only if all rightmost leaves of $F$  are $>n$ (\emph{resp.} all leftmost leaves are $\le n$, \emph{resp.} all rightmost leaves are $\le n$).

\subsection{Zigzag forests}
\label{subsec:Zigzag}

The final class of forests we consider are the ``zigzag forests'', which will play an analogous role to the $n$-Grassmannian permutations $\grass{n}\coloneqq \{w\in S_{\infty}\suchthat \des{w}\subset \{n\}\}$.
\begin{align*}
         \zigzag{n} &\coloneqq\ltfor{n}\cap \rtfor{n}=\ltfor{\{n\}}\\
          & = \{F\in \indexedforests\suchthat \qdes{F}\subset \{n\}\}.
\end{align*}

In Figure~\ref{fig:zigzag_eg} we show a forest in $\zigzag{5}$.
We refer to these as \emph{zigzag forests}, since they consist of at most one nontrivial tree whose internal nodes form a chain. These were previously considered under the name linear tree in~\cite[Section 3.4]{NT_forest}.
From the definition it is clear that we have
\begin{align*}
    \zigzag{n}\subset \ltfor{n}.
\end{align*}

\begin{figure}[!ht]
    \centering
    \includegraphics[scale=0.8]{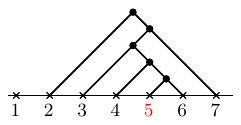}
    \caption{A forest $F\in \indexedforests$ in $\zigzag{5}$.
    \label{fig:zigzag_eg}}
\end{figure}

\subsection{Trimming and blossoming}
\label{subsec:trim_blossom}

We introduce two elementary operations of ``blossoming'' and ``trimming'' on forests, which play the role of the transformations $w\mapsto ws_i$ for $w\in S_{\infty}$ when $i\not \in \des{w}$ and $i\in \des{w}$ respectively.

\begin{defn}
For $F\in \indexedforests$ and any $i\in \NN$, the \emph{blossomed forest} $F\cdot i$ is obtained by making the $i$th leaf of $F$ into a terminal node by giving it $2$ leaf children.
 If $i\in \qdes{F}$, we define the \emph{trimmed forest} $F/i\in \indexedforests$ by removing the terminal node $v$ with $\rho_F(v)=i$.
\end{defn}

Clearly we always have $(F\cdot i)/i=F$, and if $i\in\qdes{F}$ we have $(F/i)\cdot i=F$. 
The reader curious about our choice of notation will find a satisfactory explanation in Section~\ref{sec:thompson}.

These operations are easily reflected in terms of codes. If $\sfc(F)=(c_i)_{i\in \NN}$ then for $i\in\NN$ we have
\begin{align}
\label{eqn:codeblossom}
    \sfc(F\cdot i)\coloneqq (c_1,\dots,c_{i-1},c_i+1,0,c_{i+1},c_{i+2},\dots).
\end{align}
In other words we increment the $i$th part of $\sfc(F)$ and insert a zero immediately after.
If $i\in \qdes{F}$ then $\sfc(F)=(c_1,\ldots,c_i,0,c_{i+2},c_{i+3},\ldots)$ with $c_i>0$ and
\begin{align*}
    \sfc(F/i)\coloneqq (c_1,\dots,c_{i-1},c_i-1,c_{i+2},c_{i+2},\dots).
\end{align*}
In words we decrement the $i$th part of $\sfc(F)$ and delete the zero to the immediate right.
See Figure~\ref{fig:trim_blossom_eg} depicting the twin operations. Make note of the shift in the indices comprising the support stemming from the addition/deletion of $0$s.

\begin{figure}[!ht]
    \centering
    \includegraphics[scale=0.6]{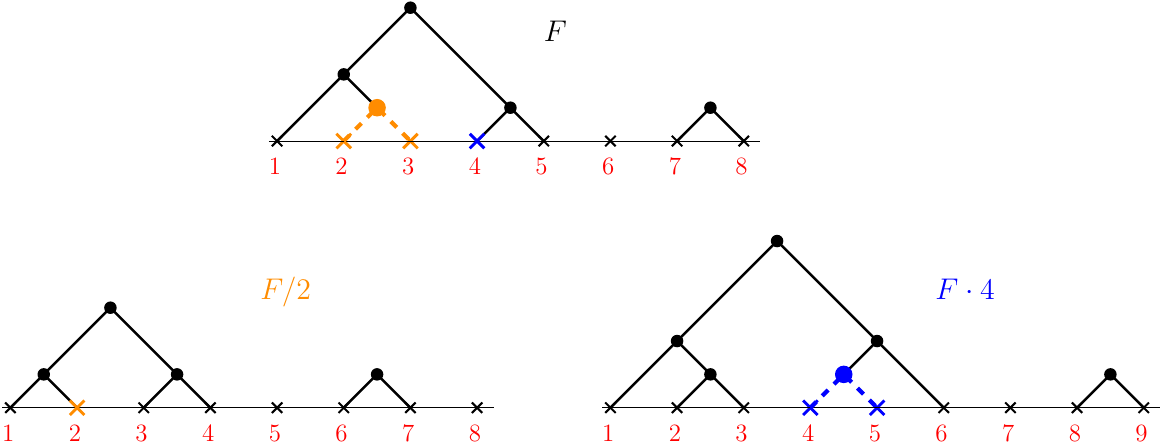}
    \caption{An $F\in\indexedforests$ with $\sfc(F)=(2,1,0,1,0,0,1,0,\dots)$, and the corresponding $F/2$ and $F\cdot 4$.}
    \label{fig:trim_blossom_eg}
\end{figure}

Iterating the notion of trimming, we obtain the notion of trimming sequences:

\begin{defn}For $F\in \indexedforests$ with $|F|=k$, we define $\Trim{F}$ recursively by setting $\Trim{\emptyset}=\{\emptyset\}$, and for $F\ne \emptyset$ we define
\begin{align*}
    \Trim{F}=\{(i_1,\ldots,i_k)\suchthat (i_1,\ldots,i_{k-1})\in \Trim{F/i_k}\text{ and }i_k\in \qdes{F}\}.
\end{align*}
\end{defn}

This plays the role of the set of reduced words $\red{w}$ for $w\in S_{\infty}$. Note that  the elements of $\Trim{F}$ are in obvious bijection with standard decreasing labelings of $F$, i.e. bijective labelings of $\internal{F}$ with numbers drawn from $\{1,\dots,|F|\}$  so that the labels decrease going down from root to terminal nodes.

\section{Forests and the Thompson monoids}
\label{sec:thompson}

We now develop the combinatorics of the \emph{Thompson monoid} $\Th$, which we will show in \Cref{sec:forest_polynomials} governs the composites of the $\tope{i}$ operators. By identifying this monoid with a monoid structure on $\indexedforests$, we will be able to index compositions of $\tope{i}$ operators as $\tope{i_1}\cdots \tope{i_k}=\tope{F}$ where $F\in \indexedforests$ and $(i_1,\ldots,i_k)\in\Trim{F}$. This is analogous to how we can index compositions of usual divided differences $\partial_{i_1}\cdots \partial_{i_k}=\partial_w$ with $w\in S_{\infty}$ for $(i_1,\ldots,i_k)$ a reduced word.

\subsection{A monoid structure on Forests}
\label{subsec:forest_monoid}

\begin{defn}
We define a monoid structure on $\indexedforests$ by taking for $F,G\in \indexedforests$ the composition $F\cdot G\in\indexedforests$ to be obtained by identifying the $i$th leaf of $F$ with the $i$th root node of $G$. The empty forest $\emptyset\in\indexedforests$ is the identity element.
\end{defn}

\begin{figure}[!ht]
    \centering
    \includegraphics[width=0.8\textwidth]{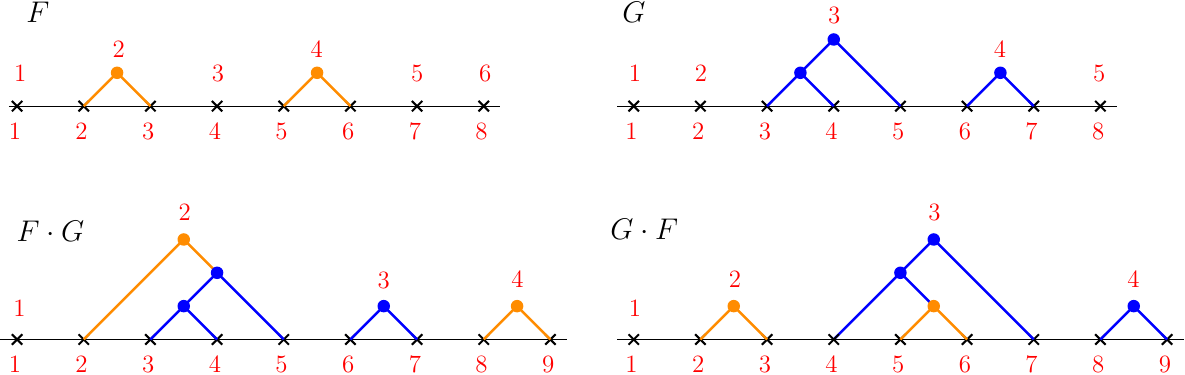}
    \caption{The products $F\cdot G$ and $G\cdot F$ for $F,G\in \indexedforests$, with both roots and leaves labeled}
    \label{fig:monoid_multiplication}
\end{figure}

If $H\in \indexedforests$, then factorizations $H=F\cdot G$ are in one-to-one correspondence with partitions $\internal{H}=A\sqcup B$ where $A$ is closed under taking parents and $B$ is closed under taking children, and then we may identify $A=\internal{F}$ and $B=\internal{G}$. An example of this is depicted in \Cref{fig:monoid_multiplication}.

Let $\wedge$ be the unique rooted plane binary tree with $|\wedge|=1$, and define $\underline{i}\in\indexedforests$ by
\begin{align*}\underline{i}=\underbrace{\ast \ast\cdots \ast}_{i-1} \wedge \ast \ast \cdots\end{align*}
We note that $F\cdot \underline{i}$ agrees with the blossoming $F\cdot i$ defined previously. 
With this notation, it is clear that for $F\in \indexedforests$ with $|F|=k$ we have
\begin{align*}
\Trim{F}=\{(i_1,\ldots,i_k): F=\underline{i_1} \cdots \underline{i_k}\}.
\end{align*}
The following shows that the $\underline{i}$ forests play an important role.

\begin{prop}
\label{prop:ForesttoCode}
     Every $F\in \indexedforests$ has a unique expression $F=\underline{1}^{c_1}\cdot\underline{2}^{c_2}\cdots$. The exponents are given by $\sfc(F)=(c_1,c_2,\ldots)$.
\end{prop}
\begin{proof}
The code map is a bijection by~\Cref{thm:ForesttoCode}, thus it suffices to show that $\sfc(\underline{1}^{c_1}\cdot\underline{2}^{c_2}\cdots)=(c_i)_{i\in \NN}$. We induct on $|c|\coloneqq \sum_{i\geq 1} c_i$. 

The result is trivial if $|c|=0$, so suppose that $|c|>0$. 
Suppose further that $n$ is the largest index so that $c_n>0$.
We have \begin{align*}\sfc(\underline{1}^{c_1}\cdot\underline{2}^{c_2}\cdots \underline{n}^{c_n})=\sfc(\underline{1}^{c_1}\cdot\underline{2}^{c_2}\cdots \underline{n}^{c_n-1}\cdot \underline{n})=\sfc(F'\cdot \underline{n})\end{align*}
where by the inductive hypothesis $\sfc(F')=(c_1,\ldots,c_{n-1},c_n-1,0,\ldots)$. Hence by~\eqref{eqn:codeblossom} we have $\sfc(F'\cdot \underline{n})=(c_1,\ldots,c_n,0,\ldots)$ as desired.
\end{proof}

The following says that the monoid $\indexedforests$ is \emph{right-cancellable}.

\begin{prop}
\label{prop:rightcancellable}
    For fixed $G\in \indexedforests$, the map $H\mapsto H\cdot G$ is an injection on $\indexedforests$.
\end{prop}
  Indeed,  by writing $G=\underline{i_1}\cdots \underline{i_k}$, we can recover $H$ from $H\cdot G$ by $H=(((H\cdot G/\underline{i_k})/\underline{i_{k-1}})\cdots )/\underline{i_1}$. We can thus define the following.
\begin{defn}
    For $F,G\in \indexedforests$, say $F\ge G$ if $F=H\cdot G$ for some $H\in \indexedforests$. If $F \ge G$ then we write $F/G\in \indexedforests$ to be the unique indexed forest with $F=(F/G)\cdot G$.
\end{defn}

The following is true in any right-cancellable monoid:

\begin{cor}
\label{prop:quotients}
    If $F\ge H$, then $G\ge F$ if and only if both $G\ge H$ and $G/H\ge F/H$. Under either supposition we have $G/F=(G/H)/(F/H)$.
\end{cor}

\subsection{The Thompson monoid}
\label{subsec:thompson_monoid}

We consider the following monoid given by generators and relations presentation  (see \Cref{rem:Thname} for an explanation of the name).

\begin{defn}
\label{defn:ThMon}
    The Thompson monoid $\Th$ is the quotient of the free monoid $\{1,2,\ldots\}^*$ by the relations $i\cdot j=j\cdot (i+1)$ for $i>j$.
\end{defn}

It turns out to describe exactly our monoid structure on $\indexedforests$. 

\begin{thm}
\label{thm:thomisom} The map $\Th\to \indexedforests$ given by $i\mapsto \underline{i}$ is a monoid isomorphism.
\end{thm}

\begin{proof}
The monoid structure on $\indexedforests$ satisfies
\begin{align*}
    \underline{i}\cdot \underline{j}=\underbrace{\ast\cdots \ast}_{j-1}\wedge\underbrace{\ast\cdots \ast}_{i-j}\wedge\ast\ast\cdots=\underline{j}\cdot \underline{i+1}\text{ whenever }i>j,
\end{align*}
It follows that the map is a well-defined monoid morphism. It is surjective since the indexed forests $\underline{i}$ generate $\indexedforests$ by \Cref{prop:ForesttoCode}.
    Using the rules $i\cdot j=j\cdot (i+1)$ for $i>j$, every element $i_1\cdots i_k\in\Th$ can be written as $1^{c_1}\cdot 2^{c_2}\cdots$ for some $c_1,c_2,\ldots$ by moving the smallest $i_j$ to the front and recursing on the remainder of the word. But each $1^{c_1}2^{c_2}\cdots$ maps to a unique indexed forest $\underline{1}^{c_1}\cdot \underline{2}^{c_2}\cdots$ by \Cref{prop:ForesttoCode}, which establishes injectivity of the map.
\end{proof}

  From now on we will tacitly identify elements $i_1\cdots i_k\in \Th$ and the associated forest $\underline{i_1} \cdots \underline{i_k}\in \indexedforests$, and so omit the underlines.

\begin{rem}
\label{rem:Thname}
    By formally adding inverses to the elements of $\Th$ we obtain the \emph{Thompson group}
    \begin{align*}
    G_{2}\coloneqq \langle \{r_i\}_{i\in \NN}\suchthat r_ir_j=r_jr_{i+1}\text{ for }i>j\rangle,
    \end{align*}
    the group of piecewise-linear homeomorphisms $f:[0,1]\to [0,1]$, all of whose nonsmooth points lie in $\mathbb{Z}[\frac{1}{2}]$ and whose slopes are powers of $2$ \cite[\S 4]{Bro87}. 
    The elements of $\Th$ correspond to those maps whose nonsmooth points have $x$-coordinates of the form $1-\frac{1}{2^k}$. We refer the reader to \cite{BelkBrown05,CFPnotes96} for details and \cite{DehTes19,Sunic07} for further combinatorial considerations. 
\end{rem}

\subsection{A monoid factorization}
\label{subsec:monoid_factorization}

Consider the following canonical decomposition for permutations  $w\in S_{\infty}$ with $\des{w}\subset [n]$, which index the $n$-variable Schubert polynomials $\schub{w}(x_1,\ldots,x_n)$.
\begin{obs}\label{obs:grassmannian}
Fix $n\geq 1$.
Every $w\in S_{\infty}$  can be uniquely written as $w=uv$ where
 $\des{u}\cap [n-1]=\emptyset$ and $v\in S_n$.
 Here $v\in S_n$ is the unique permutation so that $w(v^{-1}(1))<w(v^{-1}(2))<\cdots < w(v^{-1}(n))$ and $u=wv^{-1}$.
Moreover $\des{w}\subset [n]$ if and only if $\des{u}\subset \{n\}$, i.e. $u$ is an $n$-Grassmannian permutation.
\end{obs}
Let us give an analogue of this factorization for forests, which will be of particular importance when studying quasisymmetric coinvariants in \Cref{sec:coinvs}.  
To state it, we need the map $\tau:\indexedforests\to \indexedforests$ defined by $\tau(F)=\ast,F$, which shifts the forest one unit to the right. For $G\in \indexedforests$ of the form $G=\ast, F$ we also write $\tau^{-1}(G)=F$, i.e. $\tau^{-1}$ shifts indexed forests one unit to the left if possible.

\begin{thm}
\label{thm:forestfactorization}
Let $n\geq 1$, and $F\in\indexedforests$. Let $H\le F$ be the forest induced by all internal nodes of $F$ that are supported on $[n]$. Then $F\mapsto (\tau^{|H|}(F/H),H)$ is a bijection:
\begin{align*}\Theta_n:\indexedforests\to \{(R,H)\in \rtfor{n}\times \suppfor{n}\suchthat R=\emptyset \text{ or } \min\supp R > |H|\}.
\end{align*}
It restricts to a bijection
\begin{align*}\Theta'_n:\ltfor{n}\to \{(G,H)\in \zigzag{n}\times \suppfor{n}\suchthat G=\emptyset \text{ or } \min\supp G > |H|\}.
\end{align*}
\end{thm}

We give an example of $\Theta'_n$ in \Cref{fig:forest_factorization_eg}. 
\begin{figure}[!ht]
    \centering
    \includegraphics[width=\textwidth]{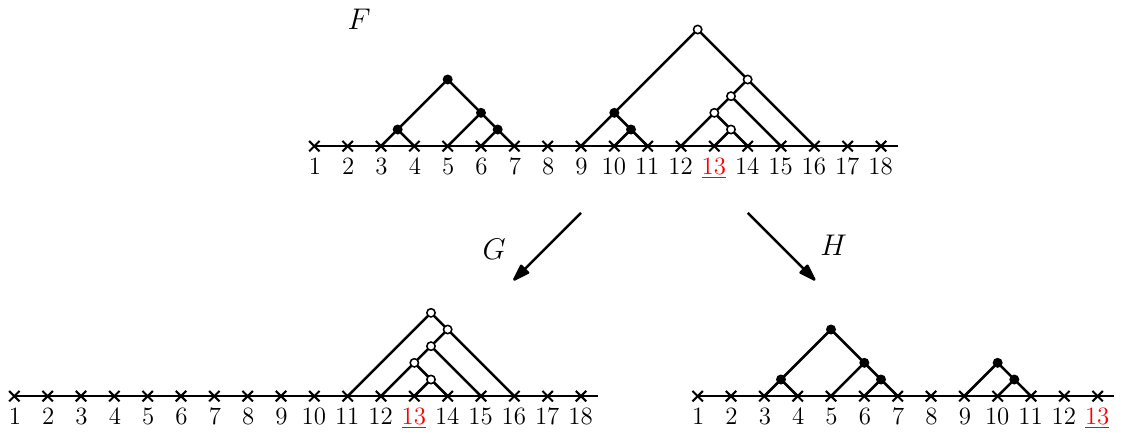}
    \caption{Example of the map $\Theta'_{n}$ for $n=13$. White and black vertices contribute to $G$ and $H$ respectively.}
    \label{fig:forest_factorization_eg}
\end{figure}

\begin{proof}
Let us first show that $\Theta_n$ is well-defined. By construction $H$ is clearly a subforest of $F$ that belongs to $\suppfor{n}$. 
By \Cref{lem:fullysupported} its first $n-|H|$ trees $T_1,\ldots,T_{n-|H|}$ have $[n]$ as the union of their leaves, and the other trees are trivial. As $F=(F/H)\cdot H$, we see that $F$ is obtained by grafting $T_1$ through $T_{n-|H|}$ to the first $n-|H|$ leaves of $F/H$. 
None of these first $n-|H|$ leaves can be the rightmost leaf of a node of $F/H$, as then the corresponding node in $F$ would be supported on $[n]$. 
It follows that $F/H\in \rtfor{n-|H|}$ and thus $\tau^{|H|}(F/H)\in \rtfor{n}$.
 So $\Theta_n$ is well-defined.

Clearly $\Theta_n$ is injective, as if $\Theta_n(F)=(R,H)$ then $F=(\tau^{-|H|}R)\cdot H$. 
Let us show surjectivity. 
Fix $(R,H)\in \rtfor{n}\times \suppfor{n}$ with $\min\supp R > |H|$. By definition of the monoid product, all nodes in $(\tau^{-|H|}R)\cdot H$ coming from $H$ are supported on $[n]$ since $H\in\suppfor{n}$. Now fix a node $v$ in $\tau^{-|H|}R$. Since $\tau^{-|H|}R\in \rtfor{n-|H|}$, the tree rooted at $v$ has a rightmost leaf descendant $>n-|H|$. Now the first $n-|H|$ trees in $H$ have leaf set $[n]$, so in $(\tau^{-|H|}R)\cdot H$ the tree rooted at the node coming from $v$ will have a rightmost leaf descendant $>n$. Thus no node in $(\tau^{-|H|}R)\cdot H$ coming from $\tau^{-|H|}R$ is supported on $[n]$.
It follows that $\Theta_n((\tau^{-|H|}R)\cdot H)=(R,H)$.

Assume now $F\in\ltfor{n}$, so that all leftmost leaves are $\leq n$, and let $\Theta'_n(F)=(G,H)$. 
If $v$ is a terminal node of $\tau^{-|H|}G$, then it has a leaf $>n-|H|$ since $\tau^{-|H|}G\in \rtfor{n-|H|}$. 
The corresponding node $v_F$ in $F=(\tau^{-|H|}G)\cdot H$ has a leaf descendant $\le n$ which implies that $v$ has also a leaf $\leq n-|H|$. 
This implies $\rho_{\tau^{-|H|}G}(v)=n-|H|$.
Since this holds for all terminal nodes of $\tau^{-|H|}G$ we have $\tau^{-|H|}G\in \zigzag{n-|H|}$, i.e. $G\in \zigzag{n}$. By the same reasoning in reverse we have that  $\tau^{-|H|}G\in \zigzag{n-|H|}$ implies that $F\in\ltfor{n}$, and thus $\Theta'_n$ is a bijection.
\end{proof}

\section{Forest polynomials $\oneforestpoly{F}$ and trimming operators $\tope{F}$}
\label{sec:forest_polynomials}

We now consider the family of \textit{forest polynomials} $\oneforestpoly{F}$ indexed by $F\in \indexedforests$ first introduced by the first and third authors \cite{NT_forest}. We also introduce composites $\tope{F}$ of the operators $\tope{i}$ indexed by the same set. These will play the roles of $\{\schub{w}:w\in S_{\infty}\}$ and $\{\partial_w:w\in S_{\infty}\}$ respectively.

\subsection{Forest polynomials $\oneforestpoly{F}$}
\label{subsec:forestpoly}

We begin by introducing the combinatorial definition. In the sequel we shall not need this; see Remark~\ref{remark:combinatorial_definition_useless}.
\begin{defn}[{\cite[Definition 3.1]{NT_forest}}]
    For $F\in \indexedforests$, define $\compatible{F}$ to be the set of all $\kappa:\internal{F}\to \NN$ such that for all $v\in \internal{F}$ with children $v_L,v_R\in \internal{F}\sqcup \NN$ we have
    \begin{itemize}
        \item $\kappa(v)\le \rho_F(v)$,
        \item If $v_L\in \internal{F}$ then
        $\kappa(v)\le \kappa(v_L)$, and if $v_R\in \internal{F}$ then
        $\kappa(v)< \kappa(v_R)$.
        \end{itemize}
    The forest polynomial $\oneforestpoly{F}$ is the generating function for $\compatible{F}$:
    \begin{align*}
    \oneforestpoly{F}=\sum_{\kappa\in \compatible{F}}\prod_{v\in \internal{F}}x_{\kappa(v)}.
    \end{align*}
\end{defn}

From $F\in \indexedforests$ and its eight fillings in Figure~\ref{fig:forest_poly_eg} we calculate that
\begin{equation}
\label{eq:forest_example_0201}
    \oneforestpoly{F}=x_{1}^{2} x_{2} + x_{1} x_{2}^{2} + x_{1}^{2} x_{3} + x_{1} x_{2} x_{3} + x_{2}^{2} x_{3} + x_{1}^{2} x_{4} + x_{1} x_{2} x_{4} + x_{2}^{2} x_{4}.
\end{equation}

\begin{figure}[!ht]
    \centering
    \includegraphics[width=0.8\textwidth]{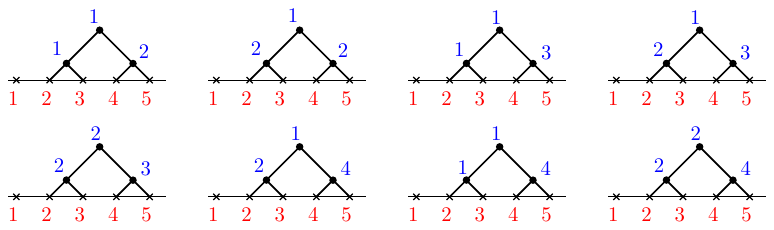}
    \caption{An $F\in \indexedforests$ with the eight fillings in $\compatible{F}$.}
    \label{fig:forest_poly_eg}
\end{figure}

Recall that the monomial expansion of Schubert polynomials can be written as
\begin{align*}
\schub{w}=\sfx^{\lcode{w}}+\sum_{\sfd<\lcode{w}}b_{\sfd}\sfx^{\sfd}
\end{align*}
where the ordering in the sum is the \emph{revlex} (reverse lexicographic) ordering.
The following fact is analogous.
\begin{prop}\label{prop:leading_monomial}
For $F\in \indexedforests$ we have the following expansion under the revlex ordering:
$$
\oneforestpoly{F}=\sfx^{\sfc(F)}+\sum_{\sfd<\sfc(F)} a_{\sfd}\sfx^{\sfd}.
$$
\end{prop}
\begin{proof}
The claim follows because the filling $\kappa(v)=\rho_F(v)$ always belongs to $\compatible{F}$, and every other filling gives a monomial that is smaller in the revlex ordering.
\end{proof}
Going back to the $F$ in Figure~\ref{fig:forest_poly_eg}, we have $\sfc(F)=(0,2,0,1,0,\dots)$, and $\sfx^{\sfc(F)}=x_2^2x_4$ is indeed the revlex leading term in $\forestpoly{F}$ computed in~\eqref{eq:forest_example_0201}.

An immediate corollary of Proposition~\ref{prop:leading_monomial} is that $\{\forestpoly{F}:F\in \indexedforests\}$ is a basis of $\poly$; even more, $\{\forestpoly{F}:F\in \ltfor{n}\}$ is a basis of $\poly_n$ for any $n\geq 1$. 
We will show this again in \Cref{prop:polynbasis} using the new divided difference formalism we will introduce shortly.

\subsection{Trimming operators $\tope{F}$}
\label{subsec:trimming_ops}

Let $\tope{}:\poly_{2}\to \poly_1$ be the operator
\begin{align}
\label{eqn:topedef}
    \tope{}(f)=\frac{f(x,0)-f(0,x)}{x}.
\end{align}
Viewing $\poly=\poly_1^{\otimes \infty}$ we have $\tope{i}=\idem^{\otimes i-1}\otimes \tope{}\otimes \idem^{\otimes \infty}$. Because of this, it turns out that composites $\tope{i_1}\cdots \tope{i_k}$ are naturally encoded by the structure of an indexed forest. For example, we can write $\tope{2}\tope{2}\tope{4}\tope{7}\tope{11}\tope{11}$ as
\begin{align*}
\idem\otimes \tope{}\left(\tope{}(\idem^{\otimes 2})\otimes\tope{}(\idem^{\otimes 2})\right) \otimes \idem \otimes \tope{}(\idem^{\otimes 2}) \otimes \idem^{\otimes 2} \otimes  \tope{}\left(\tope{}(\idem^{\otimes 2})\otimes \idem\right) \otimes \idem^{\otimes \infty}
\end{align*}
and this latter expression is nested via the parenthesization in a way that is encoded by $F=2\cdot 2 \cdot 4\cdot 7\cdot 11\cdot 11\in \indexedforests$, the forest in Figure~\ref{fig:indexed_forest_eg}. 

In this way $F$ can be thought of as encoding a composite $\tope{}$ operator taking inputs in the leaves and producing an output in the roots, which explains why the compositional structure of the $\tope{i}$ is reflected in the monoid composition on $\indexedforests$.

Using the Thompson monoid gives us a quick way to prove this identification.

\begin{prop}
\label{prop:topethompson}
    $\tope{i}\tope{j}=\tope{j}\tope{i+1}$ for $i>j$. In particular $i\mapsto \tope{i}$ induces a representation of $\Th$ via compositions of the $\tope{i}$ operators.
\end{prop}
\begin{proof}
We verify $\tope{i}\tope{j}=\idem^{\otimes j-1}\otimes \tope{}\otimes \idem^{i-j-1}\otimes \tope{}\otimes \idem^{\otimes \infty} =\tope{j}\tope{i+1}$.
\end{proof}

\begin{defn}
    For $F\in \Th$, define $\tope{F}\coloneqq \tope{i_1}\cdots \tope{i_k}$ for any expression $F=i_1\cdots i_k$.
\end{defn}

In the next section we develop the divided difference formalism relating forest polynomials $\{\oneforestpoly{F}: F\in \indexedforests\}$ to the trimming operators $\tope{F}$.

\section{Characterizing forest polynomials via trimming operators}
\label{sec:forests_as_analogues_of_schuberts}

This section forms the core of this work, the main result being \Cref{thm:forestunique}. 
Every result is exactly analogous to a corresponding result for divided differences $\partial_w$ and Schubert polynomials, with the following theorem being directly analogous to the interaction in~\eqref{eqn:partiali_on_schubs}.
We defer its proof by explicit computation to \Cref{sec:ProofThatTrimsWork}.

\begin{thm}
\label{thm:topetrims}
    For $F\in \indexedforests$ and $i\geq 1$ we have
         \begin{align*}
         \tope{i}\oneforestpoly{F}=
         \begin{cases}
            \oneforestpoly{F/i} & \text{if }i\in \qdes{F}\\
            0&\text{otherwise.}
         \end{cases}
         \end{align*}
     \end{thm}

In \Cref{fig:trimmingexample} we depict successive applications of trimming operators $\tope{i}$ to a forest polynomial $\oneforestpoly{F}$, which by \Cref{thm:topetrims} produces further forest polynomials associated to trimmed forests. If $\tope{i}$ does not appear then its application gives $0$.

\begin{figure}
    \centering    \includegraphics[scale=0.7]{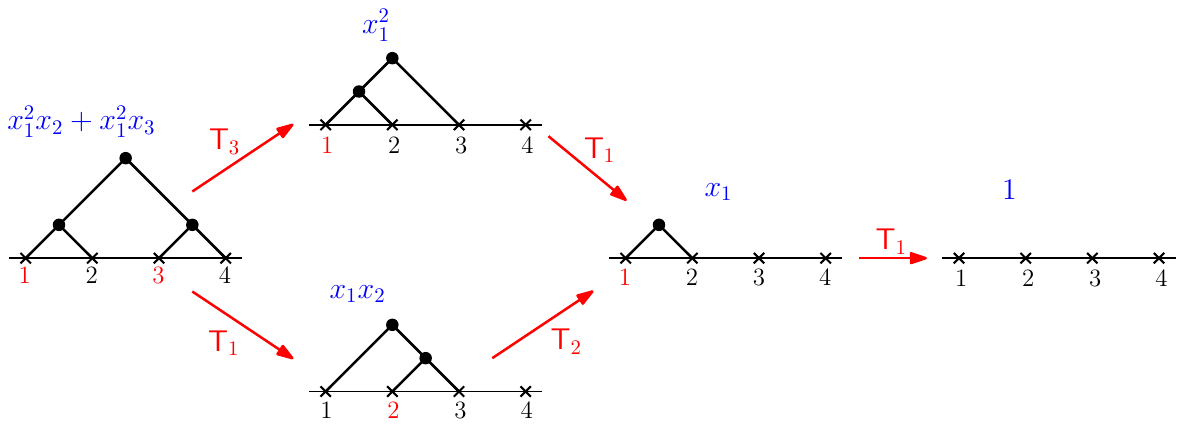}
    \caption{Sequences of $\tope{i}$ applied to  $\oneforestpoly{F}$ with $F=1\cdot 1\cdot 3\in \indexedforests$}
    \label{fig:trimmingexample}
\end{figure}

\begin{rem}
\label{remark:combinatorial_definition_useless}
    The actual definition of forest polynomials will play no role in all subsequent proofs. 
    As we shall see in \Cref{thm:forestunique}, the polynomials $\oneforestpoly{F}$ are in fact determined by the condition in \Cref{thm:topetrims}, homogeneity, and the normalization condition $\oneforestpoly{\emptyset}=1$. 
    We will use this characterization in proofs, signaling however when a simple alternative proof using the combinatorial definition can be given.
    
  The classical proof that Schubert polynomials exist (i.e. a homogeneous family of polynomials  satisfying~\eqref{eqn:partiali_on_schubs} exists) is by taking the ansatz $\schub{w_{0,n}}=x_1^{n-1}\cdots x_{n-1}$ for $w_{0,n}$ the longest permutation in $S_n$, showing  that
    $\partial_{w_{0,n-1}^{-1}w_{0,n}}\schub{w_{0,n}}=\schub{w_{0,n-1}}$ by direct computation, and then defining $\schub{u}=\partial_{u^{-1}w_{0,n}}\schub{w_{0,n}}$
    for $n$ sufficiently large so that $u\in S_n$.
    The forest polynomials do not seem to have sufficiently elementary descriptions for some well-chosen sequence of forests $F_n$ such that every other $G\in \indexedforests$ has $F_n\ge G$. 
    So it does not seem possible to proceed in a similar manner.
\end{rem}

\begin{lem}
\label{lem:kernelofall}
    $\bigcap_{i\ge n+1}\ker(\tope{i})=\poly_n$. In particular, $\bigcap_{i\ge 1}\ker(\tope{i})=\mathbb{Z}$.
\end{lem}
\begin{proof}
Clearly $\poly_n\subset \bigcap_{i\ge n+1}\ker(\tope{i})$. Conversely, if $k\ge n+1$ and $f(x_1,\ldots,x_k)$ is a polynomial depending nontrivially on $x_k$, then $\tope{k}f=\frac{1}{x_k}(f-f|_{x_k=0})\ne 0$ so $f\not \in \ker(\tope{k})$.
\end{proof}

\begin{thm}
\label{thm:forestunique}
    The family of forest polynomials $\{\oneforestpoly{F}:F\in \indexedforests\}$ is uniquely characterized by the properties $\oneforestpoly{\emptyset}=1$, $\oneforestpoly{F}$ is homogeneous, and $\tope{i}\oneforestpoly{F}=\delta_{i\in \qdes{F}}\oneforestpoly{F/i}$.
    (Here $\delta_{i\in \qdes{F}}$ equals $1$ if $i\in \qdes{F}$ and 0 otherwise.)
\end{thm}
\begin{proof} 
    It follows from the definition of forest polynomials and \Cref{thm:topetrims} that they satisfy these properties.
    Suppose there were another such family of polynomials $\{H_F:F\in \indexedforests\}$. 
    From $\tope{i}H_F=\delta_{i\in \qdes{F}}H_{F/i}$ and $H_{\emptyset}=1$ we see by induction that $H_F$ has degree $|F|$. 
    By induction, assume that we know that $H_F=\oneforestpoly{F}$ for $|F|<k$. 
    Then given some $F\in \indexedforests$ with $|F|=k$ we have $\tope{i}(\oneforestpoly{F}-H_F)=\delta_{i\in \qdes{F}}(\oneforestpoly{F/i}-H_{F/i})=0$ for all $i$. Therefore by \Cref{lem:kernelofall} we have $\oneforestpoly{F}-H_{F}\in \mathbb{Z}$. 
    But $\oneforestpoly{F}$ and $H_F$ are homogeneous of degree $|F|>1$ so they must be equal.
\end{proof}

\begin{cor}
\label{cor:TFG}
    For $F,G\in \indexedforests$ we have \begin{align*}\tope{F}\oneforestpoly{G}=\begin{cases}\oneforestpoly{G/F}&\text{if }G\ge F\\0&\text{otherwise.}\end{cases}\end{align*}
    In particular, $\ct \tope{F}\oneforestpoly{G}=\delta_{F,G}$.
\end{cor}
\begin{proof}
We induct on $|F|$. 
Let $i\in \qdes{F}$, and write $\tope{F}\oneforestpoly{G}=\tope{F/i}\tope{i}\oneforestpoly{G}$.
This equals $\oneforestpoly{(G/i)/(F/i)}$ if both $G\ge i$ and $(G/i)\ge (F/i)$, and $0$ otherwise. 
Now by \Cref{prop:quotients} the first part follows.

Next, note that when $G\ge F$, the polynomial $\oneforestpoly{G/F}$ is homogeneous of degree $|G/F|$. 
The only way that $\ct\oneforestpoly{G/F}$ does not vanish is if $|G/F|=0$,  implying $G=F$. 
Conversely if $G=F$ then $G/F=\emptyset$ so $\ct\tope{F}\oneforestpoly{G}=\ct \oneforestpoly{\emptyset}=\ct 1=1$.
\end{proof}

\begin{cor}
\label{cor:faithfultope}
    The $\tope{i}$ operators give a faithful representation of the monoid algebra $\mathbb{Z}[\Th]$.
\end{cor}
\begin{proof}
    We know by \Cref{prop:topethompson} that they give a representation, so it suffices to show that if $\sum a_F\tope{F}=0$ then all $a_F=0$. 
    By applying the linear combination to $\oneforestpoly{G}$ for any $G$, and then applying $\ct$, we indeed obtain 
    \begin{equation*}
        0=\sum a_F\ct \tope{F}\,\oneforestpoly{G}=\sum a_F\,\delta_{F,G}=a_G. \qedhere
    \end{equation*}
\end{proof}

\begin{prop}
\label{prop:forestZbasis}
    The forest polynomials $\{\oneforestpoly{F}:F\in \indexedforests\}$ form a $\mathbb{Z}$-basis for $\poly$, and we can write any $f\in \poly$ in this basis as
    \begin{align*}
        f=\sum (\ct \tope{F}f)\,\oneforestpoly{F}.
    \end{align*}
\end{prop}
\begin{proof}
If we can write $f\in \poly$ as $f=\sum a_F \,\oneforestpoly{F}$, then by \Cref{cor:TFG} we have $a_F=\ct \tope{F}\,\oneforestpoly{F}$.
Therefore, to conclude it suffices to establish the identity $f=\sum (\ct \tope{F}f)\oneforestpoly{F}$.
We do so by induction on $d\coloneqq \deg(f)$.

For $d=0$ the result follows by writing $f$ as a multiple of $\oneforestpoly{\emptyset}=1$.
Assume now $d>0$ and that the result holds for all polynomials of smaller degree.
As $\deg (\tope{i}f)<d$ for all $i$, we have
    \begin{align*}
        \tope{i}\sum (\ct \tope{F}f)\,\oneforestpoly{F}=\sum_{F\ge i}(\ct \tope{F}f)\,\oneforestpoly{F/i}=\sum_{G}(\ct \tope{G}\tope{i}f)\,\oneforestpoly{G}=\tope{i}f.
    \end{align*}
    Hence by \Cref{lem:kernelofall} we have
    \begin{align*}
        f-\sum (\ct \tope{F} f)\,\oneforestpoly{F}\in \bigcap_{i\ge 1}\ker(\tope{i})=\mathbb{Z}.
    \end{align*}
    So $f$ and $\sum (\ct \tope{F}f)\,\oneforestpoly{F}$ can only differ in their constant term. 
    But in fact both have the same constant term $\ct f$, so they are equal.
\end{proof}

\begin{prop}
\label{prop:simulkernel}
    A $\mathbb{Z}$-basis for $\ker(\tope{F})$ is given by $\{\oneforestpoly{G}:G\not \ge F\}$. In particular if $\mathcal{S}\subset \indexedforests$ is a family of forests, then
    \begin{align*}\bigcap_{F\in \mathcal{S}}\ker(\tope{F})=\ZZ\{\oneforestpoly{G}\suchthat G\not\ge F \text{ for all } F\in \mathcal{S}\}.\end{align*}
\end{prop}
\begin{proof}
    By \Cref{prop:forestZbasis} we know that $\{\oneforestpoly{G}:G\not\ge F\}\subset \ker(\tope{F})$ so it suffices to show that they span. Given $f\in \ker(\tope{F})$, we can write it as $f=\sum a_G\oneforestpoly{G}$, and we want to show that $a_G=0$ for all $G$ such that $G\ge F$.
    Applying $\tope{F}$ we see that
    \begin{equation}\label{eq:a_f_must_be_0}
        0=\tope{F}f=\sum_{G\ge F}a_G\,\oneforestpoly{G/F}.
    \end{equation}
    The forests $G/F$ are all distinct by \Cref{prop:rightcancellable}. Since forest polynomials are linearly independent we deduce that~\eqref{eq:a_f_must_be_0} holds if and only if $a_G=0$ for all $G$ such that $G\ge F$.
\end{proof}

\begin{cor}\label{cor:simultopeker}
For $A\subset \NN$, a $\ZZ$-basis for the subring \begin{align*}\bigcap_{i\not \in A}\ker(\tope{i})\subset \poly
\end{align*} is given by $\{\oneforestpoly{G}:G\in \ltfor{A}\}$.
\end{cor}
\begin{proof}
    This is a subring since for each $i\not\in A$ we have $\ker(\tope{i})=\ker(\frac{1}{x_i}(\rope{i+1}-\rope{i}))$ is the subalgebra of polynomials on which the two ring maps $\rope{i+1},\rope{i}:\poly\to \poly$ agree. The basis fact follows from \Cref{prop:simulkernel} and the definition of $\ltfor{A}$.
\end{proof}
We extract the special case $A=\{1,\dots,n\}$ separately for ease of citation.
\begin{prop}
\label{prop:polynbasis}
    $\{\oneforestpoly{G}\suchthat F\in\ltfor{n}\}$ is a $\mathbb{Z}$-basis for $\poly_n$.
\end{prop}


We conclude with a proposition concerning the interaction between forest polynomials and $\rope{1}$ which will be useful in our study of quasisymmetric coinvariants.

\begin{prop}
\label{prop:R1forest}
We have $\tope{G}\rope{1}^k=\rope{1}^k\tope{\tau^kG}$ and
    \begin{align*}
    \rope{1}^k\oneforestpoly{F}=
    \begin{cases}
    \oneforestpoly{\tau^{-k}F}&\text{if }\tau^{-k}F\text{ exists (i.e. $k<\min \supp(F)$)}\\0&\text{otherwise.}
    \end{cases}
    \end{align*}
\end{prop}
\begin{proof}
    It is direct to check that $\tope{i}\rope{1}=\rope{1}\tope{i+1}$. 
    So, for any $G\in \indexedforests$ with code $(c_1,c_2,\ldots)$ we have
    \begin{align*}
        \tope{G}\rope{1}^k=(\tope{1})^{c_1}(\tope{2})^{c_2}\cdots \rope{1}^k=\rope{1}^k(\tope{1+k})^{c_1}(\tope{2+k})^{c_2}\cdots =\rope{1}^k\tope{\tau^{k}G}
    \end{align*}
    since $\sfc(\tau^kG)=(0^k,c_1,c_2,\ldots).$
    Therefore \[\ct\tope{G}\rope{1}^k\oneforestpoly{F}=\ct\rope{1}^k\tope{\tau^kG}\oneforestpoly{F}=\delta_{\tau^kG,F},\] which by \Cref{prop:forestZbasis} means that $\rope{1}^k\oneforestpoly{F}$ equals $\oneforestpoly{\tau^{-k}F}$ if $\tau^{-k}F$ exists, and is $0$ otherwise.
\end{proof}
Compare the preceding result with its well-known classical analogue: $\rope{1}^k\schub{w}$ for $w\in S_{\infty}$ equals $0$ unless $w(i)=i$ for $1\leq i\leq k$, i.e. $\lcode{w}=(0^k,c_{k+1},\ldots)$, and if this holds then $\rope{1}^k\schub{w}=\schub{w'}$ with $w'(i)=w(i+k)-k$, i.e. $\lcode{w'}=(c_{k+1},\ldots)$.

\section{Positive expansions}
\label{sec:positivity}

We say that $f\in \poly$ is \emph{forest positive} if the coefficients $a_F$ in the expansion
\begin{equation*}
    f=\sum_{F\in \indexedforests} a_F\,\oneforestpoly{F}
\end{equation*}
are nonnegative integers.
If, in addition,  $a_F\in \{0,1\}$ then $f$ is \emph{multiplicity-free} forest positive.

\begin{lem}
\label{lem:forestpos}
    A polynomial $f$ is (resp. multiplicity-free) forest positive if and only if $\tope{i}f$ is (resp. multiplicity-free) forest positive for all $i$. 
\end{lem}
\begin{proof}
    If $f=\sum_F a_F\oneforestpoly{F}$, then $\tope{i}f=\sum_{i\in \qdes{F}} a_F\oneforestpoly{F/i}$ which immediately shows both forward directions. 
    Conversely, for any $F$ we have $a_F$ is the coefficient of $\oneforestpoly{F/i}$ in $\tope{i}f$ for any $i\in \qdes{F}$ which shows the reverse direction.
\end{proof}

In the remainder of this section our computations will be almost entirely formal consequences of the twisted Leibniz rule $\tope{i}(fg)=\tope{i}(f)\rope{i+1}(g)+\rope{i}(f)\tope{i}(g)$ from \Cref{lem:leibniz}, together with the following identities which may be verified by direct computation:
\begin{align}
\label{eqn:TR}
    \tope{j}\rope{i}=\begin{cases}\rope{i-1}\tope{j}&\text{if }j< i-1\\
    \rope{j+1}\tope{j}+\rope{j}\tope{j+1}&\text{if }j=i-1\\ \rope{i}\tope{j+1}&\text{if }j> i-1.\end{cases}
\end{align}
\begin{prop}
\label{prop:Rforestpositive}
For $F\in \indexedforests$ we have
    $\rope{i}\,\oneforestpoly{F}$ is multiplicity-free forest positive.
\end{prop}
\begin{proof}
Induct on $|F|$. 
By \Cref{lem:forestpos} it suffices to show that $\tope{j}\,\rope{i}\,\oneforestpoly{F}$ is multiplicity-free forest positive for all $j$.
If $j< i-1$ then by~\eqref{eqn:TR} we have
\begin{equation*}
\tope{j}\,\rope{i}\,\oneforestpoly{F}
=\rope{i-1}\tope{j}\,\oneforestpoly{F}=\delta_{j\in \qdes{F}}\rope{i-1}\,\oneforestpoly{F/j}.
\end{equation*}
which is multiplicity-free forest positive by induction.
If $j> i-1$ then we have by~\eqref{eqn:TR} 
\begin{equation*}
\tope{j}\,\rope{i}\,\oneforestpoly{F}=\rope{i}\,\tope{j+1}\,\oneforestpoly{F}=\delta_{j+1\in \qdes{F}}\rope{i}\,\oneforestpoly{F/(j+1)}.
\end{equation*}
which is multiplicity-free forest positive by induction.
Finally if $j=i-1$ then we have by~\eqref{eqn:TR} 
\begin{equation*}
\tope{j}\,\rope{i}\,\oneforestpoly{F}=\rope{j+1}\tope{j}\,\oneforestpoly{F}+\rope{j}\tope{j+1}\,\oneforestpoly{F}=\delta_{j\in \qdes{F}}\rope{j+1}\oneforestpoly{F/j}+\delta_{j+1\in \qdes{F}}\rope{j}\,\oneforestpoly{F/(j+1)}.
\end{equation*}
Noting that we cannot have both $j,j+1\in \qdes{F}$ by \eqref{eq:spacedqdes}, this is multiplicity-free forest positive by induction.
\end{proof}

The next theorem states that the basis $(\oneforestpoly{F})_{F\in\indexedforests}$ of $\poly$ has positive structure constants; this was first proved in \cite{NT_forest} with a complicated combinatorial interpretation for the coefficients.

\begin{thm}
\label{thm:forestmultpos}
    For $F,G\in \indexedforests$ we have $\oneforestpoly{F}\oneforestpoly{G}$ is forest positive.
\end{thm}
\begin{proof}
Induct on $\deg(\oneforestpoly{F}\oneforestpoly{G})=|F|+|G|$. By \Cref{lem:forestpos} it suffices to show that $\tope{i}(\oneforestpoly{F}\oneforestpoly{G})$ is forest positive for all $i$. 
By \Cref{lem:leibniz} we have
    \begin{equation*}
    \tope{i}(\oneforestpoly{F}\oneforestpoly{G})=(\tope{i}\,\oneforestpoly{F})\rope{i+1}\oneforestpoly{G}+(\rope{i}\,\oneforestpoly{F})\tope{i}\,\oneforestpoly{G}.
    \end{equation*} 
    It suffices to show that each term on the right-hand side is forest positive. We do the first, the second is similar.
    Note that $\tope{i}\,\oneforestpoly{F}$ is either $0$ or equals $\oneforestpoly{F/i}$ which is homogeneous of degree $|F|-1$. 
    Similarly, by Proposition~\ref{prop:Rforestpositive} we know that  $\rope{i+1}\oneforestpoly{G}$ is forest positive and homogeneous of degree $|G|$.
    So the result follows by applying the inductive hypothesis.
\end{proof}


Schubert polynomials are known to satisfy \emph{Monk's rule}, which shows that the Schubert expansion of $\schub{w}\schub{s_i}=\schub{w}(x_1+\cdots+x_i)$ is multiplicity-free. 
The same holds for forest polynomials.

\begin{thm}[forest polynomial ``Monk's Rule'']
\label{thm:monks}
    For $F\in \indexedforests$ we have $\oneforestpoly{\underline{i}}\,\oneforestpoly{F}=(x_i+x_{i-1}+x_{i-2}+\cdots+x_{1})\oneforestpoly{F}$ is multiplicity-free forest positive.
\end{thm}

\begin{proof}
    We induct on $|F|$.
    For $|F|=0$ the result is trivial, so assume that $|F|\ge 1$.
    Given $G\in \indexedforests$ with $|G|=|F|+1\ge 2$, we want to show that $\tope{G}(\forestpoly{\underline{i}}\,\oneforestpoly{F})\in \{0,1\}$.
    
    If there exists $j \in \qdes{G}$ with $j\ne i$, then by \Cref{lem:leibniz} we can write
    \begin{align*}
    \tope{G}(\forestpoly{\underline{i}}\,\oneforestpoly{F})=\tope{G/j}\tope{j}(\forestpoly{\underline{i}}\,\oneforestpoly{F})=\tope{G/j}(\rope{j}(\forestpoly{\underline{i}})\,\tope{j}(\oneforestpoly{F})).
    \end{align*}
    Now note from direct computation that \begin{align*}
    \rope{j}(\forestpoly{\underline{i}})=\rope{j}(x_i+x_{i-1}+x_{i-2}+\cdots+x_{1})=\begin{cases}\forestpoly{\underline{i-1}}&\text{if }j\le i\\ \forestpoly{\underline{i}}&\text{if }j\ge i+1\end{cases}\end{align*}
    and $\tope{j}(\oneforestpoly{F})=\delta_{j\in \qdes{F}}\oneforestpoly{F/j}$. 
    So we are done by induction.

    Otherwise, we have $\qdes{G}=\{i\}$. 
    As $\tope{G}(\forestpoly{\underline{i}}\,\oneforestpoly{F})=\tope{G/i}(\tope{i}(\forestpoly{\underline{i}}\,\oneforestpoly{F}))$, it remains to show that $\tope{G/i}(\tope{i}(\forestpoly{\underline{i}}\,\oneforestpoly{F}))$ is multiplicity-free.
    
    We claim that $\qdes{G/i}=\{j\}$ for $j=i-1$ or $i$. Indeed, any $k\in \qdes{G/i}$ must have $k\ge i-1$ since otherwise $k\in \qdes{G}$ as well, and now since $\qdes{G/i}\subset \{i-1,i\}$ we conclude $|\qdes{G/i}|=1$ by~\eqref{eq:spacedqdes}.

    If $\qdes{G/i}=\{i\}$ then by \Cref{lem:leibniz} and \eqref{eqn:TR} we can write $\tope{G/i}(\tope{i}(\forestpoly{\underline{i}}\,\oneforestpoly{F}))$ as
    \begin{align*}
    \tope{G/i}(\rope{i}(\oneforestpoly{F})+\forestpoly{\underline{i}}\tope{i}(\oneforestpoly{F}))=
    \tope{(G/i)/i}(\rope{i}\tope{i+1}(\oneforestpoly{F}))+\tope{G/i}(\forestpoly{\underline{i}}\,\tope{i}(\oneforestpoly{F})).
    \end{align*}
    At most one of the terms is nonzero since we cannot have both $i,i+1\in \qdes{F}$ by~\eqref{eq:spacedqdes}.
    If the first term is nonzero then we conclude since $\rope{i}\oneforestpoly{F/(i+1)}$ is multiplicity-free, and if the second term is nonzero then we conclude by induction that $\oneforestpoly{\underline{i}}\,\oneforestpoly{F/i}$ is multiplicity-free.

    The case $\qdes{G/i}=i-1$ is similar and left to the reader.
\end{proof}

Schubert polynomials also enjoy multiplicity-free \emph{Pieri rules} \cite{Sot96} corresponding to multiplication by elementary or homogeneous symmetric polynomials, which
happen to be forest polynomials for the forests with codes $(0^{p-k},1^k)$ and $(0^{p-1},k)$ respectively.
In view of this it is natural to inquire if forest polynomials have multiplicity-free Pieri rules as well. 
This is not the case in general; one finds multiplicities in low degree already.

\begin{rem}
Note that while all of the above positivity proofs unwind to give combinatorially nonnegative algorithms, it would be interesting to obtain the final coefficients directly as the answer to enumerative questions. We leave this to the interested reader. 
\end{rem}

\section{Fundamental quasisymmetrics and $\zigzag{n}$}
\label{sec:fundamental}

The $n$-Grassmannian permutations parametrize the special subclass of Schubert polynomials $\schub{w}$ known as the $n$-variable \emph{Schur polynomials}, which form a basis of $\sym{n}$.
In our story  $\zigzag{n}$ will play an analogous role to $\grass{n}$.
We will show that the associated forest polynomials $\{\oneforestpoly{F}:F\in \zigzag{n}\}$ lie in  $\qsym{n}$ and turn out to form the known basis of $\qsym{n}$ of \emph{fundamental quasisymmetric polynomials} \cite{Ges84,StThesis}. One consequence of this is that we can write down a new formula (\Cref{cor:qsymexpansion}) directly computing the coefficients of an quasisymmetric polynomial in its fundamental expansion. 
The only other direct formula for these coefficients in the literature is in the special case that $f\in\sym{n}$: Gessel \cite[Theorem 3]{Ges84} showed that these coefficients can be computed via the Hall inner product of $f$ with a ribbon skew-Schur polynomial.

  For an integer sequence $a=(a_1,\ldots,a_k)$ with $a_i\ge 1$ we define the set of \emph{compatible sequences} \begin{align*}\compatible{a}=\{(i_1,\ldots,i_k):a_j\ge i_j\ge i_{j+1},\text{ and if }a_j>a_{j+1}\text{ then }i_j>i_{j+1}\}.\end{align*}
  Given a sequence $\mathbf{i}=(i_1,\dots,i_k)$ we denote $\sfx_{\mathbf{i}}\coloneqq x_{i_1}\cdots x_{i_k}$.
  Then we define the \emph{slide polynomial} to be the generating function
    \begin{equation}
        \slide{a}=\sum_{\mathbf{i}\in \compatible{a}}\sfx_{\mathbf{i}}.
    \end{equation}
The notion of a compatible sequence appears in the Billey--Jockusch--Stanley formula for Schubert polynomials \cite{BJS93}. Our indexing conventions agree with \cite{NT_Ppart} and differ from \cite{AS17} as we use sequences instead of weak compositions.
\begin{eg}
    Consider $a=422$ wherein we have omitted commas and parentheses in writing the sequence for readability. We have
    \begin{align*}
\slide{422}&=\sfx^{(0,2,0,1)}+\sfx^{(2,0,0,1)}+\sfx^{(0,2,1,0)}+\sfx^{(2,0,1,0)}+\sfx^{(2,1,0,0)}+\sfx^{(1,1,0,1)}+\sfx^{(1,1,1,0)}.
    \end{align*}
    The corresponding $\compatible{a}$ are $
    \{422,411,322,311,211,421,321\}$.
\end{eg}

Like with forest polynomials, it is easy to check that the revlex leading monomial of $\slide{a}$ is $\sfx^{\sfc}$ where $\sfc=(c_i)_{i\in \NN}\in \nvect$ is determined by $c_i=\#\{a_j=i\suchthat 1\leq j\leq k\}$. The fundamental quasisymmetric polynomials constitute a subfamily of slide polynomials \cite[Lemma 3.8]{AS17}.

\begin{defn}  \label{def:qseq_fund}
Let $\qseq{n}$ be the sequences $(a_1,\dots,a_k)$ of positive integers satisfying $a_1=n$ and $a_i-a_{i+1}\in \{0,1\}$ for $1\le i\le k-1$.
If $(a_1,\ldots,a_k)\in \qseq{n}$ then $\slide{a}\in \poly_n$ is called a \emph{fundamental quasisymmetric polynomial}.
\end{defn}

\begin{thm}
\label{thm:qseqbij}
The  mapping $(a_1,\ldots,a_k)\mapsto F=a_k\cdots a_1$ is a bijection
$\qseq{n}\to \zigzag{n}$. Under this bijection we have $\slide{a}=\oneforestpoly{F}$.
\end{thm}
\begin{proof}
We dispense with the case that $()\mapsto \emptyset$ and assume that all sequences and forests in what follows are nonempty.

First, we show that the map is well-defined. 
By \Cref{prop:ForesttoCode} we have $\sfc(F)=(c_1,c_2,\ldots)$ with $c_i=\#\{j: a_j=i\}$. It follows  that the only $c_i\ne 0$ which has a zero in front of it in $\sfc(F)$ is $c_{a_1}=c_n$.
By \eqref{eq:qdes_criterion} this means that $\qdes{F}=\{n\}$, and thus $F\in \zigzag{n}$.

This map is injective because $\sfc(F)$ determines the sequence of $a_i$. 
To show it is surjective, we show that if we write $F\in \zigzag{n}$ as $F=a_k\cdots a_1$ with $a_1\ge \cdots \ge a_k \ge 1$ then $(a_1,\ldots,a_k)\in \qseq{n}$.
To see this, note that $\sfc(F)=(c_i)_{i\in \NN}$ has the property that $i$ satisfies $c_{i}\ne 0$ precisely when $i=a_j$ for some $j$.
Because $|\qdes{F}|=1$ we conclude by~\eqref{eq:qdes_criterion} that $\qdes{F}=\{a_1\}$, and thus $a_1=n$ since $F\in \zigzag{n}$.
 It also implies that when $a_{i+1}\ne n$ there is no zero in front of $c_{a_{i+1}}$ in $\sfc(F)$.
This implies $a_{i}-a_{i+1}\le 1$ and thus
we conclude that $(a_1,\ldots,a_k)\in \qseq{n}$.

Finally, to show that $\slide{a}=\oneforestpoly{F}$, we claim that it suffices to show that
\begin{equation*}
    \tope{j}\,\slide{a}=\delta_{j,a_1}\slide{a'}
\end{equation*}
where $a'=(a_2,\ldots,a_{k})$. Indeed, this implies that $\tope{F}\slide{a}=\tope{a_k\cdots a_1}\slide{a}=1$, and for $G\ne F\in \indexedforests$ with $G=b_k\cdots b_1$ we have
$\tope{G}\slide{a}=\tope{b_k}\cdots \tope{b_1}\slide{a}=\delta_{a,b}=0$
so we conclude by \Cref{prop:forestZbasis}.

Clearly $\tope{j}\,\slide{a}=0$ for $j\ge a_1+1$ as $\slide{a}$ only uses variables $x_1,\dots,x_{a_1}$.
Next, for $j=a_1$ we note that every element $\textbf{i}\in \mathcal{C}(a)$ has $i_1$ maximal and $i_1\le a_1$, so $\tope{a_i}\sfx_{\textbf{i}}=\frac{1}{\sfx_{a_1}}\delta_{i_1,a_1}\sfx_{\textbf{i}}$. Therefore 
\begin{align*}
    \tope{a_1}\,\slide{a}=\frac{1}{x_{a_1}}\sum_{\substack{\mathbf{i}\in \compatible{a}\\i_1=a_1}}\sfx_{\mathbf{i}}=\sum_{\mathbf{i}'\in \compatible{a'}}\sfx_{\mathbf{i}'}=\slide{a'}.
\end{align*}
Finally, because $\slide{a}$ is quasisymmetric we have by \Cref{thm:Tqsymchar} that $\tope{j}\slide{a}=0$ for $1\le j\le n-1$.
\end{proof}

    The identity $\slide{a}=\oneforestpoly{F}$ when $F\in \zigzag{n}$ also follows directly from the combinatorial definition of the forest polynomial: indeed the nodes in $\internal{F}$ form a path from the root with $c_1=\#\{j:a_j=1\}$ nodes with $\rho_F(v)=1$, followed by $c_2=\#\{j:a_j=2\}$ nodes with $\rho_F(v)=2$, etc. Then the conditions for a sequence to be in $\compatible{a}$ are easily seen to correspond bijectively to the ones for the colorings $\kappa$ in the definition of forest polynomials. We leave the easy verification to the reader.

\begin{eg}
Consider the element of $\zigzag{5}$ from Figure~\ref{fig:zigzag_eg}. The corresponding element of $\qseq{5}$ is $a=(5,4,3,3,2)$, and
the corresponding slide polynomial equals
\begin{align*}
    \slide{54332}=x_2x_3^2x_4x_5+x_1x_3^2x_4x_5+x_1x_2x_3x_4x_5+x_1x_2^2x_4x_5+x_1x_2^2x_3x_5+x_1x_2^2x_3x_4\end{align*}
Note that $\tope{5}\,\slide{54332}=x_2x_3^2x_4+x_1x_3^2x_4+x_1x_2x_3x_4+x_1x_2^2x_4+x_1x_2^2x_3=\slide{4332}$ as predicted by Theorem~\ref{thm:qseqbij}.
\end{eg}

We are now in position to identify a distinguished basis for $\qsym{n}$.

\begin{prop}
\label{prop:qsymbasis}
    $\qsym{n}$ has a $\mathbb{Z}$-basis $\{\slide{a}\suchthat a\in \qseq{n}\}$ of fundamental quasisymmetric polynomials.
\end{prop}
\begin{proof}
\Cref{thm:qseqbij} shows that $\{\oneforestpoly{G}\suchthat G\in \zigzag{n}\}$ is the set of fundamental quasisymmetric polynomials.
    We have by \Cref{thm:Tqsymchar} and \Cref{lem:kernelofall} that
    \begin{align*}\qsym{n}=\poly_n\cap \bigcap_{i=1}^{n-1}\ker(\tope{i})=\bigcap_{i\ne n}\ker(\tope{i}).\end{align*}
    By \Cref{prop:simulkernel} this equals $\mathbb{Z}\{\oneforestpoly{G}:G\in \ltfor{\{n\}}\}=\mathbb{Z}\{\oneforestpoly{G}:G\in\zigzag{n}\}$.
\end{proof}

In particular, using the $\tope{G}$ operators, for $f(x_1,\ldots,x_n)\in \qsym{n}$ we can directly extract the coefficients of the fundamental quasisymmetric expansion.

\begin{cor}
\label{cor:qsymexpansion}
    If $f(x_1,\ldots,x_n)\in \qsym{n}$ is homogeneous of degree $k$ then
    \begin{align*}
    f(x_1,\ldots,x_n)=\sum_{\mathbf{a}=(a_1,\dots,a_k)\in \qseq{n}} (\tope{\mathbf{a}}f)\,\slide{\mathbf{a}}
    \end{align*}
    where we have denoted the reverse composition $\tope{\mathbf{a}}\coloneqq \tope{a_k}\cdots \tope{a_1}$ for $\mathbf{a}=(a_1,\dots,a_k)$.
\end{cor}
\begin{proof}
    This follows from the formula in \Cref{prop:forestZbasis} and \Cref{thm:qseqbij}, since we have just shown that $f(x_1,\ldots,x_n)$ is in the $\mathbb{Z}$-span of $\{\forestpoly{G}:G\in \zigzag{n}\}$.
\end{proof}
\begin{eg}
\label{eg:fundexpansion}
    Say we want to decompose  $f(x_1,x_2,x_3)=2x_1^2x_2+2x_1^2x_3+2x_2^2x_3+x_1x_2^2+x_1x_3^2+x_2x_3^2 \in \qsym{3}$ into fundamental quasisymmetrics. 
    We track in Figure~\ref{fig:qsym_expansion_eg} the nonzero applications $\tope{i_3}\tope{i_2}\tope{i_1}f$ where $(i_1,i_2,i_3)\in \qseq{3}$, and read off $f=\slide{332}+2\slide{322}-3\slide{321}$.
    \begin{figure}[!ht]
        \centering
        \includegraphics[scale=0.8]{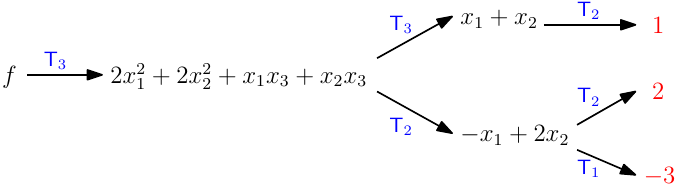}
        \caption{Trimming $f\in \qsym{3}$}
        \label{fig:qsym_expansion_eg}
    \end{figure}
\end{eg}

\section{Coinvariants}
\label{sec:coinvs}

In this section we first revisit the story of quasisymmetric coinvariants, showing that the basis of forest polynomials is perfectly adapted to their study, see~\Cref{thm:qsymidebasis}. We then describe the space of endomorphisms of $\poly_n$ that essentially commute with the multiplication by $\qsym{n}$, leading to a diagrammatic presentation of the space in the limit $n\to \infty$.

\subsection{Symmetric coinvariants}
\label{subsec:sym_coinvs}

One of the fundamental properties of the divided difference operators is that the operators $\partial_w:\poly_n\to \poly_n$ for $w\in S_n$ descend to the symmetric coinvariants $\partial_w:\coinv{n}\to \coinv{n}$. To show that $\partial_w$ descends, one shows that $\partial_i$ for $1\le i \le n-1$ stabilizes $\symide{n}$, a corollary of the fact that for $g\in \sym{n}$ and $f\in \poly_n$ that \begin{align*}\partial_i(gf)=g\partial_i(f).\end{align*}

Although usually proved by an appeal to algebraic geometry, directly from these facts one can use the usual divided difference formalism to show that the images of Schubert polynomials $\{\schub{w}\suchthat w\in S_n\}$ form a basis of $\coinv{n}$ and the images of Schubert polynomials $\{\schub{w}\suchthat w\not\in S_n\text{ and }\des{w}\subset [n]\}$ forms a basis of $\symide{n}$.
Unable to find such a proof in extant literature we include it here, if only to emphasize the parallel picture for $\qscoinv{n}$ and forest polynomials.

\begin{obs}
    $\{\schub{w}:w\in S_n\}$ forms a basis of $\coinv{n}$ and $\{\schub{w}\suchthat w\not\in S_n\text{ and }\des{w}\subset [n]\}$ forms a basis of $\symide{n}$.
\end{obs}
\begin{proof}
    Since $\{\schub{w}:\des{w}\subset [n]\}$ forms a basis of $\poly_n$ it suffices to show the basis statement for $\symide{n}$. 
    Consider the  factorization $w=uv$ into $v\in S_n$ and $u\in \grass{n}$ with $\ell(w)=\ell(u)+\ell(v)$ from \Cref{obs:grassmannian}.
    The key identity is that
    \begin{align}\label{eqn:symanalogue}
    \schub{u}\schub{v}=\schub{w}+\sum a_{v'}\schub{u' v'}\end{align}
    with $a_{v'}\in \mathbb{Z}$ where the sum is over pairs $(u',v')$ with $u'\in \grass{n}$ and $v'\in S_n$ such that $ \ell(u')>\ell(u)$ and $\ell(u'v')=\ell(u')+\ell(v')=\ell(w)$. This follows from noting that if $\ell(u')\le\ell(u)$ then $\ell(v')\ge \ell(v)$ and we have
    \begin{align*}
        \partial_{u'v'}(\schub{u}\schub{v})=\partial_{u'}\partial_{v'}(\schub{u}\,\schub{v})=\partial_{u'}(\schub{u}\,\partial_{v'}\schub{v})=\partial_{u'}\schub{u}\,\delta_{v,v'}=\delta_{u,u'}\delta_{v,v'}
    \end{align*}
    where in the second equality we used that $\schub{u'}\in \sym{n}$ and $v'\in S_n$.

     The identity~\eqref{eqn:symanalogue} shows upper-triangularity between $\{\schub{u}\schub{v}:v\ne \idem\}\subset \symide{n}$ and $\{\schub{w}:w=uv\text{ with }v\ne \idem\}=\{\schub{w}:w\not\in S_n, \des{w}\subset [n]\}$ which implies $\{\schub{w}:w\not\in S_n\text{ and } \des{w}\subset [n]\}\subset \symide{n}$.

    It remains to show that $\symide{n}\subset \mathbb{Z}\{\schub{w}:w\not\in S_n\text{ and }\des{w}\subset [n]\}$. 
    The identity~\eqref{eqn:symanalogue} also establishes upper-triangularity between  $\{\schub{u}\schub{v}\suchthat u\in \grass{n}\text{ and }v\in S_n\}$ and $\{\schub{w}\suchthat \des{w}\subset [n]\}$, which shows that $\poly_n$ is generated by $\{\schub{v}\suchthat v\in S_n\}$ as a $\sym{n}$-module. 
    Since we also know that $\{\schub{u}:\idem\ne u\in \grass{n}\}$ span the positive degree homogeneous symmetric polynomials, we are reduced to showing that $\schub{u}\schub{v}$ lies in $ \mathbb{Z}\{\schub{w}:w\not\in S_n\text{ and }\des{w}\subset [n]\}$ whenever $u\ne \idem$ and $v\in S_n$. 
    But this follows from~\eqref{eqn:symanalogue}.
\end{proof}
We note that our argument is reminiscent of computations in the proofs of \cite[Lemma 2.2 and Lemma 2.3]{RWY11}. The argument ibid. relies on a generalization of the factorization of a permutation used earlier to show that the corresponding Schubert polynomial lands in a certain ideal of $\poly_{n}$.

\subsection{Quasisymmetric coinvariants}
\label{subsec:quasi_coinvs}

Using the quasisymmetric divided difference formalism, we can follow a similar route.  Recall that $\qsymide{n}$ is the ideal in $\poly_n$ generated by all polynomials $f\in \qsym{n}$ with $\ct f=0$.
We define the \emph{quasisymmetric coinvariants} to be
\begin{align*}
\qscoinv{n}\coloneqq \poly_n/\qsymide{n}.
\end{align*}

We first establish the appropriate analogue of $\partial_i\in \End_{\sym{n}}(\poly_n)$ for our purposes.
\begin{prop}
\label{prop:fulltrim}
If $H\in \suppfor{n}$ and $g\in \qsym{n}$, then $\tope{H}(gh)=\rope{1}^{|H|}(g)\tope{H}(h)$ for all $h\in\poly$.
\end{prop}
\begin{proof}
    We proceed by induction on $|H|$. If $|H|=0$ then there is nothing to prove, so suppose now the result is true for all smaller $|H|$.
    Let $i\in \qdes{H}$.
    As $H\in \suppfor{n}$ we have $1\le i \le n-1$, so \Cref{thm:Rqsymchar} implies $\rope{i}(g)=\rope{1}(g)$.
    Together with  \Cref{lem:leibniz} and \Cref{thm:Tqsymchar} this implies
    \begin{align*}
    \tope{H}(gh)=\tope{H/i}\tope{i}(gh)=\tope{H/i}(\rope{i}(g)\tope{i}(h)+\rope{i+1}(h)\tope{i}(g))=\tope{H/i}(\rope{1}(g)\tope{i}(h)).
    \end{align*}
    We know that $H/i\in \suppfor{n-1}$. 
    Indeed, if there were a leaf $\ge n$ then as $i\le n-1$ this would become a leaf $\ge n+1$ in $(H/i)\cdot i=H$. 
    From the definition of $\qsym{n}$ we see that $\rope{1}(g)\in \qsym{n-1}$, and so by induction
    \begin{equation*}
    \tope{H}(gh)=\rope{1}^{(|H|-1)}(\rope{1}(g))\tope{H/i}(\tope{i}(h))=\rope{1}^{|H|}(g)\tope{H}(h).\qedhere 
    \end{equation*}
 \end{proof}
\begin{cor}
\label{cor:topedescend}
    For $F\in \suppfor{n}$ we have $\tope{F}(\qsymide{n})\subset \qsymide{n-|F|}$, and so $\tope{F}$ descends to a map 
    \begin{align*}
    \tope{F}:\qscoinv{n}\to \qscoinv{n-|F|}.
    \end{align*} 
    In particular, $\tope{1},\ldots,\tope{n-1}$ descend to maps
    \begin{align*}
    \tope{1},\ldots,\tope{n-1}:\qscoinv{n}\to \qscoinv{n-1}.
    \end{align*}
\end{cor}
\begin{proof}
    We have $\rope{1}^{|F|}(\qsym{n})\subset \qsym{n-|F|}$ from the definition of $\qsym{n}$, and $\rope{1}^{|F|}$ preserves the property of being a positive degree homogeneous polynomial, so we conclude by \Cref{prop:fulltrim} that $\tope{F}(\qsymide{n})\subset \qsymide{n-|F|}$.
\end{proof}

To state our key result \Cref{thm:qsymidebasis}, it is useful to introduce the partial map $\star$ taking a pair $(G,H)$ and returning the forest $(\Theta_n')^{-1}(G,H)$, where $\Theta_n'$ is from \Cref{thm:forestfactorization}:

 \begin{defn} Let $G\in \zigzag{n}$ and $H\in \suppfor{n}$. Then we define
     \begin{align*}G\star H\coloneqq\begin{cases}(\tau^{-|H|}G)\cdot H&\text{if }\min \supp G>|H|\text{ or }G=\emptyset\\\text{does not exist}&\text{otherwise.}\end{cases}\end{align*}
 \end{defn}

The second part of \Cref{thm:forestfactorization} can then be stated in the following equivalent form:
 \begin{cor}
 \label{cor:forestfactorization}
     Let $F\in \indexedforests$ and $n\geq 1$.
     Then $F\in\ltfor{n}$ if and only if we can write $F=G\star H$ with $G\in \zigzag{n}$ and $H\in \suppfor{n}$. In that case the decomposition $F=G\star H$ is unique: $H\le F$ is determined by having its set of internal nodes $\internal{H}\subset \internal{F}$ consist of all fully supported internal nodes of $F$, and $G=\tau^{|H|}(F/H)$.
 \end{cor}

Rather than just describe a basis for $\qsymide{n}$, we also describe bases of the ideals generated by homogeneous elements of $\qsym{n}$ of degree $\ge k$, which will be important in the next subsection. 

\begin{defn}
    For a nonnegative integer $k$, let $\mathcal{I}_{k,n}\subset \poly_n$ be the ideal generated by all homogeneous polynomials $f\in \qsym{n}$ with $\deg(f)\ge k$.
\end{defn}

We also define a subspace $\mathcal{I}_{k,n}^{\star} \subset \poly_n$ by
\begin{align*}\mathcal{I}_{k,n}^{\star}&=\bigoplus \mathbb{Z}\{\oneforestpoly{G\star H}\suchthat G\in \zigzag{n}, H\in \suppfor{n}, G\star H\text{ exists and }|G|\ge k\}\\
&=\bigoplus \mathbb{Z}\{\oneforestpoly{F}\suchthat F\in\ltfor{n}, \text{ and if } (G,H)=\Theta_n'(F) \text{ then }|G|\geq k\}.\end{align*}

Note that by  \Cref{prop:polynbasis} and \Cref{cor:forestfactorization}  we have \begin{align}
\label{eqn:I0n}
    \mathcal{I}_{0,n}^{\star}=\bigoplus \mathbb{Z}\{\oneforestpoly{F}\suchthat F\in \ltfor{n}\}=\poly_n=\mathcal{I}_{0,n}.
\end{align}
Directly from the definitions we also note 
\begin{align*}
    \mathcal{I}_{1,n}&=\qsymide{n},\text{ and }\\\mathcal{I}_{1,n}^{\star}&=\bigoplus \mathbb{Z}\{\oneforestpoly{F}\suchthat F\in \ltfor{n}\setminus \suppfor{n}\}.
\end{align*}

\begin{thm}
\label{thm:qsymidebasis}
We have $\mathcal{I}_{k,n}=\mathcal{I}_{k,n}^{\star}$ for all $k,n$. In particular for $k=1$, we get 
\begin{enumerate}[label=(\arabic*)]
    \item \label{8.8:it1} 
    $\qsymide{n}$ has a $\mathbb{Z}$-basis given by $\{\oneforestpoly{F}:F\in\ltfor{n}\setminus \suppfor{n}\}$.
    \item \label{8.8:it2} 
    $\qscoinv{n}$ has a $\mathbb{Z}$-basis given by $\{\oneforestpoly{F}:\suppfor{n}\}$.
    In particular its dimension is given by the Catalan number $\cat{n}$.
\end{enumerate}
\end{thm}

Note that \ref{8.8:it1} was shown in \cite[Corollary 4.3]{NT_forest} using a more computational approach, while \ref{8.8:it2} is the main result of \cite{ABB04}.

\begin{lem}
\label{lem:GHproduct} 
    If $G\in \zigzag{n}$ and $H\in \suppfor{n}$ then
    \begin{align*}\oneforestpoly{G}\oneforestpoly{H}-\delta_{G\star H\text{ exists}}\oneforestpoly{G\star H}\in \mathcal{I}_{|G|+1,n}^{\star}.\end{align*}
\end{lem}
\begin{proof}
    Since $\oneforestpoly{G}\oneforestpoly{H}$ is in $\poly_n$ and has degree $|G|+|H|$, its forest expansion only contains terms $\oneforestpoly{F}$ where $F\in\ltfor{n}$ and $|F|=|G|+|H|$. 
    Given such $F$, use~\Cref{cor:forestfactorization} to write
    \begin{align*}
    F=G'\star H'=(\tau^{-|H'|}G')\cdot H'
    \end{align*} with $H'\in \suppfor{n}$ and $G'\in \zigzag{n}$. 
    Thus $\oneforestpoly{G'}\in \qsym{n}$, and  so  \Cref{prop:fulltrim} yields
    \begin{align*}\tope{F}(\oneforestpoly{G}\oneforestpoly{H})=\tope{\tau^{-|H'|}G'}\tope{H'}(\oneforestpoly{G}\oneforestpoly{H})=\tope{\tau^{-|H'|}G'}(\rope{1}^{|H'|}(\oneforestpoly{G})\tope{H'}(\oneforestpoly{H})) 
    \end{align*}
    By \Cref{prop:R1forest} this vanishes unless $\tau^{-|H'|}G$ exists and in that case $\rope{1}^{|H'|}(\oneforestpoly{G})=\oneforestpoly{\tau^{-|H'|}G}$. We thus get by \Cref{cor:TFG}
    \begin{align*}
        \tope{F}(\oneforestpoly{G}\oneforestpoly{H})
        =
        \tope{\tau^{-|H'|}G'}(\oneforestpoly{\tau^{-|H'|}G}\oneforestpoly{H/H'})
    \end{align*}
    if $\tau^{-|H'|}G$ exists and $H'\geq H$, and is $0$ otherwise. 
    If $H'=H$ then necessarily $|G'|=|G|$, and so 
    \begin{equation}\tope{\tau^{-|H'|}G'}(\oneforestpoly{\tau^{-|H'|}G}\oneforestpoly{H/H'})=\tope{\tau^{-|H'|}G'}(\oneforestpoly{\tau^{-|H'|}G})=\delta_{G',G}.
    \qedhere
    \end{equation}
\end{proof}

\begin{proof}[Proof of \Cref{thm:qsymidebasis}]
\Cref{lem:GHproduct} implies that for each fixed degree $d$, the $\mathbb{Z}$-linear transformation between the degree $d$ homogeneous component of $\mathcal{I}^{\star}_{k,n}$ and \begin{align*}\mathbb{Z}\{\oneforestpoly{G}\oneforestpoly{H}:G\in \zigzag{n},H\in \suppfor{n},G\star H\text{ exists, }|G|\ge k,\text{ and }|G|+|H|=d\}\end{align*} taking $\oneforestpoly{G\star H}$ to $\oneforestpoly{G}\oneforestpoly{H}$ is strictly upper triangular and hence invertible. Therefore 
\begin{align*}\mathcal{I}^{\star}_{k,n}=\mathbb{Z}\{\oneforestpoly{G}\oneforestpoly{H}:G\star H\text{ exists and }|G|\ge k\}\end{align*}
and thus $\mathcal{I}^{\star}_{k,n}\subset \mathcal{I}_{k,n}$. 
As $\mathcal{I}^{\star}_{0,n}=\poly_n$ by~\eqref{eqn:I0n}, this shows that $\poly_n$ is spanned as a $\qsym{n}$-module by $\{\oneforestpoly{H}:H\in \suppfor{n}\}$. 

Now by \Cref{prop:qsymbasis} we also have $\{\oneforestpoly{G}:G\in \zigzag{n}\text{ and }|G|\ge k\}$ span the degree $\ge k$ homogeneous components of $\qsym{n}$ as a $\mathbb{Z}$-module.
Thus to show the inclusion $\mathcal{I}_{k,n}\subset \mathcal{I}^{\star}_{k,n}$ it suffices to show that $\oneforestpoly{G}\oneforestpoly{H}\in \mathcal{I}_{k,n}^{\star}$ whenever $|G|\ge k$ and $H\in \suppfor{n}$. 
This final statement follows from \Cref{lem:GHproduct}.
\end{proof}

As an application, consider the involution $\operatorname{rev}_n f(x_1,\ldots,x_n)=f(x_n,\ldots,x_1)$, which preserves $\qsym{n}$ and is thus an involution of $\qscoinv{n}$. 
We show that this involution interacts in a simple way with the basis $\{\oneforestpoly{F}\suchthat F\in \suppfor{n}\}$ of $\qscoinv{n}$ afforded by \Cref{thm:qsymidebasis} \ref{8.8:it2}. For such an $F$, we denote $\operatorname{mir}(F)$ the forest obtained by a vertical symmetry with respect to $\{1,\ldots,n\}$. 
\begin{prop}
\label{prop:involution}
    For $F\in\suppfor{n}$,
    \[\operatorname{rev}_n(\oneforestpoly{F})=(-1)^{|F|}\oneforestpoly{\operatorname{mir(F)}}\mod \qsymide{n}.\]
\end{prop}
\begin{proof}
By \Cref{thm:qsymidebasis} and \Cref{prop:forestZbasis}, we need to show that if $G\in\suppfor{n}$ with $|G|=|F|$, then $\tope{G}\operatorname{rev}_n(\oneforestpoly{F})=\delta_{G,\operatorname{mir}(F)}(-1)^{|G|}$.
For $1\leq j\leq n$, one has that $\rope{j}\operatorname{rev}_n=\operatorname{rev}_{n-1}\rope{n+1-j}$ as operators from $\poly_n$ to $\poly_{n-1}$, as both composite operators send $f(x_1,\ldots,x_n)$ to $f(x_{n-1},\ldots,x_j,0,x_{j-1},\ldots,x_1)$. 
For $1\leq j \leq n-1$ this then implies
\[\tope{j}\operatorname{rev}_n=\rope{j}\partial_{j}\operatorname{rev}_n=-\rope{j}\operatorname{rev}_n\partial_{n-j}=-\operatorname{rev}_{n-1}\rope{n+1-j}\partial_{n-j}=
-\operatorname{rev}_{n-1}\tope{n-j}.\]
Iterating this shows that $\tope{G}\operatorname{rev}_n=(-1)^{|G|}\operatorname{rev}_{n-|G|}\tope{\operatorname{mir}(G)}$, and applying this to $\oneforestpoly{F}$ gives the desired result as $\tope{\operatorname{mir}(G)}\oneforestpoly{F}=\delta_{\operatorname{mir}(G),F}=\delta_{G,\operatorname{mir}(F)}$.
\end{proof}
This is an analogue of the following classical fact for Schubert polynomials: if $w\in S_n$, then $\operatorname{rev}_n\schub{w}=(-1)^{\ell(w)}\schub{w_0ww_0}$ modulo $\symide{n}$ for $w_0$ the longest permutation in $S_n$. Note also that a combinatorial basis of $\qscoinv{n}$ with a similar behavior under $\operatorname{rev}_n$ was defined in~\cite{Cha05}.

\section{Harmonics}
\label{sec:harmonics}
In this section we compute a basis for the quasisymmetric harmonics in terms of the volume polynomials $V_F(\lambda)$ of certain ``forest polytopes'' $\cube_{F,\lambda}$ associated to a fully supported forest $F\in \suppfor{n}$ and a decreasing sequence $\lambda_1\ge \cdots \ge \lambda_n$. We also show that the quasisymmetric harmonics are spanned by the derivatives of the top degree quasisymmetric harmonics. This answers a question of Aval--Bergeron--Li \cite{ABL10}. 

We first introduce a perfect pairing between $\mathbb{Q}[x_1,\ldots,x_n]$ and $\mathbb{Q}[\lambda_1,\ldots,\lambda_n]$; see \cite{PS09}. 

\begin{defn}
    The $D$-pairing $\langle,\rangle_D:\mathbb{Q}[x_1,\ldots,x_n]\otimes \mathbb{Q}[\lambda_1,\ldots,\lambda_n]\to \mathbb{Q}$ is the bilinear form
    \begin{align*}
    \langle f,g\rangle_D=\mathrm{ev}_0^{\uplambda}\,  f(\deri{1},\ldots,\deri{n})\,g(\lambda_1,\ldots,\lambda_n),
    \end{align*}
    where $\deri{i}\coloneqq \frac{d}{d\lambda_i}$ and $\mathrm{ev}_0^{\uplambda}:\mathbb{Q}[\lambda_1,\ldots,\lambda_n]\to \mathbb{Q}$ is obtained by setting all $\lambda_i$ to zero.
\end{defn}

This pairing may be described alternatively as having $\langle \sfx^{\sfc},\uplambda^{\sfd}\rangle=\delta_{c,d}\,\sfc!$ where $\sfc=(c_1,\dots,c_n)$ and $\sfd=(d_1,\dots,d_n)$ are sequences of nonnegative integers, and $\sfc!\coloneqq c_1!\cdots c_n!$.

\begin{defn}
    The \emph{quasisymmetric harmonics}  are defined to be
    \begin{align*}
    \hqsym{n}\coloneqq &\{f\in \mathbb{Q}[\lambda_1,\ldots,\lambda_n]\suchthat \langle g,f\rangle_D=0\text{ for all } g\in \qsymide{n}\}\\=&\{f\in \mathbb{Q}[\lambda_1,\ldots,\lambda_n]\suchthat g(\deri{1},\dots,\deri{n})f=0\text{ for all } g\in \qsym{n}\text{ with }\ct g=0\}.
    \end{align*}
\end{defn}
The key insight is that we can translate the duality $\tope{F}\forestpoly{G}=\delta_{F,G}$  into a $D$-pairing duality $\langle \vope{F}(1),\forestpoly{G}\rangle_D=\delta_{F,G}$, where $\vope{F}$ is the $D$-pairing adjoint of $\tope{F}$.

There are two main steps that we will carry out.
\begin{enumerate}
    \item We determine the adjoint $\vope{i}$ of individual $\tope{i}$ as an integration operator.
    \item We interpret composites of these $\vope{i}$ applied to $1$ as recursively computing $V_F(\uplambda)$ in terms of $V_{F/i}(\uplambda)$.
\end{enumerate}

For technical reasons we will have to carry out these steps using the $D$-pairing between polynomials rings $\QQ[x_1,x_2,\dots]$ and $\QQ[\lambda_1,\lambda_2,\dots]$ in infinitely many variables, and then return to finitely many variables case by truncating appropriately.

\subsection{The adjoint to trimming under the $D$-pairing}
\label{subsec:adjoint_to_T}
All mentions of adjoints in the sequel are with respect to the $D$-pairing.
If $X\in \End(\QQ[x_1,x_2,\dots])$ then the adjoint $X^{\vee}\in \End(\QQ[\lambda_1,\lambda_2,\ldots])$ might not exist, but if it does then it is unique since $\langle f,g\rangle_D=0$ for all $f$ implies $g=0$.

\begin{defn}
    For $f\in \QQ[\lambda_1,\lambda_2,\ldots]$ we define
    \begin{align*}
        \rope{i}^{\vee}f\coloneqq f(\lambda_1,\ldots,\lambda_{i-1},\lambda_{i+1},\ldots),\qquad
        \vope{i}f\coloneqq \int_{\lambda_{i+1}}^{\lambda_i}f(\lambda_1,\ldots,\lambda_{i-1},z,\lambda_{i+2},\ldots) \,dz.
    \end{align*}
\end{defn}
\begin{prop}\label{prop:adjointness}
   The operators $\rope{i}$ and $\tope{i}$ are adjoint to $\rope{i}^{\vee}$ and $\vope{i}$ respectively. In symbols, for $g\in \QQ[x_1,x_2,\dots]$ and $f\in \QQ[\lambda_1,\lambda_2,\dots]$ we have
    \begin{align*}
    \langle g,\rope{i}^{\vee}f\rangle_D=\langle \rope{i}g,f\rangle_D,\qquad \langle g,\vope{i}f\rangle_D=\langle \tope{i}g,f\rangle_D.
    \end{align*}
    Consequently, for $F\in \indexedforests$ we have a well-defined operator $\vope{F}$ adjoint to $\tope{F}$ defined by
    \begin{align*}
    \vope{F}=
    \vope{i_k}\cdots \vope{i_1} \text{ for any } (i_1,\ldots,i_k)\in \Trim{F}.
    \end{align*}
\end{prop}
\begin{proof}
    The well-definedness of $\vope{F}$ follows by taking the adjoint of the equality $\tope{F}=\tope{i_1}\cdots \tope{i_k}$.
    We verify adjointness by checking it on monomials $f=\uplambda^{\sfc}$ and $g=\sfx^{\sfd}$ for $\sfc,\sfd\in \nvect$.

    For the adjointness of $\rope{i}$ and $\rope{i}^{\vee}$ we have
    $\langle \rope{i}\,\sfx^{\sfd},\uplambda^{\sfc}\rangle_D=\langle \sfx^{\sfd},\rope{i}^{\vee}\uplambda^{\sfc}\rangle_D=0$ if $c_i\ne 0$ and if $c_i=0$ then both are equal to $\sfd!\,\delta_{\sfd,\sfc'}$ where $\sfc'=(c_1,\ldots,c_{i-1},c_{i+1},\ldots)$.
    For the adjointness of $\tope{i}$ and $\vope{i}$ we have on the one hand that $\langle \sfx^{\sfd}, \vope{i}\,\uplambda^{\sfc} \rangle_D$ equals
    \begin{align*}\langle \sfx^{\sfd},\lambda_1^{c_1}\cdots \lambda_{i-1}^{c_{i-1}}\frac{\lambda_{i}^{c_i+1}-\lambda_{i+1}^{c_i+1}}{c_i+1}\lambda_{i+2}^{c_{i+1}}\cdots\rangle_D=\begin{cases}\sfc!&\text{if }\sfd=(c_1,\ldots,c_{i-1},c_i+1,0,c_{i+1},\ldots)\\-\sfc!&\text{if }\sfd=(c_1,\ldots,c_{i-1},0,c_i+1,c_{i+1},\ldots)\\0&\text{otherwise.}\end{cases}\end{align*}
    On the other hand, $\tope{i}(\sfx^{\sfd})$ is always a monomial, and is a multiple of $\sfx^{\sfc}$ exactly when $\sfd=(c_1,\ldots,c_{i-1},c_i+1,0,c_i,c_{i+1},\ldots)$ (in which case it is equal to $\sfx^{\sfc}$) or $\sfd=(c_1,\ldots,c_{i-1},0,c_i+1,c_{i+1},\ldots)$ (in which case it is equal to $-\sfx^{\sfc}$).
\end{proof}

\subsection{Volume polynomials}
\label{subsec:vol_as_harmonics}

The following family of ``forest polytopes'' shall play a crucial role for us. 
\begin{defn}
Let $F\in \indexedforests$ and let $\lambda=(\lambda_1,\lambda_2,\ldots)$ be a sequence with $\lambda_i\ge \lambda_{i+1}$ for all $i$. We define the \emph{forest polytope} $\cube_{F,\lambda}\subset \mathbb{R}^{\internal{F}}$ as the subset of assignments $\phi:\internal{F}\to \RR$ satisfying the following constraints. 
Letting $\phi_\lambda$ be the extension of $\phi$ to $\internal{F}\sqcup \supp(F)$ by setting $\phi_\lambda(i)=\lambda_i$, we have for all $v\in \internal{F}$ the inequalities
\begin{align*}
    \phi_\lambda(v_L)\ge \phi(v)\ge \phi_\lambda(v_R).
\end{align*}
\end{defn}

Figure~\ref{fig:forest_polytope} shows an $F\in \indexedforests$ as well as the inequalities along the left and right edges cutting out the polytope $\cube_{F,\lambda}$. In this case we have $\lambda_2\geq \phi(a)\geq \lambda_3$, $\lambda_4\geq \phi(c)\geq \lambda_5$, $\phi(a)\geq \phi(b)\geq \phi(c)$,
$\lambda_7\geq \phi(d)\geq \lambda_8$,
$\lambda_{11}\geq \phi(e)\geq \lambda_{12}$,
 and $\phi(e)\geq \phi(f)\geq \lambda_{13}$.

\begin{figure}[!ht]
    \centering
    \includegraphics[width=0.7\textwidth]{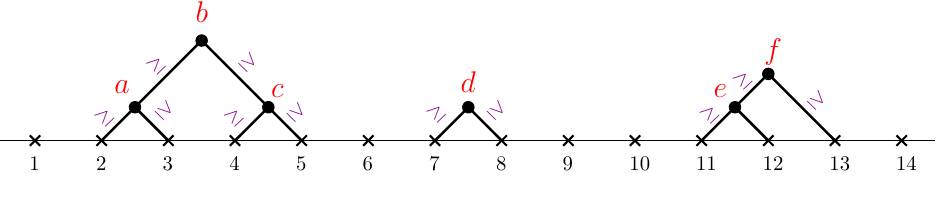}
    \caption{Inequalities defining $\cube_{F,\lambda}$ for the $F$ in Figure~\ref{fig:indexed_forest_eg}.
    \label{fig:forest_polytope}}
\end{figure}

The following lemma casts the inherent recursive structure underlying $F$ in the setting of forest polytopes.
We omit the proof as it is straightforward.
\begin{lem}
\label{lem:recursivevol}
Let $F\in \indexedforests$.
If $i\in \qdes{F}$ then the coordinate projection $\pi_v:\cube_{F,\lambda}\to [\lambda_{i+1},\lambda_i]$ has 
\begin{align*}
\pi_v^{-1}(z)=\cube_{F/i,\lambda'}
\end{align*}
where $\lambda'=(\lambda_1,\ldots,\lambda_{i-1},z,\lambda_{i+2},\ldots)$. In particular,
    \begin{align*}\vol(\cube_{F,\lambda})=\int_{\lambda_{i+1}}^{\lambda_i}\vol(\cube_{F/i,\lambda'})=\vope{i}\vol(\cube_{F/i,\lambda}).
    \end{align*}
\end{lem}

\subsection{Volumes as harmonics}
\label{subsec:vols}

\begin{defn}
    For $F\in \indexedforests$, we define the volume polynomial $V_F(\uplambda)$ associated to $F$ as $\vol(\cube_{F,\lambda})$. 
\end{defn}

The following corollary verifies that this is indeed a polynomial.

\begin{cor}\label{cor:vol_via_vope}
Let $F\in \indexedforests$. Then 
    \begin{align*}
    V_F(\uplambda)=\vope{F}(1),
    \end{align*}
    and for $f\in \poly$ we have $\langle f,V_F(\uplambda)\rangle_D=\ct \tope{F}(f)$.
\end{cor}
\begin{proof}
   Iterating \Cref{lem:recursivevol} and using that $V_\emptyset(\lambda)=1$ shows the first statement. For the second, we note that because $\vope{F}$ is adjoint to $\tope{F}$, we have
   \begin{equation*}
       \langle f,V_F(\uplambda)\rangle_D=\langle f,\vope{F}(1)\rangle_D=\langle \tope{F}f,1\rangle_D=\ct \tope{F}(f). \qedhere
   \end{equation*}
\end{proof}

As an example, for the $F$ in \Cref{fig:forest_polytope} we have $V_F(\uplambda)=\vope{11}\vope{11}\vope{7}\vope{4}\vope{2}\vope{2}(1)$, which equals
    \begin{align*}
        (\frac{1}{2}(\lambda_2^2-\lambda_3^2)(\lambda_4-\lambda_5)-\frac{1}{2}(\lambda_2-\lambda_3)(\lambda_4^2-\lambda_5^2))(\lambda_7-\lambda_8)(\frac{1}{2}(\lambda_{11}^2-\lambda_{12}^2)-(\lambda_{11}-\lambda_{12})\lambda_{13})
    \end{align*}
    The factorization of $V_F(\uplambda)$ is explained because the defining inequalities of $\cube_{F,\lambda}$ imply that we can express it as a product of forest polytopes for $G\in \indexedforests$ in shifted variable sets corresponding to the connected components of $F$.


For $n\in \NN$ consider the \emph{truncation operator} $P_n:\QQ[\lambda_1,\lambda_2,\dots]\to \QQ[\lambda_1,\dots,\lambda_n]$ defined by setting $\lambda_{i}=0$ for all $i>n$. Note that for $f\in \QQ[x_1,\dots,x_n]$ and $g\in \QQ[\lambda_1,\lambda_2,\dots]$ we have
\begin{align*}
    \langle f, g\rangle_D=\langle f, P_n(g)\rangle_D.
\end{align*}
Given a basis of homogeneous polynomials $\{g_i\}_{i\in \NN}$ for $\mathbb{Q}[\lambda_1,\ldots,\lambda_n]$, we say that a collection of homogeneous polynomials $\{h_i\}_{i\in \NN}$ in $\QQ[x_1,\ldots,x_n]$ is graded $D$-dual if $\langle h_i,g_j\rangle_D=\delta_{i,j}$. Since $\langle,\rangle_D$ is a perfect pairing when restricted to homogeneous polynomials of degree $d$ in $\mathbb{Q}[x_1,\ldots,x_n]$, the graded $D$-dual set of polynomials always exists, is unique, and is a basis for $\mathbb{Q}[x_1,\ldots,x_n]$.

Our next result, which is also a straightforward consequence of \Cref{prop:adjointness} and \Cref{cor:vol_via_vope}, shows that these volume polynomials $V_F(\uplambda)$ for $F\in \indexedforests$ are graded duals to forest polynomials. The reader should compare this result with~\cite[Corollary 12.3(2)]{PS09}.\footnote{Note that \cite[Corollary 12.3(1)]{PS09} is incorrect and the issue is highlighted in the footnote to \cite[Theorem 1.1]{hamaker2023dual}.}
\begin{thm}\label{thm:vol_as_duals_to_forests}
For all $F,G\in \indexedforests$ we have $\langle \oneforestpoly{G},V_F(\uplambda)\rangle_D=\delta_{F,G}$. 
Furthermore, the family of projected volume polynomials $\{P_nV_F(\uplambda)\}_{F\in\ltfor{n}}$ in $\QQ[\lambda_1,\ldots,\lambda_n]$ is the graded $D$-dual basis to the homogeneous basis $\{\oneforestpoly{F}\}_{F\in\ltfor{n}}$ of $\QQ[x_1,\ldots,x_n]$.
\end{thm}
\begin{proof}
By \Cref{cor:vol_via_vope} and \Cref{cor:TFG} we have \begin{align*}\langle \oneforestpoly{G},V_F(\uplambda)\rangle_D=\ct \tope{F}\oneforestpoly{G}=\delta_{F,G}.\end{align*}
    For the second part, we have by \Cref{prop:polynbasis} that $\{ \oneforestpoly{G}:G\in\ltfor{n}\}$ is a homogeneous basis for $\mathbb{Q}[x_1,\ldots,x_n]$, and $P_nV_F(\uplambda)$ are homogeneous polynomials in $\mathbb{Q}[\lambda_1,\ldots,\lambda_n]$ which satisfy $\langle \oneforestpoly{G}, P_nV_F(\lambda))\rangle_D=\langle \oneforestpoly{G}, V_F(\uplambda)\rangle_D=\delta_{F,G}$.
\end{proof}

We are ready to determine a basis for $\hqsym{n}$ in terms of volume polynomials.

\begin{thm}\label{thm:basis_harmonics}
A $\QQ$-basis for $\hqsym{n}$ is given by
    \begin{align*}
    \{V_F(\uplambda)\suchthat F\in \suppfor{n}\}.
    \end{align*}
\end{thm}
\begin{proof}
Recall by \Cref{prop:polynbasis} that $\poly_n$ has a homogeneous basis $\{\oneforestpoly{F}: F\in\ltfor{n}\}$, and by \Cref{thm:qsymidebasis} $\qsymide{n}$ has a homogeneous basis the subset $\{\oneforestpoly{F}: F\in\ltfor{n}\setminus \suppfor{n}\}$. 
As $\hqsym{n}$ is the graded $D$-orthogonal complement to $\qsymide{n}$ in $\poly_n$  and $\{P_nV_F(\uplambda)\}$ is the graded $D$-dual basis to $\{\oneforestpoly{F}: F\in\ltfor{n}\}$, we conclude that a $\mathbb{Q}$-basis for $\hqsym{n}$ is given by $\{P_nV_F(\uplambda)\suchthat F\in \suppfor{n}\}$. 
It remains to notice that $V_F(\uplambda)\in \mathbb{Q}[\lambda_1,\dots,\lambda_n]$ for $F\in \suppfor{n}$, so $P_nV_F(\uplambda)=V_F(\uplambda)$.
\end{proof}

\subsection{A conjecture of Aval--Bergeron--Li}
We now proceed to establish a generalization of a conjecture of Aval--Bergeron--Li \cite{ABL10} that posited the existence of a family of $\cat{n-1}$-many polynomials of degree $n-1$ the span of whose derivatives gave $\hqsym{n}$. We already know that the degree $n-1$ component of $\hqsym{n}$ is the top degree component and this has a basis given by the polynomials $V_F(\uplambda)$ with $F\in \suppfor{n}$ and $|F|=n-1$; there are $\cat{n-1}$ many such polynomials. We will now show that the derivatives of this top degree component of $\hqsym{n}$ span $\hqsym{n}$.

It turns out that the following proposition will formally imply the desired spanning.

\begin{prop}\label{prop:annihilator}
Let $f\in \poly_n$ be homogeneous of degree $d<n-1$, and assume that $x_1f\in \qsymide{n}$. Then we have $f\in \qsymide{n}$.
\end{prop}
\begin{proof}
  We induct on $d$. If $d=0$ then $f$ is constant. The inequality for $d$ implies $n\ge 2$ and $x_1=\forestpoly{\underline{1}}\not\in \qsymide{n}$ by Theorem~\ref{thm:qsymidebasis}\ref{8.8:it2} since $\underline{1}\in \suppfor{n}$. Thus we must have $f=0$.

   Assume now $d>0$, and write $f=\sum_{F} a_F\oneforestpoly{F}$ with $F\in\ltfor{n}$, $|F|=d$ following \Cref{prop:polynbasis}. By Theorem~\ref{thm:qsymidebasis}\ref{8.8:it2} we can assume that    $f=\sum_{F} a_F\oneforestpoly{F}$ with $F\in\suppfor{n}$, and we now want to show that $f$ is zero. Fix any $2\le i \le n-1$, so that $\tope{i}(x_1)=0$ and $\rope{i+1}(x_1)=1$. By \Cref{lem:leibniz} and \Cref{cor:topedescend} we have \begin{equation*}
      \tope{i}(x_1f)=x_1\sum_{\substack{F\in \suppfor{n}\\ i \in\qdes{F}}}a_F\,\oneforestpoly{F/i}\in \qsymide{n-1}.
   \end{equation*}
   By induction, for this to happen the sum must vanish in $\qscoinv{n-1}$. But the $F/i$ are distinct forests in $\suppfor{n-1}$, so by \Cref{thm:qsymidebasis} this implies that $a_F=0$ for any $F$ such that $i\in \qdes{F}$.
    
    There remains the case where $F$  satisfies $\qdes{F}= \{1\}$. There is a unique such $F\in\suppfor{n}$, namely $F=\underline{1}^{d}$, and $\oneforestpoly{\underline{1}^{d}}=x_1^{d}$. But then $x_1\forestpoly{\underline{1}^{d}}=x_1^{d+1}=\forestpoly{\underline{1}^{d+1}}$ and $\underline{1}^{d+1}\in \suppfor{n}$ as $d+1\le n-1$, so does not lie in $\qsymide{n}$ by Theorem~\ref{thm:qsymidebasis}\ref{8.8:it1}.
\end{proof}

\begin{lem}\label{lem:elementary_lem}
    Let $g_1,\ldots,g_r,h\in \QQ[\lambda_1,\dots,\lambda_n]$ be homogeneous polynomials with $\deg(g_i)=k$ for $1\leq i\leq r$ and $\deg(h)=d\le k$.
    Assume that for any homogeneous polynomial $f\in \QQ[x_1,\dots,x_n]$ of degree $d$ such that \begin{align*}
    f(\deri{1},\dots,\deri{n})g_1=\cdots = f(\deri{1},\dots,\deri{n})g_r=0
    \end{align*} 
    we have $\langle f,h\rangle_D=0$. Then $h$ lies in the span $W$ of $\{\deri{1}^{c_1}\cdots \deri{n}^{c_n}g_i:c_1+\cdots+c_n=k-d,1\le i \le r\}$.
\end{lem}
\begin{proof}
First, we note that for any homogeneous polynomials $g$ of degree $k$ and $f$ of degree $d$ 
\begin{align*}
    f(\deri{1},\dots,\deri{n})g=0 \Longleftrightarrow \langle f(x_1,\ldots,x_n),\deri{1}^{c_1}\cdots \deri{n}^{c_n}g\rangle_D=0 \text{ whenever }\sum c_i=k-d.
\end{align*}
Indeed, this follows from the identity
\begin{align*}\langle f,\deri{1}^{c_1}\cdots \deri{n}^{c_n}g\rangle_D=&\langle \sfx^{\sfc},f(\deri{1},\ldots,\deri{n})g\rangle_D=\sfc!\,[\uplambda^{\sfc}](f(\deri{1},\ldots,\deri{n})g).\end{align*}

Applying this equivalence to each $g_i$, we have reduced to showing that if for all homogeneous degree $d$ polynomials $f\in \mathbb{Q}[x_1,\ldots,x_n]$ we have $\langle f,W\rangle_D=0\implies \langle f,h\rangle_D=0$, then $h\in W$. But this follows from the fact that the $D$-pairing on homogeneous degree $d$ polynomials is perfect.
\end{proof}

\begin{thm}
    $\hqsym{n}$ is spanned by the derivatives of the homogeneous degree $n-1$ elements of $\hqsym{n}$.
\end{thm}
\begin{proof}
Let $h\in \hqsym{n}$ be of degree $d\le n-1$. By \Cref{lem:elementary_lem}, it suffices to show that for all homogeneous $f\in \QQ[x_1,\ldots,x_n]$ of degree $d$ such that $\langle f,h\rangle_D\ne 0$, there exists $g\in \hqsym{n}$ of degree $n-1$ such that $f(\deri{1},\ldots,\deri{n})g\ne 0$.

Fix such an $f$. 
Since $h\in \hqsym{n}$ and $\langle f,h\rangle_D\ne 0$, we have $f\not\in \qsymide{n}$. 
By ~\Cref{prop:annihilator} this implies $x_1^{n-1-d}f\not \in \qsymide{n}$. 
Thus there exists $g\in \hqsym{n}$ homogeneous of degree $n-1$ such that $\langle x_1^{n-1-d}f, g\rangle \neq 0$,  and thus $f(\deri{1},\dots,\deri{n})g \ne 0$.
\end{proof}

\subsection{Volume polynomials into monomials and monomials into forests}
\label{subsec:volumes_explicit}

We now describe the explicit expansion for $V_F(\uplambda)$ for $F\in \indexedforests$ in the basis of normalized monomials $\frac{\uplambda^{\sfc}}{\sfc!}$, and the expansions of monomials $\sfx^{\sfc}$ into the basis of forest polynomials.  

Let $\paths{F}$ denote the set of functions $\mathcal{P}:\internal{F}\to \{L,R\}$. 
By taking the union of edges $\bigcup_{v\in \internal{F}}\{v,v_{\mathcal{P}(v)}\}$ where $\{x,y\}$ denotes the edge joining $x$ and $y$, we can encode $\mathcal{P}\in \paths{{F}}$ as a collection of vertex disjoint paths travelling up from the leaves of $\widehat{F}$ which cover every node in $\internal{F}$. 
For each $\mathcal{P}$, we let $d(\mathcal{P})\coloneqq (d_i)_{i\in \NN}\in \nvect$ where $d_i$ records the length of the path that has one endpoint at leaf $i$. 
It is easy to see that $d$ is injective, and for $\mathcal{P}$ the constant $L$-function we have $d(\mathcal{P})=\sfc(F)$.

For example, Figure~\ref{fig:f_vol} shows an $F\in \indexedforests$ with the corresponding ${F}$ obtained by omitting the dotted edges. 
If we take the collection $\mathcal{P}$ of paths determined by the edges highlighted in blue, then we get $d(\mathcal{P})=(0,0,1,2,0,0,1,0,0,0,0,2,0,\dots)$. 

Given $\sfc\in \nvect$ we define $\epsilon_F(\sfc)$ as follows:
\begin{align*}   \epsilon_F(\sfc)=\left \lbrace \begin{array}{ll}  (-1)^{|\mathcal{P}^{-1}(R)|} & \text{if there exists $\mathcal{P}\in \paths{\Bin{F}}$ such that $d(\mathcal{P})=\sfc$} \\0 & \text{otherwise. } \end{array}\right.
\end{align*}

\begin{figure}[!ht]
    \centering
    \includegraphics[scale=0.75]{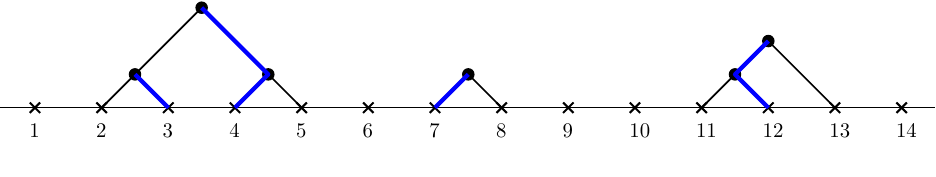}
    \caption{A forest $F\in \indexedforests$  with a path $\mathcal{P}\in \paths{\widehat{F}}$ colored in blue}
    \label{fig:f_vol}
\end{figure}

With this notation in hand we have

\begin{prop}\label{prop:explicit_volume}
    For $F\in \indexedforests$ we have $$V_F(\uplambda)=\sum_{\sfc\in \nvect}\epsilon_F(\sfc) \,\frac{\uplambda^{\sfc}}{c!}.$$
\end{prop}
\begin{proof}
We proceed by induction on $|F|$.
If $F=\emptyset$ then $V_F(\uplambda)=1$ so the formula is true, and we may now assume $|F|\ge 1$.
By \Cref{lem:recursivevol} we have
\begin{align}\label{eq:vol_initial}
V_F(\uplambda)=\vope{i}\,V_{F/i}(\uplambda)=\sum_{\sfd\in \nvect}\epsilon_{F/i}(\sfd)\,\vope{i}\,\frac{\lambda^{\sfd}}{\sfd!}.\end{align}
Given $d\in \nvect$ define compositions
\begin{align*}
    \operatorname{left}(\sfd)&=(d_1,\dots,d_{i-1},d_{i}+1,0,d_{i+1},\dots)\\
    \operatorname{right}(\sfd)&=(d_1,\dots,d_{i-1},0,d_{i}+1,d_{i+1},\dots).
\end{align*}
Then the last term of~\eqref{eq:vol_initial} can be rewritten as
\begin{equation}\label{eq:vol_final}
    V_F(\uplambda)=\sum_{\sfd\in \nvect}\epsilon_{F/i}(\sfd) \,\left(\frac{\uplambda^{\operatorname{left}(\sfd)}}{(\operatorname{left}(\sfd))!}-\frac{\uplambda^{\operatorname{right}(\sfd)}}{(\operatorname{right}(\sfd))!} \right),
\end{equation}
It is then straightforward to check that the first summand (resp. second summand) on the right-hand side of~\eqref{eq:vol_final} tracks the contribution of those paths $\mathcal{P}\in \paths{F}$ using the left (resp. right) leaf of the newly created internal node in $F$.
\end{proof}


As an application we obtain the forest expansion of monomials.
\begin{prop}\label{prop:monomials_to_forests}
   For  $\sfc\in \nvect$ we have \[\sfx^{\sfc}=\sum_{G\in \indexedforests} \epsilon_G(c)\,\forestpoly{G}.\]
\end{prop}
\begin{proof}
    We have the sequence of equalities
    \begin{align*}\ct\tope{G}\,\sfx^{\sfc}=\langle \tope{G}\,\sfx^{\sfc},1\rangle_D=\langle \sfx^{\sfc},\vope{G}(1)\rangle_D=\langle \sfx^{\sfc},V_G(\uplambda)\rangle_D=\epsilon_G(\sfc).\end{align*}
    by \Cref{prop:adjointness}, \Cref{cor:vol_via_vope}, and \Cref{prop:explicit_volume} applied in succession.
\end{proof}

Figure~\ref{fig:kostka_Eg} shows the three indexed forests that contribute to the expansion of  $x_2^2x_3$ as per \Cref{prop:monomials_to_forests}. The leftmost tree has code precisely the exponent vector of this monomial. Explicitly we have $x_2^2x_3=\oneforestpoly{F}-\oneforestpoly{G}-\oneforestpoly{H}$.
\begin{figure}[!ht]
    \centering
    \includegraphics[scale=0.6]{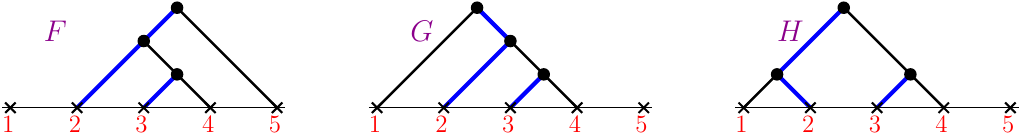}
    \caption{Three indexed forests that contribute to the monomial $x_2^2x_3$}
    \label{fig:kostka_Eg}
\end{figure}
\begin{thm}
For $F\in \indexedforests$ we have
    $V_F(\uplambda)\in \QQ[\lambda_1-\lambda_{2},\lambda_2-\lambda_{3},\ldots]$.
    The coefficients in 
    \begin{align*}V_F(\uplambda)=\sum_{\sfc =(c_1,c_2,\dots) \in \nvect} b_\sfc\,\prod_{i\ge 1}(\lambda_i-\lambda_{i+1})^{c_i}\end{align*}
    satisfy $b_\sfc=\frac{1}{\sfc!}\tope{F}\left(\prod_{i\ge 1}(\forestpoly{i})^{c_i}\right)\ge 0$.
\end{thm}
\begin{proof}
    The fact that $V_F(\uplambda)\in \QQ[\lambda_1-\lambda_{2},\lambda_2-\lambda_{3},\ldots]$ can be verified inductively by checking that $\vope{i}$ preserves this ring and noting $V_F(\uplambda)=\vope{F}(1)$ by \Cref{cor:vol_via_vope}. 
    Noting that $\oneforestpoly{i}=x_i+x_{i-1}+\cdots +x_1$, it is straightforward to check that
    \begin{equation}
   \left \langle \prod_{i\ge 1} (\forestpoly{i})^{c_i},\,\prod_{i\ge 1} (\lambda_i-\lambda_{i+1})^{d_i}\right\rangle_D=\delta_{\sfc, \sfd}\,\sfc!,\end{equation}
    and  so
    \begin{align*}
    b_\sfc=\frac{1}{\sfc!}\left\langle \prod_{i\ge 1}(\forestpoly{i})^{c_i},V_F(\uplambda)\right\rangle_D
    =\frac{1}{\sfc!}\ct\tope{F}\left(\prod_{i\ge 1}(\forestpoly{i})^{c_i}\right).
    \end{align*}
    Finally, $b_\sfc\ge 0$ since $\prod_{i\ge 1}(\forestpoly{i})^{c_i}$ is forest positive by \Cref{thm:forestmultpos} or \Cref{thm:monks}.
\end{proof}
   We note that \cite[\S 6.3]{NT_forest} may be formulated as giving a combinatorial interpretation for the coefficients $\tope[1]{F}(x_1^{c_1}(x_1+x_2)^{c_2}\cdots)=\tope[1]{F}((\oneforestpoly{1})^{c_1}(\oneforestpoly{2})^{c_2}\cdots)$.

\section{Endomorphisms of polynomials $\rope{1}$-commuting with quasisymmetrics}
\label{sec:quasisymmetric_nil_hecke}
Recall (cf. \cite{Man01}) that the ring $\End_{\sym{n}}(\poly_n)$ is generated by the operations of (multiplication by) $x_i$ and $\partial_i$, and in fact
\begin{align*}\End_{\sym{n}}(\poly_n)=\bigoplus_{w\in S_n}\poly_n\partial_w.\end{align*}
Taking the limit of these algebras we obtain the subalgebra of $\End(\poly)$ generated by all $x_i$ and $\partial_i$, which decomposes as $\bigoplus_{w\in S_{\infty}}\poly_n\partial_w$. This may be informally thought of as those endomorphisms of $\poly$ which modify only finitely many coordinates and commute, in an appropriate sense, with symmetric power series.

For quasisymmetrics we have in stark contrast the following observation.

\begin{obs}
    $\End_{\qsym{n}}(\poly_n)=\poly_n$.
\end{obs}
\begin{proof}
If $\Phi\in \End_{\qsym{n}}(\poly_n)$ then because $(x_1\cdots x_n)f$ is quasisymmetric for all $f\in \poly_n$ we have \begin{align*}x_1\cdots x_n\Phi(f)=\Phi((x_1\cdots x_n)f)=x_1\cdots x_nf\Phi(1),\end{align*} implying that $\Phi(f)=f\Phi(1)$.
\end{proof}

To find the correct analogue we have to consider $\Hom_{\qsym{n}}(\poly_n,\poly_{n-k})$ where $\poly_{n-k}$ is considered as a $\qsym{n}$-module by the map $\rope{1}^k|_{\qsym{n}}:\qsym{n}\to \qsym{n-k}$. 
This is well-defined directly from the definition of $\qsym{n}$. We note that $\rope{1}^k|_{\qsym{n}}=\rope{n-k+1}^k|_{\qsym{n}}$ by \Cref{thm:Rqsymchar}: it is the map setting $x_n=\cdots = x_{n-k+1}=0$.
\begin{thm}
\label{thm:homfrompolytopolyn-k}
We have
    \begin{align*}\Hom_{\qsym{n}}(\poly_n,\poly_{n-k})=\bigoplus_{H\in \suppfor{n}\text{ with }|H|\le k} \poly_{n-k}\rope{1}^{k-|H|}\,\tope{H}.\end{align*}
\end{thm}

\begin{proof}
For clarity we let $\mathbb{H}_{n,k}\coloneqq \Hom_{\qsym{n}}(\poly_n,\poly_{n-k})$.
    First, we show that $\rope{1}^{k-|H|}\tope{H}\in \mathbb{H}_{n,k}$. 
    By \Cref{prop:fulltrim} we have $\tope{H}\in \mathbb{H}_{n,|H|}$ for any $H\in \suppfor{n}$. 
    Next, $\rope{1}^{k-|H|}$ lies in $
    \mathbb{H}_{n-|H|,k-|H|}$ because for $f\in \qsym{n}$ and $g\in \poly_n$ we have $\rope{1}^{k-|H|}(fg)=\rope{1}^{k-|H|}(f)\rope{1}^{k-|H|}(g)$.

    We now construct functions $\{\Psi_{H}:H\in \suppfor{n}\text{ and }|H|\le k\}\subset \mathbb{H}_{n,k}$  such that $\Psi_{H}(\oneforestpoly{H'})=\delta_{H,H'}$ for all $H'\in \suppfor{n}$  with $|H'|>|H|$, and
    \begin{align*}\Psi_{H}=\rope{1}^{k-|H|}\,\tope{H}-\sum_{|H'|>|H|,\text{ }H'\in \suppfor{n}} b_{H,H'}(x_1,\ldots,x_{n-k})\,\Psi_{H',k}.\end{align*}
    We do this by backwards induction on $|H|$. For $|H|=k $ we take $\Psi_{H}=\rope{1}^{k-|H|}\,\tope{H}$. Otherwise, $\rope{1}^{k-|H|}\tope{H}(\oneforestpoly{H})=1$ and  $\rope{1}^{k-|H|}\tope{H}(\oneforestpoly{H'})=0$ when $H\ne H'\in \suppfor{n}$ and $|H'|\le |H|$, so we can take $b_{H,H'}=\rope{1}^{k-|H|}\tope{H}\oneforestpoly{H'}$.

    As the $\poly_{n-k}$-linear transformation expressing $\{\rope{1}^{k-|H|}\tope{H}:H\in \suppfor{n}\text{ and }|H|\le k\}$ in terms of $\{\Psi_{H}:H\in \suppfor{n}\text{ and }|H|\le k\}$ is  invertible by upper-triangularity, it suffices to show that $\{\Psi_{H}:H\in \suppfor{n}\text{ and }|H|\le k\}$ is a $\poly_{n-k}$-basis for $\mathbb{H}_{n,k}$.
    The $\Psi_H$ are $\poly_{n-k}$-linearly independent: if $\sum f_H(x_1,\ldots,x_{n-k})\Psi_H=0$ then for all $H'$ applying the left hand side to $\oneforestpoly{H'}$ shows that $f_{H'}=0$. 

    It remains to show that the $\Psi_H$ span. Let $\Phi\in \mathbb{H}_{n,k}$. Define 
    \begin{equation*}
    \Phi'=\Phi-\sum_{H\in \suppfor{n}} \Phi(H)\,\Psi_{H}. \end{equation*}
    We want to show $\Phi'=0$. 
    Already $\Phi'(\oneforestpoly{H})=0$ for all $H\in\suppfor{n}$ with $|H|\le k$ by the properties of the endomorphisms $\Psi_{H'}$. It remains to show $\Phi'(\oneforestpoly{H})=0$ for those $H\in \suppfor{n}$ with $|H|>k$: this is enough since the $\oneforestpoly{H}$ for $H\in \suppfor{n}$  generate $\poly_n$ as an $\qsym{n}$-module by \Cref{thm:qsymidebasis}.

    We induct on $|H|$. 
    We assume that $|H|>k $. 
    Because $H\in \suppfor{n}$ we know that $|H|\le n$ so we may assume that $k<n$. 
    Choose any $G\in \zigzag{n}$ with $\min \supp G=k+1$: these always exist as is readily checked. 
    Since $\min \supp G \le |H|$ we have that $G\star H$ does not exist and thus 
 $\oneforestpoly{G}\oneforestpoly{H}\in \mathcal{I}^{\star}_{|G|+1,n}$ by \Cref{lem:GHproduct}. By  \Cref{thm:qsymidebasis} we know that $\mathcal{I}^{\star}_{|G|+1,n}=\mathcal{I}_{|G|+1,n}$ so we may write \begin{equation*}
 \oneforestpoly{G}\oneforestpoly{H}=\sum g_{i}(x_1,\ldots,x_n)h_i(x_1,\ldots,x_n)\end{equation*}
    with $g_{i}\in \qsym{n}$ with $\deg g_i\ge |G|+1$ and therefore with $\deg h_i =|G|+|H|-\deg g_i<|H|$.

    Since $\{\oneforestpoly{H}:H\in \suppfor{n}\}$ is a $\mathbb{Z}$-basis of  $\qscoinv{n}$ by \Cref{thm:qsymidebasis} we may further assume that each $h_i=\oneforestpoly{H_i'}$ for some $H_i'\in \suppfor{n}$, and $|H_i'|=\deg \oneforestpoly{H_i'}<|H|$. We therefore have
    \begin{align*}
    (\rope{1}^{k}\,\oneforestpoly{G})\,\Phi'(\oneforestpoly{H})=\Phi'(\oneforestpoly{G}\,\oneforestpoly{H})=\sum \Phi'( g_i\,\oneforestpoly{H'_i})=\sum \rope{1}^{k}(g_i)\,\Phi'(\oneforestpoly{H'_i})=0.
    \end{align*}
    Since $\min\supp G=k+1$ we have $\rope{1}^k\,\oneforestpoly{G}\ne 0$ by \Cref{prop:R1forest}.
    So $\Phi'(\oneforestpoly{H})=0$ as desired.
\end{proof}


\begin{rem}
\label{rem:quasinil}
The limiting object 
\begin{align*}\bigoplus_k\lim_{n\to \infty}\Hom_{\qsym{n}}(\poly_{n},\poly_{n-k})=\bigoplus_{F\in \indexedforests}\poly \rope{1}^a\,\tope{F}\end{align*}
is the subalgebra of $\End(\poly)$ generated by all $x_i$, $\rope{1}$, and $\tope{i}$. This may be informally thought of as those endomorphisms of $\poly$ which act on all but finitely many coordinates as $x_i\mapsto x_{i-k}$ and commute (in an appropriate sense) with quasisymmetric power series.
\end{rem}


\subsection{Quasisymmetric nil-Hecke Algebra}
\label{subsec:quasi_nil}

The nil-Hecke algebra is the noncommutative algebra with generators denoted $x_1,x_2,\ldots$ and $\partial_1,\partial_2,\ldots$, modulo the relations
\begin{itemize}
 \item (Comm.) $x_ix_j=x_jx_i$ for all $i,j$, $\partial_i\partial_j=\partial_j\partial_i$ for $|i-j|\ge 2$, and $x_i\partial_j=\partial_jx_i$ for $j\not\in \{i-1,i\}$.
\item (Braid) $\partial_i\partial_{i+1}\partial_i=\partial_{i+1}\partial_i\partial_{i+1}$
\item (Nil-Hecke) $\partial_i^2=0$.
\item (Leibniz) $\partial_i x_i=x_{i+1}\partial_i+\idem$ and $\partial_ix_{i+1}=x_i\partial_i-\idem$.
\end{itemize}
Using these relations it is easy to straighten any combination of $x_i$ and $\partial_i$ into a $\poly$-linear combination of operators $\partial_w$ for $w\in S_{\infty}$, and this can be used to show that the nil-Hecke algebra is isomorphic to $\End_{\sym{n}}(\poly_n)$.

This also affords a ``diagrammatic presentation'', encoded by the additional relations needed to specify the presentation beyond the formal commutation relations coming from the fact that for $Z\in \{\partial,x\}$ ($x:\poly_1\to \poly_1$ representing the ``multiplication by $x$'' map), we have $Z_i=\idem^{\otimes i-1}\otimes Z\otimes \idem^{\otimes \infty}:\poly\to \poly$, where we view $\poly=\poly_1^{\otimes \infty}$. This is  given by
\begin{itemize}
    \item (Braid) $(\partial\otimes \idem)(\idem\otimes \partial)(\partial\otimes \idem)=(\idem\otimes \partial)(\partial\otimes \idem)(\idem\otimes \partial)$
    \item (nil-Hecke) $\partial^2=0$
    \item (Leibniz) $\partial (x\otimes \idem)=(\idem \otimes x)\partial+\idem^{\otimes 2}$ and $\partial(\idem \otimes x)=(x\otimes \idem)\partial-\idem^{\otimes 2}$.
\end{itemize}
If we represent $x$ and $\partial$ as in Figure~\ref{fig:xideli} and represent $F\circ G$ by stacking the diagram for $F$ on top of the diagram for $G$, the relations above can be depicted as in Figure~\ref{fig:relations_sym}.

\begin{figure}[!ht]
    \centering    \includegraphics[scale=0.8]{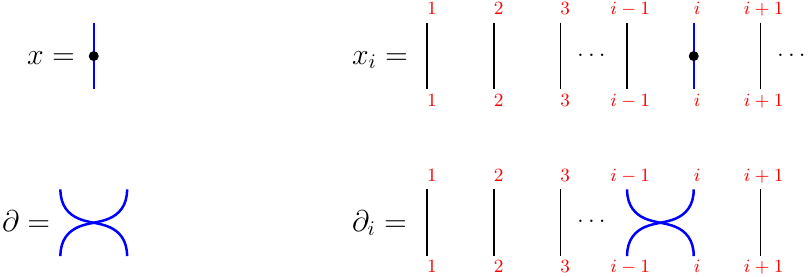}
    \caption{Diagram generators for the nil-Hecke algebra}
    \label{fig:xideli}
\end{figure}

\begin{figure}[!ht]
    \centering    \includegraphics[scale=0.8]{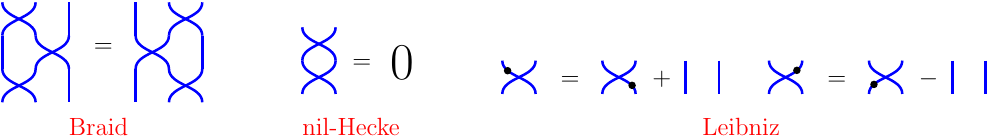}
    \caption{Diagram relations for the nil-Hecke algebra}
    \label{fig:relations_sym}
\end{figure}

As noted in \Cref{rem:quasinil} the algebra in $\End(\poly)$ generated by $\rope{1}$, $\tope{i}$ and $x_i$ may be thought of as the quasisymmetric analogue of the nil-Hecke algebra.
\begin{thm}
\label{thm:quasisymnilhecke}
    The algebra in $\End(\poly)$ generated by $\rope{1}$, $\tope{i}$ and $x_i$ has relations generated by
    \begin{enumerate}[label=(\roman*)]
        \item\label{comm_relations} (Comm.) $\tope{i}\rope{1}=\rope{1}\tope{i+1}$, for $i\ge 1$, $\rope{1}x_i=x_{i-1}\rope{1}$ for $i>1$, $x_ix_j=x_jx_i$ for all $i,j$\\$\tope{i}x_j=x_{j}\tope{i}$ if $j<i$ and $\tope{i}x_j=x_{j-1}\tope{i}$ if $j>i+1$\\
        $\tope{i}\tope{j}=\tope{j}\tope{i+1}$ for $i>j$,
        \item \label{vanishing_relations} $\rope{1}x_1=0$,
        \item \label{last_of_these} $\tope{i}x_i=\rope{1}+x_1\tope{1}+\cdots + x_i\tope{i}$ and $\tope{i}x_{i+1}=-(\rope{1}+x_1\tope{1}+\cdots + x_{i-1}\tope{i-1})$
    \end{enumerate}
\end{thm}
\begin{proof}
    All of these relations are easy to verify directly. 
    For~\ref{last_of_these}, we note by \Cref{lem:leibniz} that $\tope{i}(x_if)=\rope{i+1}f$ and $\tope{i}(x_{i+1}f)=-\rope{i}f$, and then the expressions are obtained by telescoping the identity $x_j\tope{j}=\rope{j+1}-\rope{j}$.

    Using these relations one can straighten any composition of $\rope{1},\tope{i}$ and $x_i$ into a $\poly$-linear combination of $\rope{1}^a\tope{F}$. It follows from \Cref{thm:homfrompolytopolyn-k} that there are no further relations.
\end{proof}

As the proof shows, the relations in \ref{thm:quasisymnilhecke}\ref{last_of_these}  could be simplified if we included redundant generators $\rope{i}$ in the presentation.

The quasisymmetric nil-Hecke algebra also admits a diagrammatic presentation. Note that for $Z\in \{\tope{},\rope{}\}$ we again have $Z_i=\idem^{\otimes i-1}\otimes Z\otimes \idem^{\otimes \infty}$ where $\tope{}:\poly_2\to \poly_1$ is the operator introduced in \Cref{eqn:topedef} and $\rope{}:\poly_1\to \poly_0$ is the operator $\rope{}(f)=f(0)$.

\begin{figure}[!ht]
    \centering
    \includegraphics[scale=0.8]{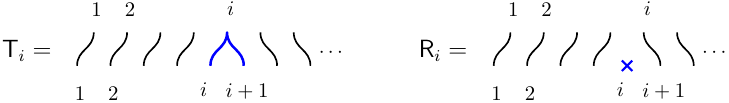}
    \caption{Diagram generators for the quasisymmetric nil-Hecke algebra} 
    \label{fig:TRx_diagrams}
\end{figure}
\begin{cor}
    The diagrammatic presentation of the quasisymmetric nil-Hecke algebra is given by the following.
    \begin{enumerate}[label=(\roman*)]
        \item \label{first_diagrammatic} $\rope{}x=0$
        \item \label{second_diagrammatic} $\tope{}(x\otimes \idem)=\idem\otimes \rope{}$ and $\tope{}(\idem \otimes x)=-\rope{}\otimes \idem$
        \item \label{third_diagrammatic} $x\tope{}=\idem\otimes \rope{}-\rope{}\otimes \idem$
    \end{enumerate}
    
    \begin{figure}[!ht]
    \centering
    \includegraphics[scale=0.8]{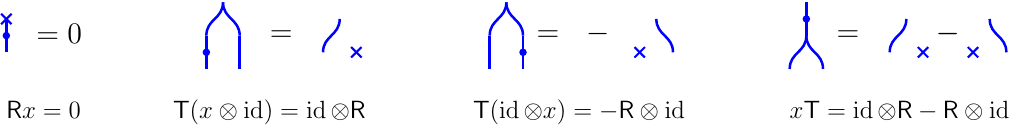}
    \caption{Diagram relations for the quasisymmetric nil-Hecke algebra}
    \label{fig:TRx_diagram_relations}
\end{figure}
\end{cor}
\begin{proof}
   First, it is immediate to verify that ~\ref{first_diagrammatic}--\ref{third_diagrammatic} are satisfied in the quasisymmetric nil-Hecke algebra. Conversely, the relations in~\Cref{thm:quasisymnilhecke}\ref{comm_relations} are all ``trivially'' satisfied since the operators act on distinct sets of variables. Also, ~\ref{first_diagrammatic} implies $\rope{1}x_1=0$ which is~\Cref{thm:quasisymnilhecke}\ref{vanishing_relations}.
    It remains to show the relations in~\Cref{thm:quasisymnilhecke}\ref{last_of_these}. 
    Using~\ref{second_diagrammatic} these amount to showing the relations $\rope{i+1}=\rope{1}+x_1\tope{1}+\cdots + x_i\tope{i}$ and $\rope{i}=\rope{1}+x_1\tope{1}+\cdots+x_{i-1}\tope{i-1}$. 
    But~\ref{third_diagrammatic} implies  $x_j\tope{j}=\rope{j+1}-\rope{j}$ so both equalities now follow.
\end{proof}
\begin{eg}
The relation $\tope{i}\rope{i+1}=\rope{i+1}\tope{i}+\rope{i}\tope{i+1}$, written as $\tope{}(\idem\otimes \rope{}\otimes \idem)=\tope{}\otimes \rope{}+\rope{}\otimes \tope{}$, follows from the chain of equalities
\begin{align*}\tope{}(\idem\otimes \rope{}\otimes \idem)=&\tope{}(x\tope{}\otimes \idem)+\tope{}(\rope{}\otimes \idem \otimes \idem)\nonumber\\=&\tope{}(x\otimes \idem)(\tope{}\otimes \idem)+\rope{}\otimes \tope{}\nonumber\\
=&(\idem\otimes \rope{})(\tope{}\otimes \idem)+\rope{}\otimes \tope{}
=\tope{}\otimes \rope{}+\rope{}\otimes \tope{}
\end{align*}
where in the second equality we used the commutation relation $\tope{}(\rope{}\otimes \idem\otimes \idem)=\rope{}\otimes \tope{}$.
\end{eg}

\appendix

\section{$m$-Quasisymmetric polynomials}
\label{sec:mQuasi}

In this appendix, we will see that essentially all results of the main body of this work have an extension to $m>1$. The exposition is intentionally terse and  proofs are omitted; details and complete proofs can be found in the arXiv version v2 \cite{NST_1} which was written for general $m$.

Given an integer $m\geq 1$, we consider the more general context of \textit{$m$-quasisymmetric polynomials}. Classically, these are defined as certain polynomials in  $\mathbb{Z}[\{z_1^{(j)},z_2^{(j)},\ldots\}_{1\le j \le m}]$ where $z_1^{(j)},z_2^{(j)},\ldots$ are considered the $j$'th colored variables \cite{Av0507, AvCh1618,BauHol08,Poirier98}. Most of these works are in the setting of formal power series instead of polynomials, but we can pass to the finite variable setting by truncating the variable sets. 
The $m$-quasisymmetric polynomials are usually defined as the linear span of a basis of ``fundamental'' $m$-quasisymmetric polynomials \cite[\S 3.2]{HsiaoPetersen10}. 

We adopt a slightly different perspective on $\qsym[m]{n}$ which we have not seen in the existing literature despite its naturality. By arranging the variables in order \begin{align*}z_1^{(1)},z_1^{(2)},\ldots,z_1^{(m)},z_2^{(1)},z_2^{(2)},\dots\end{align*} and relabeling them $x_1,x_2,\ldots,x_m,x_{m+1},x_{m+2},\dots$ we obtain the following description.

\begin{defn} \label{def:m_quasis}
    The \emph{$m$-quasisymmetric polynomials} $\qsym[m]{n}\subset \poly_n$ are those polynomials such that for any sequence $a_1,\ldots,a_k\ge 1$, the coefficients of $x_{i_1}^{a_1}\cdots x_{i_k}^{a_k}$ and $x_{j_1}^{a_1}\cdots x_{j_k}^{a_k}$ are equal whenever $1\le i_1<\cdots < i_k \le n$ and $1\le j_1<\cdots < j_k \le n$ and $i_\ell \equiv j_\ell\text{ mod }m$ for all $1\le \ell \le k$.
\end{defn}

The equivalence with the description given in \cite[\S 3.2]{HsiaoPetersen10} is straightforward. 
Beyond the convenience of having only a single alphabet, the definition also highlights a behavior with respect to translation which is difficult to see in terms of colored alphabets.

\begin{eg}
\label{eg:from_appendix_m_quasi}
In $\poly_4$, $x_3^2x_4+x_1^2x_4+x_1^2x_2$ is $2$-quasisymmetric while $        x_3^2x_4$ is $3$-quasisymmetric.
\end{eg}

\subsection{$m$-quasisymmetric divided differences}
Note that in Definition~\ref{def:m_quasis} the condition on the monomials whose coefficients must be equal can be rephrased as saying that the coefficients of $\sfx^{\mathsf{c}}$ and $\sfx^{\mathsf{c'}}$ are equal if $\mathsf{c'}$ is obtained from $\mathsf{c}$ by adding or removing consecutive strings of $m$ zeros in $\mathsf{c}$.

In what follows we write $0^m$ for a list of $m$ zeros, so that $\rope{i}^m(f)=f(x_1,\ldots,x_{i-1},0^m,x_{i},\ldots,x_{n-m})$.
For $f\in \poly$ consider the long range divided difference
 \begin{align*}\partial_i^{\underline{m}}(f)=\frac{f-f(x_1,\ldots,x_{i-1},x_{i+m},x_{i+1},\ldots,x_{i+m-1},x_i,x_{i+m+1},\ldots)}{x_i-x_{i+m}}.
 \end{align*}

\begin{defn}
\label{defn:topem}
We define the operator $\tope[m]{i}:\poly\to \poly$ by any of the equivalent expressions
\begin{align*}
\tope[m]{i}f\coloneqq \rope{i}^m\partial_i^{\underline{m}}f=\rope{i+1}^m\partial_i^{\underline{m}}f=\frac{\rope{i+1}^mf-\rope{i}^mf}{x_i}.
\end{align*}
\end{defn}

This is the $m$-quasisymmetric divided difference.
We can express $\tope[m]{i}$ in terms of $\tope{i}$ and $\rope{i}$ via the identity
$\tope[m]{i}=\tope{i}\rope{i+1}^{m-1}$.
We have the following analogue of the characterization of quasisymmetry.
\begin{thm}
    Let $f\in \poly_n$. Then $f\in \qsym[m]{n}$ if and only if $\rope{1}^mf=\cdots =\rope{n-m+1}^mf$. Consequently, 
    $\qsym[m]{n}$ is a ring.
    Furthermore, $f\in \poly_n$ is $m$-quasisymmetric if and only if $\tope[m]{1}f=\cdots =\tope[m]{n-m}f=0$.
\end{thm}
The reader may check that the polynomial
 $f=x_3^2x_4+x_1^2x_4+x_1^2x_2\in\poly_4$ from Example~\ref{eg:from_appendix_m_quasi} belongs to $\qsym[2]{4}$ by calculating $\tope[2]{1}f$ and $\tope[2]{2}$ and seeing that both equal $0$.



\begin{lem}[Twisted Leibniz rule]
\label{lem:mleibniz}
For $f,g\in \poly$ we have
$\tope[m]{i}(fg)=\tope[m]{i}(f)\rope{i+1}^m(g)+\rope{i}^m(f)\tope[m]{i}(g)$.
\end{lem}

\subsection{Indexed forests}
\label{sub:indexed_forest}

An \emph{$(m+1)$-ary rooted plane tree} $T$ is a rooted plane tree where each node has either $m+1$ children $v_0,\ldots,v_m$ or $0$ children. The notions of internal node, leaf, and trivial tree $\ast$ are natural generalizations of the $m=1$ case.


\begin{defn}\label{def:m_indexed_forest}
    An \emph{$m$-indexed forest} is an infinite sequence  $T_1,T_2,\ldots$ of $(m+1)$-ary trees where all but finitely many of the trees are $\ast$.
    We write $\indexedforests[m]$ for the set of all $m$-indexed forests.
\end{defn}

As an example, \Cref{fig:2-indexed_forest_eg} depicts an $F\in \indexedforests[2]$. We identify the leaves of $F$ with $\mathbb{N}$ as before. Write $\internal{F}=\bigcup_{i=1}^{\infty}\internal{T_i}$, and $|F|=|\internal{F}|$, and 
identify nodes with $\internal{F}\sqcup \NN$. For $v\in \internal{F}$ we always write
$   v_0,\ldots,v_m\in \internal{F}\sqcup \NN$
for the children of $v$ from left to right. 

Finally, for $F\in \indexedforests[m]$ we define its \emph{support} $\supp(F)$ to be the set of leaves in $\NN$ associated to the nontrivial trees in $F$, and for fixed $n\ge 1$ we define the class of forests
\begin{align*}
\suppfor[m]{n}=\{F\in \indexedforests[m]\suchthat \supp(F)\subset [n]\}.
\end{align*}
 The cardinality of $\suppfor[m]{n}$ is given by \emph{Raney numbers} \cite{Ra60}-- write $n=mq+r$ where $0\leq r\leq m-1$.
    Then we have
    \begin{align}
    \label{eq:appendix_raney}
    |\suppfor[m]{n}|=\frac{r+1}{n+1}\binom{n+q}{q}=\frac{r+1}{(m+1)q+r+1}\binom{(m+1)q+r+1}{q}.
    \end{align}

\begin{figure}[!ht]
    \centering
    \includegraphics[scale=0.55]{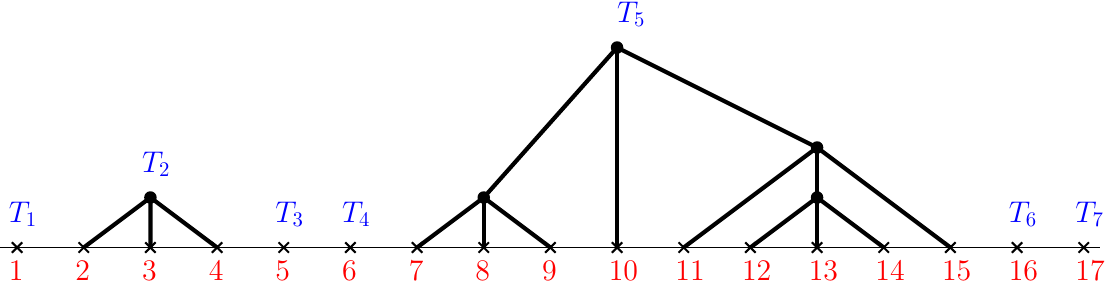}
    \caption{A $2$-indexed forest in $\suppfor[2]{15}$}
    \label{fig:2-indexed_forest_eg}
\end{figure}

Let $F\in \indexedforests[m]$ and $v\in\internal{F}$. The flag value $\rho_F(v)$ is the label of the leaf obtained by going down \emph{leftmost} edges starting from $v$. The \emph{code} $\sfc(F)$ is then defined as in Definition~\ref{defn:code_forest}.

\begin{thm}
\label{thm:mForesttoCode}
    The map $\sfc:\indexedforests[m]\to \nvect$ is a bijection.
\end{thm}

For $F\in \indexedforests[m]$, let $\qdes{F}\coloneqq \{\rho_F(v)\suchthat v \text{ a terminal node in } F\}.$ 
For $F$ in \Cref{fig:2-indexed_forest_eg} we have $\qdes{F}=\{2,7,12\}$.
One has $i\in \qdes{F}$ if and only if $ \Longleftrightarrow c_{i}>0 \text{ and } c_{i+1}=\cdots=c_{i+m}=0.$ for $\sfc(F)=(c_i)_i$. In particular,
 if $i,j\in \qdes{F}$ then $|i-j|\ge m+1$.
Define
\begin{align*}
 \ltfor[m]{n} &\coloneqq \{F\in \indexedforests[m]\suchthat \qdes{F}\subset [n]\}.
\end{align*}
Note that $\suppfor[m]{n}\subset \ltfor[m]{n}$.
For the forest $F$ in Figure~\ref{fig:2-indexed_forest_eg}, we have $F\in \ltfor[2]{n}$ for all $n\geq 12$.
Finally we define the set of \emph{zigzag forests} by
\begin{align*}
         \zigzag[m]{n} &\coloneqq\ltfor[m]{\{n-m+1,n-m+2,\ldots,n\}}
\end{align*}
An element of $\zigzag[2]{6}$ is shown in~\Cref{fig:2-zigzag_eg}. Note that it also belongs to $\zigzag[2]{7}$.

\begin{figure}[!ht]
    \centering
    \includegraphics[scale=0.85]{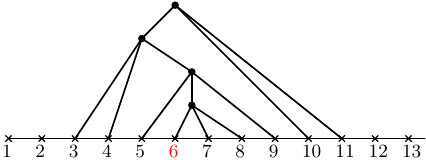}
    \caption{A forest $F\in \indexedforests[2]$ in $\zigzag[2]{6}$ and $\zigzag[2]{7}$  with $\qdes{F}=\{6\}$}
    \label{fig:2-zigzag_eg}
\end{figure}

\begin{defn}
For $F\in \indexedforests[m]$ and any $i$, the \emph{blossomed forest} $F\cdot i$ is obtained by making the $i$th leaf of $F$ into a terminal node by giving it $m+1$ leaf children.
    If $i\in \qdes{F}$, we define the \emph{trimmed forest} $F/i\in \indexedforests[m]$ by removing the terminal node $v$ with $\rho_F(v)=i$.
\end{defn}


The notion of trimming sequences $\Trim{F}$ is the same as in the $m=1$ case, and these sequences are again in bijection with standard decreasing labelings of $F$,

\subsection{$m$-indexed forests and the $m$-Thompson monoids}
\label{sub:mthompson}

We define a monoid structure on $\indexedforests[m]$ by taking for $F,G\in \indexedforests[m]$ the composition $F\cdot G\in\indexedforests[m]$ to be obtained by identifying the $i$th leaf of $F$ with the $i$th root node of $G$.

If we write $\underline{i}=\underbrace{\ast \ast\cdots \ast}_{i-1} \pitchfork \ast \ast \cdots$ where $\pitchfork$ has one internal node with $m+1$ leaf children, then 
$\Trim{F}=\{(i_1,\ldots,i_k): F=\underline{i_1} \cdots \underline{i_k}\}$. 
Every $F\in \indexedforests[m]$ has a unique expression $F=\underline{1}^{c_1}\cdot\underline{2}^{c_2}\cdots$. The exponents are given by $\sfc(F)=(c_1,c_2,\ldots)$.


\begin{defn}
\label{defn:mThMon}
    The $m$-Thompson monoid $\Th[m]$ is the quotient of the free monoid $\{1,2,\ldots\}^*$ by the relations $i\cdot j=j\cdot (i+m)$ for $i>j$.
\end{defn}

\begin{thm}
\label{thm:mthomisom} The map $\Th[m]\to \indexedforests[m]$ given by $i\mapsto \underline{i}$ is a monoid isomorphism.
\end{thm}


  From now on we will tacitly identify elements $i_1\cdots i_k\in \Th[m]$ of the Thompson monoid and the associated forest $\underline{i_1} \cdots \underline{i_k}$ in $\indexedforests[m]$, and so omit the underlines from now on.



 For $F,G\in \indexedforests[m]$, say $F\ge G$ if $F=H\cdot G$ for some $H\in \indexedforests[m]$. If $F \ge G$ then we write $F/G\in \indexedforests[m]$ to be the unique indexed forest with $F=(F/G)\cdot G$.




\subsection{Forest polynomials $\forestpoly[m]{F}$ and trimming operators $\tope[m]{F}$}
\label{sub:mforest_polynomials}

We now introduce a new family of polynomials $\forestpoly[m]{F}$ indexed by $F\in \indexedforests[m]$ which we call $m$-forest polynomials.

\begin{defn}
    For $F\in \indexedforests[m]$, define $\compatible{F}$ to be the set of all $\kappa:\internal{F}\to \NN$ such that for all $v\in \internal{F}$ with children $v_0,\ldots,v_m\in \internal{F}\sqcup \NN$ we have
    \begin{itemize}
        \item $\kappa(v)\le \rho_F(v)$
        \item If $v_i\in \internal{F}$ then
        $\kappa(v)\le \kappa(v_i)-i$
        \item $\kappa(v)\equiv \rho(v)$ mod $m$.
    \end{itemize}
    The $m$-forest polynomial $\forestpoly[m]{F}$ is the generating function for $\compatible{F}$:
    \begin{align*}\forestpoly{F}=\sum_{\kappa\in \compatible{F}}\prod_{v\in \internal{F}}x_{\kappa(v)}.
    \end{align*}
\end{defn}

\begin{prop}\label{prop:mleading_monomial}
For $F\in \indexedforests[m]$ with code $\sfc(F)=(c_1,c_2,\ldots)$, we have
$$\forestpoly[m]{F}=\sfx^{\sfc(F)}+\sum_{\sfd<\sfc(F)} a_{\sfd}\sfx^{\sfd}$$
where the revlex ordering is used. Furthermore, if $c_i=0$ for all $i>m$ then $\forestpoly[m]{F}=\sfx^{\sfc(F)}$.
\end{prop}


For $F\in \indexedforests[2]$ in Figure~\ref{fig:2-forest_poly_eg} we have
$\forestpoly{F}=x_{2} x_{3} x_{6}^{2} + x_{2} x_{3} x_{4} x_{6} + x_{2} x_{3} x_{4}^{2}$.
Note that $\sfc(F)=(0,1,1,0,0,2,0,\dots)$ and $\sfx^{\sfc(F)}=x_2x_3x_6^2$ is indeed the revlex leading term.

\begin{figure}[!ht]
    \centering
    \includegraphics[width=\textwidth]{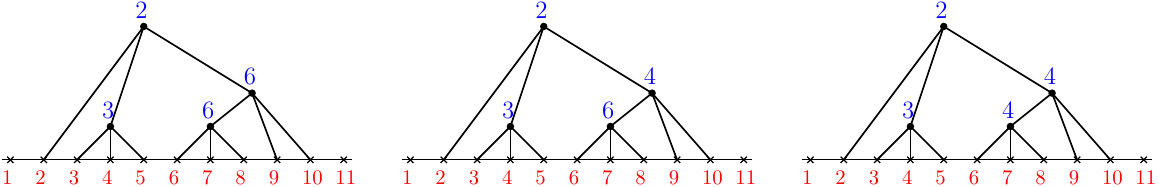}
    \caption{An $F\in \indexedforests[2]$ with the three fillings in $\compatible{F}$}
    \label{fig:2-forest_poly_eg}
\end{figure}




\begin{prop}
\label{prop:topemthompson}
    $\tope[m]{i}\tope[m]{j}=\tope[m]{j}\tope[m]{i+m}$ for $i>j$. In particular $i\mapsto \tope[m]{i}$ induces a representation of $\Th[m]$ via compositions of the $\tope[m]{i}$ operators.
\end{prop}

In particular, for $F\in \Th[m]$ we get a well-defined operator  $\tope[m]{F}\coloneqq \tope[m]{i_1}\cdots \tope[m]{i_k}$ for any expression $F=i_1\cdots i_k$.



\subsection{Characterizing $m$-Forest polynomials via trimming operators}
\label{sub:mforests_as_analogues_of_schuberts}

\begin{thm}
\label{thm:topemtrims}
    For $F\in \indexedforests[m]$ and $i\geq 1$ we have
         \begin{align}\tope[m]{i}\forestpoly[m]{F}=\begin{cases}\forestpoly[m]{F/i}&\text{if }i\in \qdes{F}\\0&\text{otherwise.}\end{cases}\end{align}\end{thm}


\begin{thm}
\label{thm:mforestunique}
    The family of $m$-forest polynomials $\{\forestpoly[m]{F}:F\in \indexedforests[m]\}$ is uniquely characterized by the properties $\forestpoly[m]{\emptyset}=1$, $\forestpoly[m]{F}$ is homogeneous, and $\tope[m]{i}\forestpoly[m]{F}=\delta_{i\in \qdes{F}}\forestpoly[m]{F/i}$.
\end{thm}

\begin{cor}
\label{cor:mTFG}
    For $F,G\in \indexedforests[m]$ we have \begin{align}\tope[m]{F}\forestpoly[m]{G}=\begin{cases}\forestpoly[m]{G/F}&\text{if }G\ge F\\0&\text{otherwise.}\end{cases}\end{align}
    In particular, $\ct \tope[m]{F}\forestpoly[m]{G}=\delta_{F,G}$.
\end{cor}

As a consequence we obtain the following.
\begin{prop}
\label{prop:mforestZbasis}
    The $m$-forest polynomials $\{\forestpoly[m]{F}:F\in \indexedforests[m]\}$ form a $\mathbb{Z}$-basis for $\poly$, and we can write any $f\in \poly$ in this basis as
    \begin{align*}f=\sum (\ct \tope[m]{F}f)\,\forestpoly[m]{F}.\end{align*}
    Additionally, $\{\forestpoly[m]{G}\suchthat F\in\ltfor[m]{n}\}$ is a $\mathbb{Z}$-basis for $\poly_n$.
\end{prop}

\subsection{Positive expansions}
\label{sub:mpositivity}



We group the chief results of Section~\ref{sec:positivity} as one itemized result.
\begin{thm}
The following positivity results hold.
\begin{enumerate}
    \item For $F\in \indexedforests[m]$ we have
    $\rope{i}^m\,\forestpoly[m]{F}$ is multiplicity-free $m$-forest positive.
    \item For $F,G\in \indexedforests[m]$ we have $\forestpoly[m]{F}\forestpoly[m]{G}$ is $m$-forest positive.
    \item ($m$-forest polynomial ``Monk's Rule'') For $F\in \indexedforests[m]$ we have $\forestpoly[m]{\underline{i}}\,\forestpoly{F}=(x_i+x_{i-m}+x_{i-2m}+\cdots+x_{i\text{ mod }m})\forestpoly[m]{F}$ is multiplicity-free $m$-forest positive. (We take $i\text{ mod }m$ to be the representative of $i$ modulo $m$ in $\{1,\dots,m\}$.
\end{enumerate}
\end{thm}

The following result which captures positivity of expansions between $m$-forest polynomials for varying $m$ is new. 
The straightforward proof is omitted.
\begin{thm}
   For any $k\geq 1$, $m$-forest polynomials are $km$-forest positive. In particular, forest polynomials are  $m$-forest positive.
\end{thm}

\subsection{Fundamental $m$-quasisymmetrics and $\zigzag[m]{n}$}
\label{sub:mfundamental}


We translate the definition of $m$-fundamental quasisymmetric polynomials~\cite{AvCh1618,BauHol08} to our single alphabet setting.
  For an integer sequence $a=(a_1,\ldots,a_k)$ with $a_i\ge 1$ we define the set of \emph{$m$-compatible sequences} \begin{align*}\compatible[m]{a}=\{(i_1,\ldots,i_k):i_j\equiv a_j\text{ mod }m,\text{ }a_j\ge i_j\ge i_{j+1},\text{ and if }a_j>a_{j+1}\text{ then }i_j>i_{j+1}\}.\end{align*}
  Then we define the \emph{$m$-slide polynomial} to be the generating function
    \begin{equation*}
        \slide[m]{a}=\sum_{\mathbf{i}\in \compatible[m]{a}}\sfx_{\mathbf{i}}.
    \end{equation*}
The notion of an $m$-compatible sequence is a straightforward generalization of compatible sequences for $m=1$. The definition of an $m$-slide polynomial is then a straightforward generalization of the notion of (ordinary) slide polynomials \cite{AS17}.
Like with forest polynomials, it is easy to check that the revlex leading monomial of $\slide[m]{a}$ is $\sfx^{\sfc}$ where $\sfc=(c_i)_{i\in \NN}\in \nvect$ is determined by $c_i=\#\{a_j=i\suchthat 1\leq j\leq k\}$.
Furthermore for large $m$ we have the equality $\slide[m]{a}=\sfx^{\sfc}$.

Just as the ordinary fundamental quasisymmetric polynomials constitute a subfamily of slide polynomials \cite[Lemma 3.8]{AS17}, so too do the $m$-fundamental quasisymmetric polynomials constitute a subfamily of the $m$-slides.


\begin{defn}  \label{def:mqseq}
Let $\qseq[m]{n}$  be the set of sequences $(a_1,\ldots,a_k)$ satisfying
    \begin{enumerate}[label=(\roman*)]
        \item \label{it:8.2i} $a_1\ge \cdots \ge a_k\ge 1$
        \item \label{it:8.2ii}$n\ge a_1\ge n-m+1$
        \item \label{it:8.2iii} $a_i-a_{i+1}\le m$ for $1\leq i\leq k-1$.
    \end{enumerate}
If $(a_1,\ldots,a_k)\in \qseq[m]{n}$ then $\slide[m]{a}\in \poly_n$ is called an \emph{$m$-fundamental quasisymmetric polynomial}.
\end{defn}

Up to the change of $m$ alphabets to a single one, this notion corresponds to the one in the literature: see~\cite[\S 3.2]{HsiaoPetersen10} for a straightforward comparison.

\begin{thm}
\label{thm:mqseqbij}
The  mapping $(a_1,\ldots,a_k)\mapsto F=a_k\cdots a_1$ is a bijection
$\qseq[m]{n}\to \zigzag[m]{n}$. Under this bijection we have $\slide[m]{a}=\forestpoly[m]{F}$.
\end{thm}


\begin{eg}
Consider the element of $\zigzag[2]{6}$ from Figure~\ref{fig:2-zigzag_eg}. The corresponding element of $\qseq[2]{6}$ is $a=(6,5,3,3)$, and
the corresponding $2$-slide polynomial equals
\begin{align*}
    \slide[2]{6533}=x_3^2x_5x_6+x_1x_3x_5x_6+x_1^2x_5x_6+x_1^2x_3x_6+x_1^2x_3x_4.
\end{align*}
Note that $\tope[2]{6}\,\slide[2]{6533}=x_3^2x_5+x_1x_3x_5+x_1^2x_5+x_1^2x_3=\slide[2]{533}$ as predicted by Theorem~\ref{thm:mqseqbij}.
\end{eg}

We may now identify a distinguished basis for $\qsym[m]{n}$, noting that the following result is sometimes taken in the literature as the definition of $\qsym[m]{n}$.

\begin{prop}
\label{prop:mqsymbasis}
    $\qsym[m]{n}$ has a $\mathbb{Z}$-basis $\{\slide[m]{a}\suchthat a\in \qseq[m]{n}\}$ of fundamental $m$-quasisymmetric polynomials.
\end{prop}

\subsection{Coinvariants}
\label{sub:coinvs}
 Recall that $\qsymide[m]{n}$ is the ideal in $\poly_n$ generated by all polynomials $f\in \qsym[m]{n}$ with $\ct f=0$. We define the \emph{$m$-quasisymmetric coinvariants} to be
\begin{align*}
\qscoinv[m]{n}\coloneqq \poly_n/\qsymide[m]{n}.
\end{align*}


\begin{cor}
\label{cor:topemdescend}
    For $F\in \suppfor[m]{n}$ we have $\tope[m]{F}(\qsymide[m]{n})\subset \qsymide[m]{n-m|F|}$, and so $\tope[m]{F}$ descends to a map 
    $\tope[m]{F}:\qscoinv[m]{n}\to \qscoinv[m]{n-m|F|}$.
    In particular, $\tope[m]{1},\ldots,\tope[m]{n-m}$ descend to maps
    \begin{align*}
    \tope[m]{1},\ldots,\tope[m]{n-m}:\qscoinv[m]{n}\to \qscoinv[m]{n-m}.
    \end{align*}
\end{cor}

With the help of a natural generalization of Theorem~\ref{thm:forestfactorization} incorporating $m$, we then obtain:
\begin{thm}
\label{thm:mqsymidebasis}
\begin{enumerate}[label=(\arabic*)]
    \item \label{A:it1} $\qsymide[m]{n}$ has a $\mathbb{Z}$-basis given by $\{\forestpoly[m]{F}:F\in\ltfor[m]{n}\setminus \suppfor[m]{n}\}$.
    \item \label{A:it2} $\qscoinv[m]{n}$ has a $\mathbb{Z}$-basis given by $\{\forestpoly[m]{F}:\suppfor[m]{n}\}$.
    In particular its dimension is given by the Raney number in ~\eqref{eq:appendix_raney}. This  recovers \cite[Theorem 5.1]{Av0507} by taking $n=pm$.
\end{enumerate}
\end{thm}


\subsection{Harmonics}
\label{sub:harmonics}

    The \emph{$m$-quasisymmetric harmonics}  are defined to be
    \begin{align*}
    \hqsym[m]{n}\coloneqq &\{f\in \mathbb{Q}[\lambda_1,\ldots,\lambda_n]\suchthat \langle g,f\rangle_D=0\text{ for all } g\in \qsymide[m]{n}\}\\=&\{f\in \mathbb{Q}[\lambda_1,\ldots,\lambda_n]\suchthat g(\deri{1},\dots,\deri{n})f=0\text{ for all } g\in \qsym[m]{n}\text{ with }\ct g=0\}.
    \end{align*}

    For $f\in \QQ[\lambda_1,\lambda_2,\ldots]$ we define
    \begin{align*}
       \vope[m]{i}\coloneqq \int_{\lambda_{i+m}}^{\lambda_i}f(\lambda_1,\ldots,\lambda_{i-1},z,\lambda_{i+m+1},\ldots) dz=(\rope{i+1}^{\vee})^{m-1}\,\vope{i}f
    \end{align*}

\begin{prop}\label{prop:madjointness}
   The operator $\tope[m]{i}$ is adjoint to  $\vope[m]{i}$. 
    Consequently, for $F\in \indexedforests[m]$ we have a well-defined operator $\vope[m]{F}$ adjoint to $\tope[m]{F}$ defined by
   $
    \vope[m]{F}=
    \vope[m]{i_k}\cdots \vope[m]{i_1} \text{ for any } (i_1,\ldots,i_k)\in \Trim{F}.
    $
\end{prop}

Let $F\in \indexedforests[m]$ and let $\lambda=(\lambda_1,\lambda_2,\ldots)$ be a sequence with $\lambda_i\ge \lambda_{i+1}$ for all $i$. We define the \emph{forest polytope} $\cube_{F,\lambda}\subset \mathbb{R}^{\internal{F}}$ as the subset of assignments $\phi:\internal{F}\to \RR$ satisfying the following constraints. 
Letting $\phi_\lambda$ be the extension of $\phi$ to $\internal{F}\sqcup \supp(F)$ by setting $\phi_\lambda(i)=\lambda_i$, we have for all $v\in \internal{F}$ the inequalities
\begin{align*}
    \phi_\lambda(v_L)\ge \phi(v)\ge \phi_\lambda(v_R).
\end{align*}

Thus the defining inequalities only involve edges in a nested forest $\Bin{F}$ that we now introduce. Define the coloring of  $F\in \indexedforests[m]$ as the map $\col: \internal{F}\sqcup \NN\to \mathbb{Z}/m\mathbb{Z}$ by $\col(v)=(\rho_F(v)\!\!\mod m)$. For $v\in \internal{F}$ with children $v_0,\ldots,v_m\in \internal{F}\sqcup \NN$ we have  $\col(v_i)\equiv \col(v)+i\text{ mod }m$. 

We define $\Bin{F}$ to be the nested binary plane forest obtained by deleting all edges connecting $v\in \internal{F}$ to one of its internal children $v_1,\ldots,v_{m-1}$, and when referring to $v\in \internal{F}$ as an internal node of $\Bin{F}$, we write $v_L\coloneqq v_0$ and $v_R\coloneqq v_m$ for the left and right children of $v$. 
The connected components of $\Bin{F}$ are monochromatic binary trees, which we can then color with the common color of their vertices, as in \Cref{fig:coloredforest}.
\begin{figure}[h]
    \centering \includegraphics[width=\textwidth]{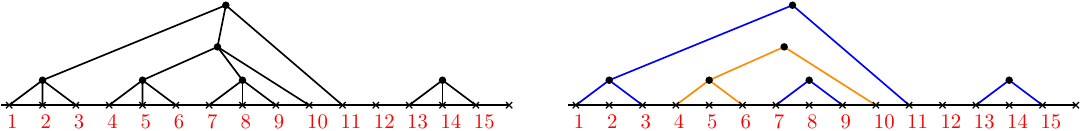}
    \caption{A forest $F\in \indexedforests[2]$ and its associated colored $\Bin{F}$}
    \label{fig:coloredforest}
\end{figure}



As in Lemma~\ref{lem:recursivevol}, given $F\in \indexedforests[m]$ and $i\in \qdes{F}$, we can consider the projection $\pi_v:\cube_{F,\lambda}\to [\lambda_{i+m},\lambda_i]$ which then satisfies 
$\pi_v^{-1}(z)=\cube_{F/i,\lambda'}$
for $\lambda'=(\lambda_1,\ldots,\lambda_{i-1},z,\lambda_{i+m+1},\ldots)$. 
It follows that \begin{align*}\vol(\cube_{F,\lambda})=\int_{\lambda_{i+m}}^{\lambda_i}\vol(\cube_{F/i,\lambda'})=\vope[m]{i}\vol(\cube_{F/i,\lambda}).
    \end{align*}

For $F\in \indexedforests[m]$, we define the volume polynomial $V_F(\uplambda)$ associated to $F$ as $\vol(\cube_{F,\lambda}).$

\begin{cor}\label{cor:mvol_via_vope}
Let $F\in \indexedforests[m]$. Then 
    $V_F(\uplambda)=\vope[m]{F}(1),$
    and for $f\in \poly$ we have $\langle f,V_F(\uplambda)\rangle_D=\ct \tope[m]{F}(f)$.
\end{cor}


\begin{thm}\label{thm:mvol_as_duals_to_mforests}
For all $F,G\in \indexedforests[m]$ we have $\langle \forestpoly[m]{G},V_F(\uplambda)\rangle_D=\delta_{F,G}$. 
Furthermore, the family of projected volume polynomials $\{P_nV_F(\uplambda)\}_{F\in\ltfor[m]{n}}$ in $\QQ[\lambda_1,\ldots,\lambda_n]$ is the graded $D$-dual basis to the homogeneous basis $\{\forestpoly[m]{F}\}_{F\in\ltfor[m]{n}}$ of $\QQ[x_1,\ldots,x_n]$.
\end{thm}

We are ready to determine a basis for $\hqsym[m]{n}$ in terms of volume polynomials.

\begin{thm}\label{thm:basis_mharmonics}
A $\QQ$-basis for $\hqsym[m]{n}$ is given by $
    \{V_F(\uplambda)\suchthat F\in \suppfor[m]{n}\}.
    $
Furthermore, $\hqsym[m]{n}$ is spanned by the derivatives of the homogeneous degree $\lfloor (n-1)/m\rfloor$ elements of $\hqsym[m]{n}$.
\end{thm}
The second half of the preceding result rests on the following generalization of~\Cref{prop:annihilator}.
\begin{prop}\label{prop:mannihilator}
Let $f\in \poly_n$ be homogeneous of degree $d<\lfloor (n-1)/m\rfloor$, and assume that $x_1f\in \qsymide[m]{n}$. Then we have $f\in \qsymide[m]{n}$.
\end{prop}

Let $\paths{\Bin{F}}$ denote the set of functions $\mathcal{P}:\internal{F}\to \{L,R\}$. 
Like before, we can encode $\mathcal{P}\in \paths{\widehat{F}}$ as a collection of vertex disjoint paths traveling up from the leaves of $\widehat{F}$ which cover every node in $\internal{F}$. 
For each $\mathcal{P}$, we let $d(\mathcal{P})\coloneqq (d_i)_{i\in \NN}\in \nvect$ where $d_i$ records the length of the path that has one endpoint at leaf $i$. 
It is easy to see that $d$ is injective, and for $\mathcal{P}$ the constant $L$-function we have $d(\mathcal{P})=\sfc(F)$.

%

Given $\sfc\in \nvect$ we define $\epsilon_F(\sfc)$ to equal $(-1)^{|\mathcal{P}^{-1}(R)|}$ if there exists $\mathcal{P}\in \paths{\Bin{F}}$ such that $d(\mathcal{P})=\sfc$, and $0$ otherwise.
With this notation in hand we have

\begin{prop}\label{prop:explicit_mvolume}
    For $F\in \indexedforests[m]$ we have $V_F(\uplambda)=\sum_{\sfc\in \nvect}\epsilon_F(\sfc) \,\frac{\uplambda^{\sfc}}{c!}.$
    Dually, for  $\sfc\in \nvect$ we have $\sfx^{\sfc}=\sum_{G\in \indexedforests[m]} \epsilon_G(c)\,\forestpoly[m]{G}.$
\end{prop}

\begin{thm}
For $F\in \indexedforests[m]$ we have
    $V_F(\uplambda)\in \QQ[\lambda_1-\lambda_{m+1},\lambda_2-\lambda_{m+2},\ldots]$. 
    The coefficients in 
    \begin{align*}V_F(\uplambda)=\sum_{\sfc =(c_1,c_2,\dots) \in \nvect} b_\sfc\,\prod_{i\ge 1}(\lambda_i-\lambda_{i+m})^{c_i}\end{align*}
    satisfy $b_\sfc=\frac{1}{\sfc!}\tope[m]{F}\left(\prod_{i\ge 1}(\forestpoly[m]{i})^{c_i}\right)\ge 0$.
\end{thm}

\subsection{A quasisymmetric nil-Hecke algebra}

For the same reason as in Section~\ref{sec:quasisymmetric_nil_hecke}, we have to consider $\Hom_{\qsym[m]{n}}(\poly_n,\poly_{n-k})$ where $\poly_{n-k}$ is considered as a $\qsym[m]{n}$-module by the map $\rope{1}^k|_{\qsym[m]{n}}:\qsym[m]{n}\to \qsym[m]{n-k}$, which is well-defined directly from the definition of $\qsym[m]{n}$. We note that if $k$ is a multiple of $m$ then $\rope{1}^k|_{\qsym[m]{n}}=\rope{n-k+1}^k|_{\qsym[m]{n}}$ by \Cref{thm:Rqsymchar} which is the map setting $x_n=\cdots = x_{n-k+1}=0$.
\begin{thm}
We have
    \begin{align*}\Hom_{\qsym[m]{n}}(\poly_n,\poly_{n-k})=\bigoplus_{H\in \suppfor[m]{n}\text{ with }m|H|\le k} \poly_{n-k}\rope{1}^{k-m|H|}\,\tope[m]{H}.\end{align*}
\end{thm}

Like before, the limiting object 
\begin{align*}\bigoplus_k\lim_{n\to \infty}\Hom_{\qsym[m]{n}}(\poly_{n},\poly_{n-k})=\bigoplus_{F\in \indexedforests[m]}\poly \rope{1}^a\,\tope[m]{F}\end{align*}
is the subalgebra of $\End(\poly)$ generated by all $x_i$, $\rope{1}$, and $\tope[m]{i}$.

\begin{thm}
\label{thm:mquasisymnilhecke}
    The algebra in $\End(\poly)$ generated by $\rope{1}$, $\tope[m]{i}$ and $x_i$ has relations generated by
    \begin{enumerate}[label=(\roman*)]
        \item (Comm.) $\tope[m]{i}\rope{1}=\rope{1}\tope[m]{i+1}$, for $i\ge 1$, $\rope{1}x_i=x_{i-1}\rope{1}$ for $i>1$, $x_ix_j=x_jx_i$ for all $i,j$\\$\tope[m]{i}x_j=x_{j}\tope[m]{i}$ if $j<i$ and $\tope[m]{i}x_j=x_{j-m}\tope[m]{i}$ if $j>i+m$\\
        $\tope[m]{i}\tope[m]{j}=\tope[m]{j}\tope[m]{i+m}$ for $i>j$,
        \item  $\rope{1}x_1=0$ and $\tope[m]{i}x_{i+j}=0$ for $1\le j \le m-1$
        \item $\tope[m]{i}x_i=\rope{1}^m+x_1\tope[m]{1}+\cdots + x_i\tope[m]{i}$ and $\tope[m]{i}x_{i+m}=-(\rope{1}^m+x_1\tope[m]{1}+\cdots + x_{i-1}\tope[m]{i-1})$
    \end{enumerate}
\end{thm}

\section{Proof of \Cref{thm:topemtrims}}
\label{sec:ProofThatTrimsWork}

We give a combinatorial proof of Theorem~\ref{thm:topemtrims}. 
An algebraic proof may be found in \cite{NST_2}.

 For any $k$, let
 $C_k(F)=\{\kappa\in \compatible{F}\suchthat k,k+1,\ldots,k+(m-1)\not\in \im{\kappa}\},$ and define $\Phi_k:\NN\setminus \{k,\ldots,k+(m-1)\}\to \NN$
 and its inverse $\Phi_k^{-1}:\NN\to \{k,\ldots,k+(m-1)\}$ by \begin{align*}\Phi_k(a)\coloneqq \begin{cases}
 a&\text{if }a\le k-1\\a-m&\text{if }a\ge k+m\end{cases}\text{ and }\Phi_{k}^{-1}(a)=\begin{cases}a&\text{if }a\le k-1\\a+m&\text{if }a\ge k.\end{cases}\end{align*}
 Then we have \begin{equation}
     \rope{k}^{m}(\oneforestpoly{F})=\rope{k}^m(\sum_{\kappa\in \compatible{F}}\prod_{v\in \internal{F}}x_{\kappa(v)})=\sum_{\kappa\in C_{k}(F)}\prod_{v\in \internal{F}}x_{\Phi_k\kappa(v)}.
 \end{equation}

 Consider the map $f:\NN\setminus \{i,\ldots,i+(m-1)\}\to \NN\setminus \{i+1,\ldots,i+m\}$  and its inverse $f^{-1}:\NN\setminus \{i+1,\ldots,i+m\}\to \NN\setminus \{i,\ldots,i+(m-1)\}$ by \begin{align*}f(a)=\begin{cases}a-m&\text{if }a=i+m\\ a&\text{otherwise}\end{cases}\text{ and }f^{-1}(a)=\begin{cases}a+m&\text{if }a=i\\a&\text{otherwise.}\end{cases}\end{align*}
 We will use the following fact often to show that various compatible labellings retain the compatibility inequalities between internal children after being modified by one of the above functions.
 \begin{clm}
 \label{clm:4functions}
     For $g$ being any of the functions $\Phi_{k},\Phi_{k}^{-1}$, $f$, $f^{-1}$, the following holds: if $a,b$ are in the domain of $g$ and $0\le j \le m$ is such that $b\le a-j$ and $b\equiv a-j$ mod $m$, then
     \begin{align*}g(a)-g(b)\ge j.\end{align*}
 \end{clm}
 \begin{proof}
In all cases $g$ is the unique increasing bijection from $\NN\setminus A$ to $\NN\setminus B$ for some finite sets $A,B$. It also satisfies $g(x)\equiv x$ mod $m$ for all $x\notin A$.
Thus $b\le a-j\leq a$ implies $g(a)\ge g(b)$, while $b\equiv a-j$ mod $m$ implies $g(a)-g(b)\equiv j$ mod $m$. From there the conclusion follows immediately in all cases but one: if $j=m$, then we must forbid $g(a)=g(b)$, and indeed this cannot hold since $b\leq a-m<a$ and $g$ is a bijection.
 \end{proof}
 We claim that for any $\kappa\in C_i(F)$ we have $f\kappa\in C_{i+1}(F)$. 
 Since $i+1,\ldots,i+m\not\in \im{f(\kappa)}$ it remains to check that $f\in \compatible{F}$. Let $v\in \internal{F}$. Then $f(\kappa(v))\le \kappa(v)\le \rho_F(v)$, $f(\kappa(v))\equiv \kappa(v)\equiv \col(v)$, and by \Cref{clm:4functions} for $v,v_j\in \internal{F}$ we have \begin{align*} f(\kappa(v_j))-f(\kappa(v))\ge j.\end{align*}

 Additionally, the map $f^*:\kappa \mapsto f\kappa$ is injective as $f$ is injective, and $\Phi_i\kappa=\Phi_{i+1}f\kappa$. It follows that
 \begin{equation}
 \label{eq:TiforestFcombi}
     \tope[m]{i}\oneforestpoly{F}=\frac{\rope{i+1}^m\oneforestpoly{F}-\rope{i}^m\oneforestpoly{F}}{x_i}=\sum_{\kappa'\in C_{i+1}(F)\setminus \im{f^*}}\frac{1}{x_i}\prod_{v\in \internal{F}}x_{\Phi_{i+1}\kappa'(v)}.
 \end{equation}

 \begin{clm}
 Let $\kappa'\in C_{i+1}(F)$. Then $\kappa'\not\in C_{i+1}(F)\setminus \im{f^*}$ if and only if $i\in \qdes{F}$ and the terminal node $u$ with $\rho_F(u)=i$ has $\kappa'(u)=i$.
 \end{clm}
 \begin{proof}
 We have $\kappa'\not \in C_{i+1}(F)\setminus \im{f^*}$ if and only if $f^{-1}\kappa'\not \in C_i(F)$. Note that if $i\in \qdes{F}$ and the terminal node $u$ with $\rho_F(u)=i$ has $\kappa'(u)=i$ then $f^{-1}\kappa'(u)=i+m>\rho_F(u)$ and so $f^{-1}\kappa'(u)\not \in \compatible{F}\supset C_i(F)$. Therefore it suffices to show that $f^{-1}\kappa'\not \in C_i(F)$ implies there is a terminal node $u$ with $\rho_F(u)=i$ and $\kappa'(u)=i$.

 We first show that $f^{-1}\kappa'\not \in C_i(F)$ implies there is some $v\in \internal{F}$ with $\kappa'(v)=i$ and $\rho_F(v)<i+m$. We do this by checking that all other conditions besides $f^{-1}\kappa'(v)\le \rho_F(v)$ for $f^{-1}\kappa'(v)$ to lie in $C_i(F)$ are satisfied. Note that $i,\ldots,i+m\not\in \im{f^{-1}\kappa'}$, for $v\in \internal{F}$ we have $f^{-1}(\kappa(v))\equiv \kappa(v)\equiv \col(v)$, and by \Cref{clm:4functions} we have for $v,v_j\in \internal{F}$ that
 \begin{align*}f^{-1}(\kappa'(v_j))-f^{-1}(\kappa'(v))\ge j.\end{align*}

 Therefore as all other conditions for $f^{-1}\kappa'\in C_i(F)$ are met, we have $f^{-1}\kappa'\not\in C_i(F)$ exactly if there is $v\in \internal{F}$ with $f^{-1}\kappa'(v)>\rho_F(v)$. 
 Because $f^{-1}\kappa'(v)=\kappa'(v)\le \rho_F(v)$ if $\kappa'(v)\ne i$, the inequality $f^{-1}\kappa'(v)>\rho_F(v)$ happens precisely if $\kappa'(v)=i$ and $f^{-1}\kappa'(v)=i+m>\rho_F(v)$.

 Now from this $v$ with $\kappa'(v)=i$ and $\rho_F(v)<i+m$, we construct the desired $u$. We have $i=\kappa'(v)\le \kappa'(v_0)\le \kappa'(v_{0^2})\le \cdots \le \kappa'(v_{0^k})\le  \rho_F(v)<i+m$ where $v_{0^k}\in \internal{F}$ is the last internal left descendant of $v$. Because $\rho_F(v)\equiv \kappa'(v)\equiv i$ mod $m$ we must have $\rho_F(v)=i$ and so additionally $\kappa'(v_{0^k})=i$. Therefore $u=v_{0^k}$ has $\kappa'(u)=\rho_F(u)=u_0=i$.

 We claim that $u$ is terminal. If not, let $1\le j\le m$ be the first index with $u_j\in \internal{F}$. Then \begin{align*}i+j=\kappa'(u)+j\le \kappa'(u_j)\le \rho_F(u_j)=i+j,\end{align*} so $\kappa'(u_j)=i+j$, contradicting that $\kappa'\in C_{i+1}(F)$.

 Therefore $u$ is terminal with $\kappa'(u)=\rho_F(u)=i$ and in particular $i\in \qdes{F}$.
 \end{proof}
 Returning to the proof of \Cref{thm:topemtrims}, we may now conclude that if $i\not\in \qdes{F}$ then $C_{i+1}(F)\setminus \im{f^*}=\emptyset$ and so by~\Cref{eq:TiforestFcombi} we have$\tope[m]{i}\oneforestpoly{F}=0$. On the other hand, suppose $i\in \qdes{F}$ and let $u$ be the terminal node with $\rho_F(u)=i$. Then by the claim we know that \begin{align*}C_{i+1}(F)\setminus \im{f^*}=\{\kappa'\in C_{i+1}(F):\kappa'(u)=i\}.\end{align*}
 \begin{clm}
     If $i\in \qdes{F}$ then there is a bijection $\{\kappa'\in C_{i+1}(F):\kappa'(u)=i\}\to \compatible{F/i}$
     given by $\kappa' \mapsto \kappa''=\Phi_{i+1}\kappa'|_{\internal{F}\setminus u}$ (identifying $\internal{F}\setminus u=\internal{F/i}$).
 \end{clm}
 \begin{proof}
Before starting we show that for $v\in \internal{F}\setminus u$ we have $\Phi_{i+1}\rho_F(v)=\rho_{F/i}(v)$ (which explains the presence of $\Phi_{i+1}$ in the statement). This is because by definition of the monoid structure on $\indexedforests[m]$, for any $F\ge G$ we have $\rho_{F/G}(v)=\rho_G(w)$ for $w$ the root of the $\rho_{F}(v)$'th tree of $G\in \indexedforests[m]$. 
Taking $G=\underline{i}$ we directly see that $\rho_G(w)=\Phi_{i+1}(\rho_F(v))$.

 First we check that the map from the claim is well-defined. Let $v\in \internal{F/i}=\internal{F}\setminus u$. Then $\Phi_{i+1}\kappa'(v)\le \Phi_{i+1}\rho_F(v)=\rho_{F/i}(v)$, $\Phi_{i+1}\kappa'(v)\equiv \kappa'(v)\equiv \col(v)\equiv \overline{\rho}_{F/i}(v)$ mod $m$, and by \Cref{clm:4functions} we have for $v,v_j\in \internal{F/i}$ that \begin{align*}\Phi_{i+1}\kappa'(v_j)-\Phi_{i+1}\kappa'(v)\ge j.\end{align*}

 This map is clearly injective, so it remains to check surjectivity. Given $\kappa''\in \compatible{F/i}$ we claim that $\kappa'\in \compatible{F}$ where \begin{align*}\kappa'(v)=\begin{cases}\Phi_{i+1}^{-1}\kappa''(v)&\text{if }v\ne u\\i&\text{if }v=u.\end{cases}\end{align*}
 If this is the case then it is readily apparent that $\Phi_{i+1}\kappa'|_{\internal{F}\setminus u}=\kappa''$ so surjectivity will follow.

 Let $v\in \internal{F}$. If $v=u$ then we have $\kappa'(v)=i=\rho_F(v)$ which shows $\kappa'(v)\le \rho_F(v)$ and $\kappa'(v)\equiv \col(v)$ mod $m$. If $v\ne u$ then $\kappa'(v)=\Phi_{i+1}^{-1}\kappa''(v)\le \Phi_{i+1}^{-1}\rho_{F/i}(v)=\rho_F(v)$ and $\kappa'(v)=\Phi_{i+1}^{-1}\kappa''(v)\equiv \kappa'(v)\equiv \rho_{F/i}(v)\equiv \rho_F(v)$. Finally, if $v,v_j\in \internal{F}$ it remains to show that \begin{align*}\kappa'(v_j)-\kappa'(v)\ge j.\end{align*} If $v_j=u$ then $\kappa'(v_j)-\kappa'(v)=i-\Phi_{i+1}^{-1}\kappa''(v)\ge i-\Phi_{i+1}^{-1}\rho_{F/i}(v)=i-\rho_F(v)\ge i-(\rho_F(v_j)-j)=j$ (where the last inequality is because $\rho_F(v)\le \rho_F(v_j)$ and $\col(v)\equiv \col(v_j)-j$ mod $m$). If $v_j\ne u$ then if follows from \Cref{clm:4functions} that $\Phi_{i+1}^{-1}\kappa''(v_j)-\Phi_{i+1}^{-1}\kappa''(v)\ge j$.
 \end{proof}
 Given this claim, we now conclude
 \begin{align*}
 \oneforestpoly{F/i}=\sum_{\kappa''\in \compatible{F/i}}\prod_{v\in \internal{F/i}}x_{\kappa''(v)}=&\sum_{\kappa'\in C_{i+1}(F)\setminus \im{f^*}}\prod_{v\in \internal{F/i}}x_{\Phi_{i+1}\kappa'(v)}\\=&\sum_{\kappa'\in C_{i+1}(F)\setminus \im{f^*}}\frac{1}{x_i}\prod_{v\in \internal{F}}x_{\Phi_{i+1}\kappa'(v)}=\tope[m]{i}\oneforestpoly{F}.
 \end{align*}
 where in the second last equality we used that $\Phi_{i+1}\kappa'(u)=\Phi_{i+1}(i)=i$.

\pagebreak

\section{Table of $\oneforestpoly{F}$ for all $F\in \suppfor{5}$}
\label{sec:table}
\begin{center}
\scriptsize
\begin{longtable}{|c|l|} 
\hline
$c(F)$ & $\oneforestpoly{F}$  \\
\hline
$ (0, 0, 0, 0, 0) $ & $ 1 $\\ 
$ (1, 0, 0, 0, 0) $ & $ x_{1} $\\ 
$ (0, 1, 0, 0, 0) $ & $ x_{1} + x_{2} $\\ 
$ (0, 0, 1, 0, 0) $ & $ x_{1} + x_{2} + x_{3} $\\ 
$ (0, 0, 0, 1, 0) $ & $ x_{1} + x_{2} + x_{3} + x_{4} $\\ 
$ (2, 0, 0, 0, 0) $ & $ x_{1}^{2} $\\ 
$ (1, 1, 0, 0, 0) $ & $ x_{1} x_{2} $\\ 
$ (1, 0, 1, 0, 0) $ & $ x_{1}^{2} + x_{1} x_{2} + x_{1} x_{3} $\\ 
$ (1, 0, 0, 1, 0) $ & $ x_{1}^{2} + x_{1} x_{2} + x_{1} x_{3} + x_{1} x_{4} $\\ 
$ (0, 2, 0, 0, 0) $ & $ x_{1}^{2} + x_{1} x_{2} + x_{2}^{2} $\\ 
$ (0, 1, 1, 0, 0) $ & $ x_{1} x_{2} + x_{1} x_{3} + x_{2} x_{3} $\\ 
$ (0, 1, 0, 1, 0) $ & $ x_{1}^{2} + 2 x_{1} x_{2} + x_{2}^{2} + x_{1} x_{3} + x_{2} x_{3} + x_{1} x_{4} + x_{2} x_{4} $\\ 
$ (0, 0, 2, 0, 0) $ & $ x_{1}^{2} + x_{1} x_{2} + x_{2}^{2} + x_{1} x_{3} + x_{2} x_{3} + x_{3}^{2} $\\ 
$ (0, 0, 1, 1, 0) $ & $ x_{1} x_{2} + x_{1} x_{3} + x_{2} x_{3} + x_{1} x_{4} + x_{2} x_{4} + x_{3} x_{4} $\\ 
$ (3, 0, 0, 0, 0) $ & $ x_{1}^{3} $\\ 
$ (2, 1, 0, 0, 0) $ & $ x_{1}^{2} x_{2} $\\ 
$ (2, 0, 1, 0, 0) $ & $ x_{1}^{2} x_{2} + x_{1}^{2} x_{3} $\\ 
$ (2, 0, 0, 1, 0) $ & $ x_{1}^{3} + x_{1}^{2} x_{2} + x_{1}^{2} x_{3} + x_{1}^{2} x_{4} $\\ 
$ (1, 2, 0, 0, 0) $ & $ x_{1} x_{2}^{2} $\\ 
$ (1, 1, 1, 0, 0) $ & $ x_{1} x_{2} x_{3} $\\ 
$ (1, 1, 0, 1, 0) $ & $ x_{1}^{2} x_{2} + x_{1} x_{2}^{2} + x_{1} x_{2} x_{3} + x_{1} x_{2} x_{4} $\\ 
$ (1, 0, 2, 0, 0) $ & $ x_{1}^{3} + x_{1}^{2} x_{2} + x_{1} x_{2}^{2} + x_{1}^{2} x_{3} + x_{1} x_{2} x_{3} + x_{1} x_{3}^{2} $\\ 
$ (1, 0, 1, 1, 0) $ & $ x_{1}^{2} x_{2} + x_{1}^{2} x_{3} + x_{1} x_{2} x_{3} + x_{1}^{2} x_{4} + x_{1} x_{2} x_{4} + x_{1} x_{3} x_{4} $\\ 
$ (0, 3, 0, 0, 0) $ & $ x_{1}^{3} + x_{1}^{2} x_{2} + x_{1} x_{2}^{2} + x_{2}^{3} $\\ 
$ (0, 2, 1, 0, 0) $ & $ x_{1}^{2} x_{2} + x_{1}^{2} x_{3} + x_{1} x_{2} x_{3} + x_{2}^{2} x_{3} $\\ 
$ (0, 2, 0, 1, 0) $ & $ x_{1}^{2} x_{2} + x_{1} x_{2}^{2} + x_{1}^{2} x_{3} + x_{1} x_{2} x_{3} + x_{2}^{2} x_{3} + x_{1}^{2} x_{4} + x_{1} x_{2} x_{4} + x_{2}^{2} x_{4} $\\ 
$ (0, 1, 2, 0, 0) $ & $ x_{1} x_{2}^{2} + x_{1} x_{2} x_{3} + x_{1} x_{3}^{2} + x_{2} x_{3}^{2} $\\ 
$ (0, 1, 1, 1, 0) $ & $ x_{1} x_{2} x_{3} + x_{1} x_{2} x_{4} + x_{1} x_{3} x_{4} + x_{2} x_{3} x_{4} $\\ 
$ (4, 0, 0, 0, 0) $ & $ x_{1}^{4} $\\ 
$ (3, 1, 0, 0, 0) $ & $ x_{1}^{3} x_{2} $\\ 
$ (3, 0, 1, 0, 0) $ & $ x_{1}^{3} x_{2} + x_{1}^{3} x_{3} $\\ 
$ (3, 0, 0, 1, 0) $ & $ x_{1}^{3} x_{2} + x_{1}^{3} x_{3} + x_{1}^{3} x_{4} $\\ 
$ (2, 2, 0, 0, 0) $ & $ x_{1}^{2} x_{2}^{2} $\\ 
$ (2, 1, 1, 0, 0) $ & $ x_{1}^{2} x_{2} x_{3} $\\ 
$ (2, 1, 0, 1, 0) $ & $ x_{1}^{2} x_{2}^{2} + x_{1}^{2} x_{2} x_{3} + x_{1}^{2} x_{2} x_{4} $\\ 
$ (2, 0, 2, 0, 0) $ & $ x_{1}^{2} x_{2}^{2} + x_{1}^{2} x_{2} x_{3} + x_{1}^{2} x_{3}^{2} $\\ 
$ (2, 0, 1, 1, 0) $ & $ x_{1}^{2} x_{2} x_{3} + x_{1}^{2} x_{2} x_{4} + x_{1}^{2} x_{3} x_{4} $\\ 
$ (1, 3, 0, 0, 0) $ & $ x_{1} x_{2}^{3} $\\ 
$ (1, 2, 1, 0, 0) $ & $ x_{1} x_{2}^{2} x_{3} $\\ 
$ (1, 2, 0, 1, 0) $ & $ x_{1} x_{2}^{2} x_{3} + x_{1} x_{2}^{2} x_{4} $\\ 
$ (1, 1, 2, 0, 0) $ & $ x_{1} x_{2} x_{3}^{2} $\\ 
$ (1, 1, 1, 1, 0) $ & $ x_{1} x_{2} x_{3} x_{4} $\\ 
\hline
\end{longtable}
\end{center}

\bibliographystyle{hplain}
\bibliography{main}

\end{document}